\documentclass[reqno]{amsart}
\usepackage[utf8]{inputenc}
\usepackage{amscd,amsfonts,amsmath,amssymb,amsthm,bbm,dsfont,centernot,mathtools}
\usepackage{color,mathrsfs,stackrel}
\usepackage[margin=1in]{geometry}
\usepackage[colorlinks=true,linkcolor=blue,citecolor=red]{hyperref}

\newtheorem{corollary}{Corollary}[section]
\newtheorem{lemma}[corollary]{Lemma}

\newtheorem{theorem}[corollary]{Theorem}

\theoremstyle{definition}
\newtheorem{definition}[corollary]{Definition}
\newtheorem{remark}[corollary]{Remark}

\newtheorem*{acknowledgements}{\sc Acknowledgements}

\numberwithin{equation}{section}

\allowdisplaybreaks

\newcommand{\abs}[1]{\lvert #1 \rvert}
\newcommand{\norm}[1]{\lVert #1 \rVert}
\newcommand{\spr}[2]{\langle #1, #2\rangle}

\def\XXint#1#2#3{{\setbox0=\hbox{$#1{#2#3}{\int}$ }
\vcenter{\hbox{$#2#3$ }}\kern-.6\wd0}}

\def\mint{\Xint{\rotatebox[origin][30]{$-$}}}

\def \argmin {\mathop {\rm argmin}\nolimits}
\def \div {\mathop {\rm div}\nolimits}

\def \dist {\mathop {\rm dist}\nolimits}

\def \de {\mathrm{d}}
\def \e {\varepsilon}
\def \R {\mathbb R}

\begin{document}

\title[Existence of solutions to a phase--field model of dynamic fracture]{Existence of solutions to a phase--field model of dynamic fracture with a crack--dependent dissipation}
\author{Maicol Caponi}
\address{SISSA, Via Bonomea 265, 34136 Trieste, Italy}
\email{mcaponi@sissa.it, maicol.caponi@tu-dresden.de}

% \email{maicol.caponi@tu-dresden.de}
% \address{Technische Universit\"at Dresden, Fakult\"at Mathematik, Zellescher Weg 12-14, 01069 Dresden}
% \email{maicol.caponi@tu-dresden.de}
\thanks{Preprint SISSA 06/2018/MATE}

\begin{abstract}
We propose a phase--field model of dynamic fracture based on the Ambrosio--Tortorelli's approximation, which takes into account dissipative effects due to the speed of the crack tips. By adapting the time discretization scheme contained in~\cite{LOS}, we show the existence of a dynamic crack evolution satisfying an energy--dissipation balance, according to Griffith's criterion. Finally, we analyze the dynamic phase--field model of~\cite{BLR,Lar} with no dissipative terms.
\end{abstract}

\maketitle



%------------------------------------------
% Introduction
%------------------------------------------

\section{Introduction}

In this paper, we present a phase--field model of dynamic brittle fracture based on a suitable adaptation of Griffith's dynamic criterion~\cite{Mott}, and different from the one proposed in~\cite{BLR,Lar,LOS}. Following these papers, we rely on the Ambrosio--Tortorelli's functional~\cite{AT}, which provides a good approximation of the corresponding stationary problem. 

In the quasi--static setting, namely when the external data vary slowly compared to the elastic wave speed of the material, Griffith's criterion~\cite{Grif} states that during the crack growth there is an exact balance between the decrease in stored elastic energy and the energy used to increase the crack. This principle is turned into a precise definition for sharp--interface models by Francfort and Marigo in~\cite{FM}, where, in the antiplane case, the following energy functional is considered:
\begin{equation}\label{eq:ph_Grif}
\frac{1}{2}\int_{\Omega\setminus \Gamma}\abs{\nabla u}^2\de x+\mathcal H^{d-1}(\Gamma).
\end{equation} 
Here, $\Omega\subset\R^d$ is an open bounded set, which represents the reference configuration of the elastic material, the closed set $\Gamma\subset\overline\Omega$ describes the crack, and $u\in H^1(\Omega\setminus \Gamma)$ is the antiplane displacement. The first term in~\eqref{eq:ph_Grif} is the stored elastic energy, while the second one, called surface energy, models the energy used to produce a crack. In this setting, for a given time--dependent Dirichlet datum $t\mapsto w(t)$, a quasi--static evolution is a time--dependent pair $t\mapsto(u(t),\Gamma_t)$ which satisfies the minimality condition
\begin{equation}\label{eq:ph_min}
\frac{1}{2}\int_{\Omega\setminus \Gamma_t}\abs{\nabla u(t)}^2\de x+\mathcal H^{d-1}(\Gamma_t)\le \frac{1}{2}\int_{\Omega\setminus \Gamma^*}\abs{\nabla u^*}^2\de x+\mathcal H^{d-1}(\Gamma^*)
\end{equation}
among all pairs $(u^*,\Gamma^*)$, where $\Gamma^*$ is a closed set with $\Gamma^*\supseteq \Gamma_t$, and $u^*\in H^1(\Omega\setminus \Gamma^*)$ with $u^*=w(t)$ on $\partial\Omega\setminus \Gamma^*$.  The minimum problem~\eqref{eq:ph_min} is coupled with the irreversibility condition $\Gamma_s\subseteq \Gamma_t$ for every $s\le t$ (meaning the crack can only increase in time), and with the Griffith's energy balance for every $t$
\begin{equation*}
\frac{1}{2}\int_{\Omega\setminus \Gamma_t}\abs{\nabla u(t)}^2\de x+\mathcal H^{d-1}(\Gamma_t)= \frac{1}{2}\int_{\Omega\setminus \Gamma_0}\abs{\nabla u(0)}^2\de x+\mathcal H^{d-1}(\Gamma_0)+\text{work of external data}.
\end{equation*}

The study of~\eqref{eq:ph_Grif} dates back to Mumford and Shah in~\cite{Mum-Shah} initially in the context of image segmentation; we refer to~\cite{BFM} and the reference therein for applications to fracture. The challenges in the analysis of this functional lead Ambrosio and Tortorelli to introduce in~\cite{AT} a regularized version of~\eqref{eq:ph_Grif}: the set $\Gamma$ is replaced by a function $v\in[0,1]$ which takes a value close to~0 in a small neighborhood of $\Gamma$, and a value close to~1 far from it. More precisely, for every $\e>0$ they consider
\begin{equation*}
\mathcal E_\e(u,v):=\frac{1}{2}\int_\Omega (v^2+\eta_\e)\abs{\nabla u}^2\de x,\quad\mathcal H_\e(v):=\frac{1}{4\e}\int_\Omega\abs{1-v}^2\de x+\e\int_\Omega\abs{\nabla v}^2\de x,
\end{equation*}
for $u,v\in H^1(\Omega)$, with $0<\eta_\e\ll\e$. A minimum point $(u_\e,v_\e)$ of $\mathcal E_\e+\mathcal H_\e$ provides a good approximation of a minimizer $(u,\Gamma)$ of~\eqref{eq:ph_Grif} as $\e\to 0^+$, in the sense that $u_\e$ is close to $u$, $v_\e$ is close to 0 near $\Gamma$, and $\mathcal E_\e(u_\e,v_\e)+\mathcal H_\e(v_\e)$ approximates the energy~\eqref{eq:ph_Grif}. The minimality condition~\eqref{eq:ph_min} is replaced by
\begin{equation}\label{eq:ph_min2}
\mathcal E_\e(u_\e(t),v_\e(t))+\mathcal H_\e(v_\e(t))\le \mathcal E_\e(u^*,v^*)+\mathcal H_\e(v^*)
\end{equation}
among all pairs $(u^*,v^*)$ with $v^*\le v_\e(t)$ and $u^*=w(t)$ on $\partial\Omega$. Notice that the inequality $v^*\le v_\e(t)$ reflects the inclusion $\Gamma^*\supseteq \Gamma_t$. As before, the minimum problem~\eqref{eq:ph_min2} is complemented by the irreversibility condition $0\le v_\e(t)\le v_\e(s)\le 1$ for every $s\le t$, and by the Griffith's energy balance for every time; we refer to~\cite{Gi} for the convergence of this evolution, as $\e\to 0^+$, toward a sharp--interface one. 

In particular, a quasi--static phase--field evolution $t\mapsto(u_\e(t),v_\e(t))$ satisfies:
\begin{itemize}
\item[$(Q_1)$] for every $t\in[0,T]$ the function $u_\e(t)$ solves $\div([(v_\e(t))^2+\eta_\e]\nabla u_\e(t))=0$ in $\Omega$ with suitable boundary conditions;
\item[$(Q_2)$] the map $t\mapsto v_\e(t)$ is non increasing ($v_\e(t)\le v_\e(s)$ for $0\le s\le t\le T$) and for every $t\in[0,T]$ the function $0\le v_\e(t)\le 1$ solves 
$$\mathcal E_\e(u_\e(t),v_\e(t))+\mathcal H_\e(v_\e(t))\le \mathcal E_\e(u_\e(t),v^*)+\mathcal H_\e(v^*)\quad\text{for every $v^*\le v_\e(t)$};$$
\item[$(Q_3)$] for every $t\in[0,T]$ the pair $(u_\e(t), v_\e(t))$ satisfies the Griffith's energy balance
\begin{equation*}
\mathcal E_\e(u_\e(t),v_\e(t))+\mathcal H_\e(v_\e(t))=\mathcal E_\e(u_\e(0),v_\e(0))+\mathcal H_\e(v_\e(0))+\text{work of external data}.
\end{equation*}
\end{itemize}

In the dynamic case, the first condition is replaced by the wave equation, while in the energy balance we need to take into account the kinetic energy term. Developing these principles, in~\cite{BLR,Lar} the authors propose the following phase--field model of dynamic crack propagation:
\begin{itemize}
\item[$(D_1)$] $u_\e$ solves $\ddot u_\e-\div([v_\e^2+\eta_\e]\nabla u_\e)=0$ in $(0,T)\times\Omega$ with suitable boundary and initial conditions;
\item[$(D_2)$] the map $t\mapsto v_\e(t)$ is non increasing and for every $t\in[0,T]$ the function $0\le v_\e(t)\le 1$ solves
$$\mathcal E_\e(u_\e(t),v_\e(t))+\mathcal H_\e(v_\e(t))\le \mathcal E_\e(u_\e(t),v^*)+\mathcal H_\e(v^*)\quad\text{for every $v^*\le v_\e(t)$};$$
\item[$(D_3)$] for every $t\in[0,T]$ the pair $(u_\e(t), v_\e(t))$ satisfies the Griffith's dynamic energy balance
\begin{align*}
&\frac{1}{2}\int_\Omega|\dot u_\e(t)|^2\de x+\mathcal E_\e(u_\e(t),v_\e(t))+\mathcal H_\e(v_\e(t))\\
&=\frac{1}{2}\int_\Omega|\dot u_\e(0)|^2\de x+\mathcal E_\e(u_\e(0),v_\e(0))+\mathcal H_\e(v_\e(0))+\text{work of external data}.
\end{align*}
\end{itemize}
A solution to this model can be approximated by mean of a time discretization with an alternate scheme: to pass from the previous time to the next one, one first solves the wave equation for $u$ keeping $v$ fixed, and then the minimum problem for $v$ keeping $u$ fixed. This method is used in~\cite{LOS} to prove the existence of a pair $(u,v)$ satisfying $(D_1)$--$(D_3)$ in the more general linear elastic case, that is when the displacement $u$ is vector--valued and $\abs{\nabla u}^2$ is replaced by $\mathbb C Eu\cdot Eu$, where $\mathbb C$ is the elastic tensor and $Eu:=\frac{1}{2}(\nabla u+\nabla u^T)$ is the symmetrized gradient. For technical reasons, a viscoelastic dissipative term is added to $(D_1)$, which means they consider
\begin{equation}\label{eq:ph_diss}
\ddot u_\e-\div([v_\e^2+\eta_\e]\mathbb C(Eu_\e+E\dot u_\e))=0\quad\text{in $(0,T)\times\Omega$}.
\end{equation}

So far there are no existence results for $(D_1)$--$(D_3)$ unless we consider the formulation~\eqref{eq:ph_diss} in $(D_1)$. The disadvantage of the viscoelastic dissipative term appears when we consider the behavior of the solution as $\e\to 0^+$, a problem that is out of the scope of this paper. If we were able to prove the convergence of the solution toward a dynamic sharp--interface evolution, then the energy--dissipation balance for the damped wave equation in cracked domains~\cite{DM-Lar,T1} would imply that the limit crack does not depend on time. 

To avert this problem, we propose here a different model which avoids viscoelastic terms on the displacement and consider dissipative effects due to the speed of the crack tips. More precisely, given a natural number $k\in\mathbb N\cup\{0\}$, we consider a dynamic phase--field evolution $t\mapsto(u_\e(t),v_\e(t))$ satisfying:
\begin{itemize}
\item[$(\tilde D_1)$] $u_\e$ solves $\ddot u_\e-\div([(v_\e^+)^2+\eta_\e]\mathbb C Eu_\e)=0$ in $(0,T)\times\Omega$ with suitable boundary and initial conditions;
\item[$(\tilde D_2)$] the map $t\mapsto v_\e(t)$ is non increasing and for a.e. $t\in(0,T)$ the function $v_\e(t)\le 1$ solves the variational inequality
$$\mathcal E_\e(u_\e(t),v^*)-\mathcal E_\e(u_\e(t),v_\e(t))+\mathcal H_\e(v^*)-\mathcal H_\e(v_\e(t))+(\dot v_\e(t),v^*-v_\e(t))_{H^k(\Omega)}\ge 0\quad\text{for every }v^*\le v_\e(t);$$
\item[$(\tilde D_3)$] for every $t\in[0,T]$ the pair $(u_\e(t), v_\e(t))$ satisfies the Griffith's dynamic energy--dissipation balance
\begin{equation}\label{eq:ph_e1}
\begin{aligned}
&\frac{1}{2}\int_\Omega|\dot u_\e(t)|^2\de x+\mathcal E_\e(u_\e(t),v_\e(t))+\mathcal H_\e(v_\e(t))+\int_0^t\norm{\dot v_\e(s)}_{H^k(\Omega)}^2\de s\\
&=\frac{1}{2}\int_\Omega|\dot u_\e(0)|^2\de x+\mathcal E_\e(u_\e(0),v_\e(0))+\mathcal H_\e(v_\e(0))+\text{work of external data},
\end{aligned}
\end{equation}
\end{itemize}
where in this case $\mathcal E_\e(u,v)\coloneqq\frac{1}{2}\int_\Omega [(v^+)^2+\eta_\e]\mathbb CEu\cdot Eu\de x$ for $u\in H^1(\Omega;\R^d)$ and $v\in H^1(\Omega)$, with the convention $H^0(\Omega):=L^2(\Omega)$ for $k=0$. Notice that, in order to obtain the Griffith's energy balance, we need to consider the dissipative term $\int_0^t\norm{\dot v_\e}_{H^k(\Omega)}^2\de s$. This one guarantees more regularity in time for the phase--field function, more precisely that $v_\e\in H^1(0,T;H^k(\Omega))$, and, as explained in Remark~\ref{rem:ph_diss}, is related to a dissipation depending on the crack tips velocities. 

In the quasi--static setting, a condition similar to $(\tilde D_2)$ can be found in~\cite{Negri,ABN}, where it defines a unilateral gradient flow evolution for the phase--field function $v_\e$. In sharp--interface models, this crack--dependent term arises in the study of the so--called vanishing viscosity evolutions, which are linked to the analysis of local minimizers of Griffith's functional~\eqref{eq:ph_Grif}, see for example~\cite{R,Laz-Toa2}. We point out that a similar dissipation also appears in~\cite{LN} for a one--dimensional dynamic debonding model.

By adapting the time discretization scheme of~\cite{LOS}, we show the existence of a dynamic phase--field evolution $(u_\e,v_\e)$ which satisfies $(\tilde D_1)$--$(\tilde D_3)$, provided that $k>d/2$, where $d$ is the dimension of the ambient space. This condition is crucial to obtain the validity of the Griffith's dynamic energy--dissipation balance~\eqref{eq:ph_e1}, since in our case the viscoelastic dissipation used in~\cite{LOS} is not present.

We conclude this paper by analyzing the dynamic phase--field model with no viscous terms $(D_1)$--$(D_3)$ in the linear elastic case. We show the existence of an evolution $t\mapsto (u_\e(t),v_\e(t))$ which satisfies $(D_1)$ and $(D_2)$, but only an energy inequality (see~\eqref{eq:ph_enin2}), instead of $(D_3)$.

The paper is organized as follows: in Section~\ref{sec:ph2} we describe our model and in Theorem~\ref{thm:ph:main_res} we state our main existence result. Section~\ref{sec:ph3} is devoted to the study of the time discretization scheme. We construct an approximation of our evolution by solving, with an alternate minimization procedure, the problems $(\tilde D_1)$ and $(\tilde D_2)$. Next, we show that this discrete evolution satisfies the estimate~\eqref{eq:ph_est}, which allows us to pass to the limit as the time step tends to zero. For every $k\in\mathbb N\cup\{0\}$ we obtain the existence of a dynamic evolution $t\mapsto (u_\e(t),v_\e(t))$ which satisfies $(\tilde D_1)$ and $(\tilde D_2)$, and the energy--dissipation inequality~\eqref{eq:ph_enin}. We complete the proof of Theorem~\ref{thm:ph:main_res} in Section~\ref{sec:ph4}, where we prove that for $k>d/2$ our evolution is more regular in time, and it satisfies the Griffith's dynamic energy--dissipation balance~\eqref{eq:ph_e1}. Finally, in Section~\ref{sec:ph5} we study the dynamic phase--field model without dissipative terms $(D_1)$--$(D_3)$.

%---------------------------------
% Notation and preliminary result
%---------------------------------

\section{Notation and formulation of the model}\label{sec:ph2}

The space of $m\times d$ matrices with real entries is denoted by $\R^{m\times d}$; in case $m=d$, the subspace of symmetric matrices is denoted by $\R^{d\times d}_{sym}$. We denote by $A^T$ the transpose of $A\in\R^{d\times d}$, and by $A^{sym}$ its symmetric part, namely $A^{sym}:=\frac{1}{2}(A+A^T)$. Given two vectors $a_1,a_2\in \R^d$, their scalar product is denoted by $a_1\cdot a_2$; the same notation is also used to denote the scalar product between two matrices in $\mathbb R^{m\times d}$. 

The partial derivatives with respect to the variable $x_i$ are denoted by $\partial_i$. Given a function $f\colon\R^d\to\R^m$, we denote its Jacobian matrix by $\nabla f$, whose components are $(\nabla f)_{ij}:=\partial_j f_i$ for $i=1,\dots,m$ and $j=1,\dots,d$. When $f\colon\R^d\to \R$, we use $\Delta f$ to denote its the Laplacian, which is defined as $\Delta f:=\sum_{i=1}^d\partial^2_{ii}f$. We set $\nabla^2 f:=\nabla(\nabla f)$ and $\Delta^2 f:=\Delta(\Delta f)$, and inductively we define $\nabla^kf$ and $\Delta^k f$ for every $k\in\mathbb N\cup\{0\}$, with the convention $\nabla^0f=\Delta^0f:=f$. For a tensor field $F\colon \R^d\to\R^{m\times d}$, by $\div F$ we mean its divergence with respect to lines, namely $(\div F)_i:=\sum_{j=1}^d\partial_jF_{ij}$ for $i=1,\dots,m$. 

We adopt standard notation for Lebesgue and Sobolev spaces on open subsets $\Omega$ of $\mathbb R^d$. According to the context, for every $m\in\mathbb N$ we use $(\cdot,\cdot)_{L^2(\Omega)}$ to denote the scalar product in $L^2(\Omega;\R^m)$, and $\norm{\cdot}_{L^p(\Omega)}$ to denote the norm in $L^p(\Omega;\R^m)$ for $1\le p\le\infty$. A similar convention is also used to denote the scalar products and the norms in Sobolev spaces. The boundary values of a Sobolev function are always intended in the sense of traces; the $(d-1)$--dimensional Hausdorff measure is denoted by $\mathcal H^{d-1}$. Given a bounded open set~$\Omega$ with Lipschitz boundary, we denote by $\nu$ the outer unit normal vector to $\partial\Omega$, which is defined $\mathcal H^{d-1}$--a.e. on the boundary.

The norm of a generic Banach space $X$ is denoted by $\Vert\cdot\Vert_X$; when $X$ is an Hilbert space, we use $(\cdot,\cdot)_X$ to denote its scalar product. We denote by $X'$ the dual of $X$, and by $\spr{\cdot}{\cdot}_{X'}$ the duality product between $X'$ and $X$. Given two Banach spaces $X_1$ and $X_2$, the space of linear and continuous maps from $X_1$ to $X_2$ is denoted by $\mathscr L(X_1;X_2)$; given $\mathbb A\in\mathscr L(X_1;X_2)$ and $u\in X_1$, we write $\mathbb A u\in X_2$ to denote the image of $u$ under $\mathbb A$. 

Given an open interval $(a,b)\subseteq\R$, $L^p(a,b;X)$ is the space of $L^p$--functions from $(a,b)$ to $X$; we use $W^{k,p}(a,b;X)$ and $H^k(a,b;X)$ (for $p=2$) to denote the Sobolev space of functions from $(a,b)$ to $X$ with $k$ derivatives. Given $u\in W^{1,p}(a,b;X)$, we denote by $\dot u\in L^p(a,b;X)$ its derivative in the sense distributions. The set of continuous functions from $[a,b]$ to $X$ is denoted by $C^0([a,b];X)$; we also use $C_w^0([a,b];X)$ to denote the set of weakly continuous functions from $[a,b]$ to $X$, namely the collection of maps $u\colon [a,b]\to X$ such that $t\mapsto \spr{x'}{u(t)}_{X'}$ is continuous from $[a,b]$ to $\R$ for every $x'\in X'$. When dealing with an element $u\in H^1(a,b;X)$ we always assume $u$ to be the {\it continuous} representative of its class. In particular, it makes sense to consider the pointwise value $u(t)$ for every $t\in[a,b]$. 

Let $T$ be a positive number and let $\Omega\subset\mathbb R^d$ be a bounded open set with Lipschitz boundary. We fix two (possibly empty) Borel subsets $\partial_{D_1}\Omega$, $\partial_{D_2}\Omega$  of $\partial\Omega$, and we denote by $\partial_{N_1}\Omega$, $\partial_{N_2}\Omega$ their complements. We introduce the spaces
\begin{equation*}
H^1_{D_1}(\Omega;\mathbb R^d):=\{u\in H^1(\Omega;\mathbb R^d): u=0\text{ on }\partial_{D_1}\Omega\},\quad H^1_{D_2}(\Omega):=\{v\in H^1(\Omega): v=0\text{ on }\partial_{D_2}\Omega\},
\end{equation*}
and we denote by $H^{-1}_{D_1}(\Omega;\mathbb R^d)$ the dual space of $H^1_{D_1}(\Omega;\mathbb R^d)$. The transpose of the natural embedding $H^1_{D_1}(\Omega;\mathbb R^d)\hookrightarrow L^2(\Omega;\mathbb R^d)$ induces the embedding of $L^2(\Omega;\mathbb R^d)$ into $H^{-1}_{D_1}(\Omega;\mathbb R^d)$, which is defined by $$\spr{g}{\phi}_{H^{-1}_{D_1}(\Omega)}:=(g,\phi)_{L^2(\Omega)}\quad\text{for $g\in L^2(\Omega;\mathbb R^d)$ and $\phi\in H^1_{D_1}(\Omega;\mathbb R^d)$}.$$ 

Let $\mathbb C\colon \Omega\to\mathscr L(\mathbb R^{d\times d}_{sym};\mathbb R^{d\times d}_{sym})$ be a fourth--order tensor field satisfying the following natural assumptions in linear elasticity:
\begin{align}
&\mathbb C\in L^\infty(\Omega;\mathscr L(\mathbb R^{d\times d}_{sym};\mathbb R^{d\times d}_{sym})),\label{eq:ph_A1}\\
&(\mathbb C(x)\xi_1)\cdot \xi_2= \xi_1\cdot(\mathbb C(x) \xi_2)\quad\text{for a.e. }x\in\Omega\text{ and for every } \xi_1, \xi_2\in\mathbb R^{d\times d}_{sym},\\
&\mathbb C(x) \xi\cdot \xi\ge \lambda_0|\xi|^2\quad\text{for a.e. }x\in\Omega\text{ and for every } \xi\in\mathbb R^{d\times d}_{sym}\label{eq:ph_A3},
\end{align}
for a constant $\lambda_0>0$. Thanks to second Korn's inequality (see, e.g.,~\cite{OSY}) there exists a constant $C_K>0$, depending on $\Omega$, such that
\begin{equation*}
\norm{u}_{H^1(\Omega)}\le C_K(\norm{u}_{L^2(\Omega)}+\norm{Eu}_{L^2(\Omega)})\quad\text{for every }u\in H^1(\Omega;\mathbb R^d),
\end{equation*}
where $Eu$ is the symmetrized gradient of $u$, namely $Eu:=\frac{1}{2}(\nabla u+\nabla u^T)$. By combining  Korn's inequality with~\eqref{eq:ph_A3}, we obtain that $\mathbb C$ satisfies the following ellipticity condition of integral type: 
\begin{equation}\label{eq:ph_ell}
(\mathbb CEu,Eu)_{L^2(\Omega)}\ge c_0\norm{ u}^2_{H^1(\Omega)}-c_1\norm{u}_{L^2(\Omega)}^2\quad\text{for every }u\in H^1(\Omega;\mathbb R^d),
\end{equation}
for two positive constants $c_0$ and $c_1$.

We fix $\e>0$ and we consider a map $b\colon \R\to[0,+\infty)$ satisfying
\begin{align}
&\text{$b\in C^1(\R)$ is convex and non decreasing},\label{eq:ph_breg1} \\
&\text{$b(s)\ge\eta$ for every $s\in\R$ and some $\eta>0$}\label{eq:ph_breg2}.
\end{align}
We define the elastic energy $\mathcal E\colon H^1(\Omega;\mathbb R^d)\times H^1(\Omega)\to[0,\infty]$, the surface energy $\mathcal H\colon H^1(\Omega)\to [0,\infty)$, and the kinetic energy $\mathcal K\colon L^2(\Omega;\R^d)\to [0,\infty)$ in the following way:
\begin{align*}
&\mathcal E(u,v):=\frac{1}{2}\int_\Omega b(v(x))\mathbb C(x) Eu(x)\cdot Eu(x)\de x,\\
&H(v):=\frac{1}{4\e}\int_\Omega\abs{1-v(x)}^2\de x+\e\int_\Omega\abs{\nabla v(x)}^2\de x,\\
&\mathcal K(w):=\frac{1}{2}\int_\Omega\abs{w(x)}^2\de x
\end{align*}
for $u\in H^1(\Omega;\R^d)$, $v\in H^1(\Omega)$, and $w\in L^2(\Omega;\R^d)$. For every $k\in\mathbb N\cup\{0\}$ we also define the dissipative energy $\mathcal G_k\colon H^k(\Omega)\to [0,\infty)$ as
\begin{equation*}
\mathcal G_k(\sigma):=\sum_{i=0}^k\alpha_i\int_\Omega |\nabla^i\sigma(x)|^2\de x
\end{equation*}
for $\sigma\in H^k(\Omega)$, where $\alpha_i$, $i=0,\dots,k$, are non negative numbers with $\alpha_0,\alpha_k>0$ (recall that $H^0(\Omega):=L^2(\Omega)$ for $k=0$). By~\cite[Corollary 4.16]{A}, the functional $\mathcal G_k$ induces a norm on $H^k(\Omega)$ which is equivalent to the standard one. In particular, there exist two constants $\beta_0,\beta_1>0$ such that
\begin{equation*}
    \beta_0\norm{\sigma}_{H^k(\Omega)}^2\le \mathcal G_k(\sigma)\le \beta_1\norm{\sigma}_{H^k(\Omega)}^2\quad\text{for every }\sigma\in H^k(\Omega).
\end{equation*}
Finally, we define the total energy $\mathcal F\colon H^1(\Omega;\mathbb R^d)\times L^2(\Omega;\mathbb R^d)\times H^1(\Omega)\to[0,\infty]$ as
\begin{equation*}
\mathcal F(u,w,v):=\mathcal K(w)+\mathcal E(u,v)+\mathcal H(v)\quad \text{for $u\in H^1(\Omega;\R^d)$, $w\in L^2(\Omega;\R^d)$, and $v\in H^1(\Omega)$}.
\end{equation*}

Throughout the paper we always assume that $\mathbb C$ satisfies~\eqref{eq:ph_A1}--\eqref{eq:ph_A3}, $b$ satisfies~\eqref{eq:ph_breg1} and~\eqref{eq:ph_breg2}, and that $\e$ is a fixed positive number. Given
\begin{align}
&w_1\in H^2(0,T;L^2(\Omega;\mathbb R^d))\cap H^1(0,T;H^1(\Omega;\mathbb R^d)),\quad w_2\in H^1(\Omega)\cap H^k(\Omega)\text{ with $w_2\le 1$ on $\partial_{D_2}\Omega$},\label{eq:ph_bd1}\\
&f\in L^2(0,T;L^2(\Omega;\mathbb R^d)),\quad g\in H^1(0,T;H^{-1}_{D_1}(\Omega;\mathbb R^d)),\label{eq:ph_bd2}\\
&u^0-w_1(0)\in H^1_{D_1}(\Omega;\mathbb R^d),\quad u^1\in L^2(\Omega;\mathbb R^d),\quad v^0-w_2\in H^1_{D_2}(\Omega)\cap H^k(\Omega)\text{ with }v^0\le 1\text{ in }\Omega,\label{eq:ph_id}
\end{align}
we search a pair $(u,v)$ which solves the {\it elastodynamics system}
\begin{equation}
\ddot u(t)-\div[b(v(t))\mathbb C Eu(t)]=f(t)+g(t)\quad\text{in } \Omega,\quad t\in[0,T],\label{eq:ph_elsys}\\
\end{equation}
with boundary conditions formally written as
\begin{align}
&u(t)=w_1(t)\quad\text{on }\partial_{D_1}\Omega,\quad t\in[0,T],\label{eq:ph_bc1}\\
&v(t)=w_2\quad\text{on }\partial_{D_2}\Omega,\quad t\in[0,T],\label{eq:ph_bc2}\\
&(b(v(t))\mathbb C Eu(t))\nu= 0\quad\text{on }\partial_{N_1}\Omega,\quad t\in[0,T],\label{eq:ph_bc3}
\end{align}
and initial conditions 
\begin{equation}\label{eq:ph_ic}
u(0)=u^0,\quad \dot u(0)=u^1,\quad v(0)=v^0\quad\text{in $\Omega$}.
\end{equation} 
In addition, we require the {\it irreversibility condition}:
\begin{equation}\label{eq:ph_irr}
v(t)\le v(s)\quad\text{in $\Omega$}\quad\text{for $0\le s\le t\le T$},
\end{equation}
and for a.e. $t\in(0,T)$ the following {\it crack stability condition}:
\begin{equation}\label{eq:ph_mincon}
\mathcal E(u(t),v^*)-\mathcal E(u(t),v(t))+\mathcal H(v^*)-\mathcal H(v(t))+\sum_{i=0}^k\alpha_i(\nabla^i\dot v(t),\nabla^i v^*-\nabla^i v(t))_{L^2(\Omega)}\ge 0
\end{equation}
among all $v^*-w_2\in H^1_{D_2}(\Omega)\cap H^k(\Omega)$ with $v^*\le v(t)$. Notice that the space $H^1(\Omega)\cap H^k(\Omega)$ coincides with either $H^1(\Omega)$ (when $k=0$) or $H^k(\Omega)$ (for $k\ge 1$). Finally, for every $t\in[0,T]$ we ask the Griffith's dynamic energy--dissipation balance:
\begin{equation}\label{eq:ph_Genb}
    \mathcal F(u(t),\dot u(t), v(t))+\int_0^t\mathcal G_k(\dot v(s))\de s=\mathcal F(u^0,u^1,v^0)+\mathcal W_{tot}(u,v;0,t),
\end{equation}
where $\mathcal W_{tot}(u,v;t_1,t_2)$ is the {\it total work} on $(u,v)$ over the time interval $[t_1,t_2]\subseteq[0,T]$, defined as
\begin{align*}
\mathcal W_{tot}(u,v;t_1,t_2)&:=\int_{t_1}^{t_2}\left[(f(s),\dot u(s)-\dot w_1(s))_{L^2(\Omega)}+(b(v(s))\mathbb CEu(s),E\dot w_1(s))_{L^2(\Omega)}\right]\de s\\
&\quad-\int_{t_1}^{t_2}\left[(\dot u(s),\ddot w_1(s))_{L^2(\Omega)}+\spr{\dot g(s)}{u(s)-w_1(s)}_{H^{-1}_{D_1}(\Omega)}\right]\de s+(\dot u(t_2),\dot w_1(t_2))_{L^2(\Omega)}\\
&\quad+\spr{g(t_2)}{u(t_2)-w_1(t_2)}_{H^{-1}_{D_1}(\Omega)}-(\dot u(t_1),\dot w_1(t_1))_{L^2(\Omega)}-\spr{g(t_1)}{u(t_1)-w_1(t_1)}_{H^{-1}_{D_1}(\Omega)}.
\end{align*}

\begin{remark}
A simple prototype for the function $b$ is given by
\begin{equation*}
b(s)\coloneqq (\max\{s,0\})^2+\eta\quad\text{for $s\in\R$}.
\end{equation*}
In this case, the elastic energy becomes
\begin{equation}\label{eq:ph_elastic}
\mathcal E(u,v)= \frac{1}{2}\int_\Omega [(\max\{v(x),0\})^2+\eta]\mathbb C(x) Eu(x)\cdot Eu(x)\de x
\end{equation}
for $u\in H^1(\Omega;\R^d)$ and $v\in H^1(\Omega)$, which corresponds to the dynamic phase--field model $(\tilde D_1)$--$(\tilde D_3)$ considered in the introduction. Usually, in the phase--field setting, the elastic energy is defined as
\begin{equation*}
\frac{1}{2}\int_\Omega [(v(x))^2+\eta]\mathbb C(x) Eu(x)\cdot Eu(x)\de x
\end{equation*}
for $u\in H^1(\Omega;\R^d)$ and $v\in H^1(\Omega)$, with $v$ satisfying $0\le v\le 1$. In our case, due to the presence of the dissipative term introduced in $(\tilde D_2)$ and $(\tilde D_3)$, we need to consider phase--field functions $v$ which may assume negative values. Therefore, we have to slightly modify the elastic energy functional by considering~\eqref{eq:ph_elastic}.
\end{remark}

\begin{remark}\label{rem:ph_diss}
We give an idea of the meaning of the term $\mathcal G_k(\dot v)$ in the phase--field setting, by comparing it with a dissipation, in the sharp--interface case, which depends on the velocity of the crack tips. We consider just an example in the particular case $d=2$ and $k=0$ of a rectilinear crack $\Gamma_t:=\{(\sigma,0):\sigma\le s(t)\}$ moving along the $x_1$--axis, with $s\in C^1([0,T])$, $s(0)=0$, and $\dot s(t)\ge 0$ for every $t\in[0,T]$. In view of the analysis done in~\cite{AT}, the sequence $v_\e(t)$ which best approximate $\Gamma_t$ takes the following form:
\begin{equation*}
v_\e(t,x):=\Psi\left(\frac{\dist(x,\Gamma_t)}{\e}\right)\quad \text{for }(t,x)\in[0,T]\times\mathbb R^2.
\end{equation*}
Here, $\Psi\colon\mathbb R\to[0,1]$ is a $C^1$ function satisfying $\Psi(s)=0$ for $\abs{s}\le \delta$, with $0<\delta<1$, and $\Psi(s)=1$ for $\abs{s}\ge 1$. The function $v_\e\in C^1([0,T]\times\mathbb R^2)$ is constantly 0 in a $\e\delta$--neighborhood of $\Gamma_t$, and takes the value $1$ outside a $\e$--neighborhood of $\Gamma_t$. Moreover, its time derivative satisfies
\begin{equation*}
\dot v_\e(t,x)=-\frac{\dot s(t)}{\e}\partial_1\Phi\left(\frac{x-(s(t),0)}{\e}\right)\quad\text{for }(t,x)\in[0,T]\times\mathbb R^2,
\end{equation*}
where $\Phi(y):=\Psi(\dist(y,\Gamma_0))$ for $y\in\R^2$. In particular for every $t\in[0,T]$ we deduce
\begin{align*}
\norm{\dot v_\e(t)}_{L^2(\Omega)}^2=\frac{\dot s(t)^2}{\e^2}\int_{\mathbb R^2}\left|\partial_1\Phi\left(\frac{x-(s(t),0)}{\e}\right)\right|^2\de x=\dot s(t)^2\int_{\mathbb R^2}\abs{\partial_1\Phi(y)}^2\de y=C_\Phi\dot s(t)^2.
\end{align*}
Therefore, this term can be used to detect the dissipative effects due to the velocity of the moving crack. With similar computations, if there are $m$ crack-tips, with different velocities $\dot s_i(t)$, $i=1,\dots,m$, then the term $\norm{\dot v_\e(t)}_{L^2(\Omega)}^2$ corresponds to a dissipation of the form $\sum_{i=1}^m C_i\dot s^2_i(t)$, with $C_i$ positive constants.
\end{remark}

To precise the notion of solution to the problem~\eqref{eq:ph_elsys}--\eqref{eq:ph_Genb}, we consider a pair of functions $(u,v)$ satisfying the following regularity assumptions: 
\begin{align}
&u\in C^0([0,T];H^1(\Omega;\mathbb R^d))\cap C^1([0,T];L^2(\Omega;\mathbb R^d))\cap H^2(0,T;H^{-1}_{D_1}(\Omega;\mathbb R^d)),\label{eq:ph_reg1}\\
&u(t)-w_1(t)\in H^1_{D_1}(\Omega;\mathbb R^d)\text{ for every }t\in[0,T],\\
&v\in C^0([0,T];H^1(\Omega))\cap H^1(0,T;H^k(\Omega)),\label{eq:ph_reg3}\\
&v(t)-w_2\in H^1_{D_2}(\Omega)\text{ and }v(t)\le 1\text{ in }\Omega\text{ for every }t\in[0,T].\label{eq:ph_reg4}
\end{align} 

\begin{definition}
Let $w_1$, $w_2$, $f$, and $g$ be as in~\eqref{eq:ph_bd1} and~\eqref{eq:ph_bd2}. We say that $(u,v)$ is a {\it weak solution} to the elastodynamics system~\eqref{eq:ph_elsys} with boundary conditions~\eqref{eq:ph_bc1}--\eqref{eq:ph_bc3}, if $(u,v)$ satisfies~\eqref{eq:ph_reg1}--\eqref{eq:ph_reg4}, and for a.e. $t\in (0,T)$ we have
\begin{equation}\label{eq:ph_weak_form}
\spr{\ddot u(t)}{\psi}_{H^{-1}_{D_1}(\Omega)}+(b(v(t))\mathbb C Eu(t),E\psi)_{L^2(\Omega)}=(f(t),\psi)_{L^2(\Omega)}+\spr{g(t)}{\psi}_{H^{-1}_{D_1}(\Omega)}
\end{equation}
for every $\psi\in H^1_{D_1}(\Omega;\mathbb R^d)$.
\end{definition}

\begin{remark}
Since $v(t)\le 1$ for every $t\in[0,T]$ and $b$ satisfies~\eqref{eq:ph_breg1} and~\eqref{eq:ph_breg2}, the function $b(v(t))$ belongs to $L^\infty(\Omega)$ for every $t\in[0,T]$. Hence, the equation~\eqref{eq:ph_weak_form} makes sense for every $\psi\in H^1_{D_1}(\Omega;\mathbb R^d)$. Moreover, if $(u,v)$ satisfies~\eqref{eq:ph_reg1}--\eqref{eq:ph_reg4}, then the function $(t_1,t_2)\mapsto\mathcal W_{tot}(u,v;t_1,t_2)$ is well defined and continuous, thanks to the previous assumptions on $\mathbb C$, $b$, $w_1$, $f$, and $g$.
\end{remark}

We state now our main result, whose proof will be given at the end of Section~\ref{sec:ph4}.

\begin{theorem}\label{thm:ph:main_res}
Let $k>d/2$ and let $w_1$, $w_2$, $f$, $g$, $u^0$, $u^1$, and $v^0$ be as in~\eqref{eq:ph_bd1}--\eqref{eq:ph_id}. Then there exists a weak solution $(u,v)$ to the problem~\eqref{eq:ph_elsys}--\eqref{eq:ph_bc3} with initial conditions~\eqref{eq:ph_ic} satisfying the irreversibility condition~\eqref{eq:ph_irr}, the crack stability condition~\eqref{eq:ph_mincon}, and the Griffith's dynamic energy--dissipation balance~\eqref{eq:ph_Genb}.
\end{theorem}

\begin{remark}
According to Griffith's dynamic criterion (see~\cite{Mott}), we expect the sum of kinetic and elastic energy to be dissipated during the evolution, while it is balanced when we take into account the surface energy associated to the phase--field function $v$. This happens in our case if we also consider $\int_0^t \mathcal G_k(\dot v)\de s$. The presence of this term takes into account the rate at which the function $v$ is decreasing and it is a consequence of the crack stability condition~\eqref{eq:ph_mincon}. 

We need $k>d/2$ in order to obtain the energy equality~\eqref{eq:ph_Genb}. Indeed, in this case the embedding $H^k(\Omega)\hookrightarrow C^0(\overline\Omega)$ is continuous and compact (see, e.g.,~\cite[Theorem 6.2]{A}), which implies that $\dot v(t)\in C^0(\overline\Omega)$ for a.e. $t\in(0,T)$. This regularity is crucial, since we obtain~\eqref{eq:ph_Genb} throughout another energy balance (see~\eqref{eq:ph_j1}), which is well defined only when $\dot v(t)\in L^\infty(\Omega)$. 
\end{remark}

\begin{remark}
In Theorem~\ref{thm:ph:main_res} we consider only the case of zero Neumann boundary data. Anyway, the previous result can be easily adapted to Neumann boundary conditions of the form
\begin{equation}\label{eq:ph_bc4}
(b(v(t))\mathbb C Eu(t))\nu=F(t)\quad\text{on }\partial_{N_1}\Omega,\quad t\in[0,T],
\end{equation}
provided that $F\in H^1(0,T;L^2(\partial_{N_1}\Omega;\mathbb R^d))$. In this case a {\it weak solution} to the problem~\eqref{eq:ph_elsys} with Dirichlet boundary conditions~\eqref{eq:ph_bc1} and~\eqref{eq:ph_bc2}, and Neumann boundary condition~\eqref{eq:ph_bc4} is a pair $(u,v)$ satisfying~\eqref{eq:ph_reg1}--\eqref{eq:ph_reg4} and for a.e. $t\in(0,T)$ the equation
\begin{equation*}
\spr{\ddot u(t)}{\psi}_{H^{-1}_{D_1}(\Omega)}+(b(v(t))\mathbb C Eu(t),E\psi)_{L^2(\Omega)}=(f(t),\psi)_{L^2(\Omega)}+\spr{\tilde g(t)}{\psi}_{H^{-1}_{D_1}(\Omega)}
\end{equation*} 
for every $\psi\in H^1_{D_1}(\Omega;\mathbb R^d)$, where the term $\tilde g(t)\in H^{-1}_{D_1}(\Omega;\mathbb R^d)$ is defined for $t\in[0,T]$ as
\begin{equation*}
\spr{\tilde g(t)}{\psi}_{H^{-1}_{D_1}(\Omega)}:=\spr{g(t)}{\psi}_{H^{-1}_{D_1}(\Omega)}+\int_{\partial_{N_1}\Omega}F(t,x)\cdot\psi(x)\de\mathcal H^{d-1}(x)\quad\text{for }\psi\in H^1_{D_1}(\Omega;\mathbb R^d).
\end{equation*}
Since $\tilde g\in H^1(0,T;H^{-1}_{D_1}(\Omega;\mathbb R^d))$, we can apply Theorem~\ref{thm:ph:main_res} with $\tilde g$ instead of $g$, and we derive the existence of a weak solution $(u,v)$ to~\eqref{eq:ph_elsys}--\eqref{eq:ph_bc2} with Neumann boundary condition~\eqref{eq:ph_bc4}.
\end{remark}

In the next lemma we show that for $k>d/2$ the Griffith's dynamic energy--dissipation balance can be rephrased in the following identity:
\begin{equation}\label{eq:ph_stcon}
\partial_v\mathcal E(u(t),v(t))[\dot v(t)]+\partial\mathcal H(v(t))[\dot v(t)]+\mathcal G_k(\dot v(t))= 0\quad\text{for a.e. }t\in(0,T),
\end{equation}
where the derivatives $\partial_v\mathcal E$ and $\partial\mathcal H$ take the form
\begin{align*}
&\partial_v\mathcal E(u,v)[\chi]=\frac{1}{2}\int_\Omega \dot b(v(x))\chi(x)\mathbb C(x) Eu(x)\cdot Eu(x)\de x\quad\text{for }u\in H^1(\Omega;\mathbb R^d)\text{ and } v,\chi\in H^1(\Omega)\cap L^\infty(\Omega),\\
&\partial\mathcal H(v)[\chi]=\frac{1}{2\e}\int_\Omega (v(x)-1)\chi(x)\de x+2\e\int_\Omega \nabla v(x)\cdot\nabla\chi(x)\de x\quad\text{for }v,\chi\in H^1(\Omega).
\end{align*}

\begin{lemma}\label{lem:ph_equiv}
Let $k>d/2$ and let $w_1$, $w_2$, $f$, $g$, $u^0$, $u^1$, and $v^0$ be as in~\eqref{eq:ph_bd1}--\eqref{eq:ph_id}. Assume that $(u,v)$ is a weak solution to the problem~\eqref{eq:ph_elsys}--\eqref{eq:ph_bc3} with initial conditions~\eqref{eq:ph_ic}. Then the Griffith's dynamic energy--dissipation balance~\eqref{eq:ph_Genb} is equivalent to the identity~\eqref{eq:ph_stcon}.
\end{lemma}

\begin{proof}
We follow the same techniques of~\cite[Lemma 2.6]{DMS}. Let us fix  $0<h<T$ and let us define the function 
\begin{equation*}
\psi_h(t):=\frac{u(t+h)-u(t)}{h}-\frac{w_1(t+h)-w_1(t)}{h}\quad \text{for }t\in[0,T-h].
\end{equation*}
We use $\psi_h(t)$ as test function in~\eqref{eq:ph_weak_form} first at time $t$, and then at time $t+h$. By summing the two expressions and integrating in a fixed time interval $[t_1,t_2]\subseteq[0,T-h]$, we obtain the identity 
\begin{align}
&\int_{t_1}^{t_2}\spr{\ddot u(t+h)+\ddot u(t)}{\psi_h(t)}_{H^{-1}_{D_1}(\Omega)}\de t+\int_{t_1}^{t_2}(b(v(t+h))\mathbb C Eu(t+h)+b(v(t))\mathbb C Eu(t),E\psi_h(t))_{L^2(\Omega)}\de t\nonumber\\
&=\int_{t_1}^{t_2}(f(t+h)+f(t),\psi_h(t))_{L^2(\Omega)}\de t+\int_{t_1}^{t_2}\spr{g(t+h)+g(t)}{\psi_h(t)}_{H^{-1}_{D_1}(\Omega)}\de t.\label{eq:ph_qqq0}
\end{align}
We study these four terms separately. By performing an integration by parts, the first one becomes 
\begin{align*}
&\int_{t_1}^{t_2}\spr{\ddot u(t+h)+\ddot u(t)}{\psi_h(t)}_{H^{-1}_{D_1}(\Omega)}\de t\\
&=-\int_{t_1}^{t_2}(\dot u(t+h)+\dot u(t),\dot \psi_h(t))_{L^2(\Omega)}\de t+(\dot u(t_2+h)+\dot u(t_2), \psi_h(t_2))_{L^2(\Omega)}-(\dot u(t_1+h)+\dot u(t_1),\psi_h(t_1))_{L^2(\Omega)}\\
&=-\frac{1}{h}\int_{t_2}^{t_2+h}\norm{\dot u(t)}_{L^2(\Omega)}^2\de t+\frac{1}{h}\int_{t_1}^{t_1+h}\norm{\dot u(t)}_{L^2(\Omega)}^2\de t+\frac{1}{h}\int_{t_1}^{t_2}(\dot u(t+h)+\dot u(t),\dot w_1(t+h)-\dot w_1(t))_{L^2(\Omega)}\de t\\
&\quad+(\dot u(t_2+h)+\dot u(t_2), \psi_h(t_2))_{L^2(\Omega)}-(\dot u(t_1+h)+\dot u(t_1),\psi_h(t_1))_{L^2(\Omega)}.
\end{align*}
Since $u,w_1\in C^1([0,T];L^2(\Omega;\mathbb R^d))$, by sending $h\to 0^+$ we deduce
\begin{align}
&\begin{aligned}
\lim_{h\to 0^+}\left[-\frac{1}{h}\int_{t_2}^{t_2+h}\norm{\dot u(t)}_{L^2(\Omega)}^2\de t+\frac{1}{h}\int_{t_1}^{t_1+h}\norm{\dot u(t)}_{L^2(\Omega)}^2\de t\right]=-\norm{\dot u(t_2)}_{L^2(\Omega)}^2+\norm{\dot u(t_1)}_{L^2(\Omega)}^2,
\end{aligned}\\
&\begin{aligned}
&\lim_{h\to 0^+}\left[(\dot u(t_2+h)+\dot u(t_2), \psi_h(t_2))_{L^2(\Omega)}-(\dot u(t_1+h)+\dot u(t_1),\psi_h(t_1))_{L^2(\Omega)}\right]\\
&=2\norm{\dot u(t_2)}_{L^2(\Omega)}^2-2(\dot u(t_2),\dot w_1(t_2))_{L^2(\Omega)}-2\norm{\dot u(t_1)}_{L^2(\Omega)}^2+2(\dot u(t_1),\dot w_1(t_1))_{L^2(\Omega)}.
\end{aligned}
\end{align}
Notice that the sequence $\frac{1}{h}[\dot w_1(\,\cdot\,+h)-\dot w_1]$ converges strongly to $\ddot w_1$ in $L^2(t_1,t_2;L^2(\Omega;\mathbb R^d))$ as $h\to 0^+$, since $\dot w_1$ belongs to $H^1(0,T;L^2(\Omega;\mathbb R^d))$. Therefore, there exist a sequence $h_m\to 0^+$ as $m\to\infty$, and a function $\kappa\in L^2(t_1,t_2)$ such that for a.e. $t\in(t_1,t_2)$
\begin{align*}
&\frac{1}{h_m}(\dot u(t+h_m)+\dot u(t),\dot w_1(t+h_m)-\dot w_1(t))_{L^2(\Omega)}\to 2(\dot u(t),\ddot w_1(t))_{L^2(\Omega)}\quad\text{as $m\to\infty$},\\
&\left|\frac{1}{h_m}(\dot u(t+h_m)+\dot u(t),\dot w_1(t+h_m)-\dot w_1(t))_{L^2(\Omega)}\right|\le 2\norm{\dot u}_{L^\infty(0,T;L^2(\Omega))}\kappa(t)\quad\text{for every $m\in\mathbb N$}.
\end{align*}
By the dominated convergence theorem we derive
\begin{equation}\label{eq:ph_qqq2}
\lim_{h\to 0^+}\frac{1}{h}\int_{t_1}^{t_2}(\dot u(t+h)+\dot u(t),\dot w_1(t+h)-\dot w_1(t))_{L^2(\Omega)}\de t= 2\int_{t_1}^{t_2}(\dot u(t),\ddot w_1(t))_{L^2(\Omega)}\de t,
\end{equation}
since the limit does not depend on the subsequence $\{h_m\}_{m\in\mathbb N}$. For the term involving $f$, we observe that $f(\,\cdot\,+h)\to f$ and $\psi_h\to \dot u-\dot w_1$ in $L^2(t_1,t_2;L^2(\Omega;\mathbb R^d))$ as $h\to 0^+$. Hence, we have
\begin{equation}
\lim_{h\to 0^+}\int_{t_1}^{t_2}(f(t+h)+f(t),\psi_h(t))_{L^2(\Omega)}\de t=2\int_{t_1}^{t_2}(f(t),\dot u(t)-\dot w_1(t))_{L^2(\Omega)}\de t.
\end{equation}
By using the identity
\begin{align*}
&\int_{t_1}^{t_2}\spr{g(t+h)+g(t)}{\psi_h(t)}_{H^{-1}_{D_1}(\Omega)}\de t\\
&=\frac{2}{h}\int_{t_2}^{t_2+h}\spr{g(t)}{u(t)-w_1(t)}_{H^{-1}_{D_1}(\Omega)}\de t-\frac{2}{h}\int_{t_1}^{t_1+h}\spr{g(t)}{u(t)-w_1(t)}_{H^{-1}_{D_1}(\Omega)}\de t\\
&\quad-\frac{1}{h}\int_{t_1}^{t_2}\spr{g(t+h)-g(t)}{u(t+h)+u(t)-w_1(t+h)-w_1(t)}_{H^{-1}_{D_1}(\Omega)}\de t, 
\end{align*}
and proceeding as before, we also deduce
\begin{align}
&\lim_{h\to 0^+}\int_{t_1}^{t_2}\spr{g(t+h)+g(t)}{\psi_h(t)}_{H^{-1}_{D_1}(\Omega)}\de t\\
&=2\spr{g(t_2)}{u(t_2)-w_1(t_2)}_{H^{-1}_{D_1}(\Omega)}-2\spr{g(t_1)}{u(t_1)-w_1(t_1)}_{H^{-1}_{D_1}(\Omega)}-2\int_{t_1}^{t_2}\spr{\dot g(t)}{u(t)-w_1(t)}_{H^{-1}_{D_1}(\Omega)}\de t\nonumber.
\end{align}
It remains to study the last term, that can be rephrased in the following way
\begin{align*}
&\int_{t_1}^{t_2}(b(v(t+h))\mathbb C Eu(t+h)+b(v(t))\mathbb C Eu(t),E\psi_h(t))_{L^2(\Omega)}\de t\\
&=\frac{1}{h}\int_{t_2}^{t_2+h}(b(v(t))\mathbb C Eu(t),Eu(t))_{L^2(\Omega)}\de t-\frac{1}{h}\int_{t_1}^{t_1+h}(b(v(t))\mathbb C Eu(t),Eu(t))_{L^2(\Omega)}\de t\\
&\quad-\frac{1}{h}\int_{t_1}^{t_2}([b(v(t+h))-b(v(t))]\mathbb C Eu(t),Eu(t+h))_{L^2(\Omega)}\de t\\
&\quad-\frac{1}{h}\int_{t_1}^{t_2}(b(v(t+h))\mathbb C Eu(t+h)+b(v(t))\mathbb C Eu(t),Ew_1(t+h)-Ew_1(t))_{L^2(\Omega)}\de t.
\end{align*}
Since $H^k(\Omega)\hookrightarrow C^0(\overline\Omega)$, we deduce that $v$ belongs to the space $C^0([0,T];C^0(\overline\Omega))$. This property, together with $b\in C^1(\R)$ and $u\in C^0([0,T];H^1(\Omega;\R^d))$, implies
\begin{equation}
\begin{aligned}
&\lim_{h\to 0^+}\left[\frac{1}{h}\int_{t_2}^{t_2+h}(b(v(t))\mathbb C Eu(t),Eu(t))_{L^2(\Omega)}\de t-\frac{1}{h}\int_{t_1}^{t_1+h}(b(v(t))\mathbb C Eu(t),Eu(t))_{L^2(\Omega)}\de t\right]\\
&=(b(v(t_2))\mathbb C Eu(t_2),Eu(t_2))_{L^2(\Omega)}-(b(v(t_1))\mathbb C Eu(t_1),Eu(t_1))_{L^2(\Omega)}.
\end{aligned}
\end{equation}
Moreover, the sequence $\frac{1}{h}[v(\,\cdot\,+h)-v]$ converges strongly to $\dot v$ in $L^2(t_1,t_2;C^0(\overline\Omega))$ as $h\to 0^+$. Therefore, there exist a subsequence $h_m\to 0^+$ as $m\to\infty$ and a function $\kappa\in L^2(t_1,t_2)$ such that for a.e. $t\in(t_1,t_2)$
\begin{align*}
&\frac{v(t+h_m)-v(t)}{h_m}\to \dot v(t)\quad\text{in $C^0(\overline\Omega)$}\quad\text{as $m\to\infty$},\\
&\left\Vert \frac{v(t+h_m)-v(t)}{h_m}\right\Vert_{L^\infty(\Omega)}\le \kappa(t)\quad\text{for every $m\in\mathbb N$}.
\end{align*}
Thanks to~\eqref{eq:ph_breg1}, we can apply Lagrange's theorem to derive for a.e. $t\in(t_1,t_2)$
\begin{align*}
&\frac{1}{h_m}(b(v(t+h_m))-b(v(t))\mathbb CEu(t),Eu(t+h_m))_{L^2(\Omega)}\to (\dot b(v(t))\dot v(t)\mathbb CEu(t),Eu(t))_{L^2(\Omega)}\quad\text{as $m\to\infty$},\\
&\left|\frac{1}{h_m}([b(v(t+h_m))-b(v(t))]\mathbb CEu(t), Eu(t+h_m))_{L^2(\Omega)}\right|\le \dot b(1)\norm{\mathbb C}_{L^\infty(\Omega)} \norm{Eu}^2_{L^\infty(0,T;L^2(\Omega))}\kappa(t),
\end{align*}
since $u\in C^0([0,T];H^1(\Omega;\mathbb R^d))$ and $v(t)\le 1$ for every $t\in[0,T]$. The dominated convergence theorem yields
\begin{equation}
\lim_{h\to 0^+}\frac{1}{h}\int_{t_1}^{t_2}([b(v(t+h))-b(v(t))]\mathbb CEu(t),Eu(t+h))_{L^2(\Omega)}\de t=\int_{t_1}^{t_2}(\dot b(v(t))\dot v(t)\mathbb CEu(t),Eu(t))_{L^2(\Omega)}\de t,
\end{equation}
being the limit independent on the sequence $\{h_m\}_{m\in\mathbb N}$. Finally, notice that $\frac{1}{h}[Ew_1(\,\cdot\,+h)-Ew_1]$ converges strongly to $E\dot w_1$ in $L^2(t_1,t_2;L^2(\Omega;\mathbb R^{d\times d}))$ as $h\to 0^+$. By arguing as in~\eqref{eq:ph_qqq2}, this fact gives 
\begin{equation}\label{eq:ph_qqq3}
\begin{aligned}
&\lim_{h\to 0^+}\frac{1}{h}\int_{t_1}^{t_2}(b(v(t+h))\mathbb C Eu(t+h)+b(v(t))\mathbb C Eu(t),Ew_1(t+h)-Ew_1(t))_{L^2(\Omega)}\de t\\
&=2\int_{t_1}^{t_2}(b(v(t))\mathbb CEu(t),E\dot w_1(t))_{L^2(\Omega)}\de t.
\end{aligned}
\end{equation}
We combine together~\eqref{eq:ph_qqq0}--\eqref{eq:ph_qqq3} to derive 
\begin{equation*}
\begin{aligned}
&\mathcal K(\dot u(t_2))+\mathcal E(u(t_2), v(t_2))-\frac{1}{2}\int_{t_1}^{t_2} (\dot b(v(t))\dot v(t)\mathbb CEu(t),Eu(t))_{L^2(\Omega)}\de t\\
&=\mathcal K(\dot u(t_1))+\mathcal E(u(t_1), v(t_1))+\mathcal W_{tot}(u,v;t_1,t_2)
\end{aligned}
\end{equation*}
for every $t_1,t_2\in[0,T)$ with $t_1<t_2$. Since all terms in the previous equality are continuous with respect to $t_2$, we deduce that a weak solution to~\eqref{eq:ph_elsys}--\eqref{eq:ph_bc3} with initial conditions~\eqref{eq:ph_ic} satisfies the energy balance
\begin{equation}\label{eq:ph_eneq2}
\begin{aligned}
&\mathcal K(\dot u(t_2))+\mathcal E(u(t_2), v(t_2))-\frac{1}{2}\int_{t_1}^{t_2} (\dot b(v(t))\dot v(t)\mathbb CEu(t),Eu(t))_{L^2(\Omega)}\de t\\
&=\mathcal K(\dot u(t_1))+\mathcal E(u(t_1),v(t_1))+\mathcal W_{tot}(u,v;t_1,t_2)
\end{aligned}
\end{equation}
for every $t_1,t_2\in[0,T]$ with $t_1<t_2$.

Let us assume now~\eqref{eq:ph_stcon}. Since $v\in H^1(0,T;H^k(\Omega))$, the function $t\mapsto \zeta(t):=\mathcal H(v(t))$ is absolutely continuous on $[0,T]$, with $\dot \zeta(t)=\partial\mathcal H(v(t))[\dot v(t)]$ for a.e. $t\in(0,T)$. By integrating~\eqref{eq:ph_stcon} over $[t_1,t_2]\subseteq[0,T]$, we obtain
\begin{equation}\label{eq:ph:in_stcon}
-\frac{1}{2}\int_{t_1}^{t_2} (\dot b(v(t))\dot v(t)\mathbb CEu(t),Eu(t))_{L^2(\Omega)}\de t=\mathcal H(v(t_2))-\mathcal H(v(t_1))+\int_{t_1}^{t_2}\mathcal G_k(v(t))\de t.
\end{equation}
The above identity, together with~\eqref{eq:ph_eneq2}, implies the Griffith's dynamic energy--dissipation balance~\eqref{eq:ph_Genb}. On the other hand, if~\eqref{eq:ph_Genb} is satisfied, by comparing it with~\eqref{eq:ph_eneq2} we deduce~\eqref{eq:ph:in_stcon} for every interval $[t_1,t_2]\subseteq [0,T]$, from which~\eqref{eq:ph_stcon} follows.
\end{proof}

\begin{remark}\label{rem:ph_varin}
When $k>\frac{d}{2}$, the crack stability condition~\eqref{eq:ph_mincon} is equivalent for a.e. $t\in(0,T)$ to the following variational inequality
\begin{equation}\label{eq:ph_varin}
\partial_v\mathcal E(u(t),v(t))[\chi]+\partial \mathcal H(v(t))[\chi]+\sum_{i=0}^k\alpha_i(\nabla^i\dot v(t),\nabla^i\chi)_{L^2(\Omega)}\ge 0
\end{equation}
among all $\chi\in H^1_{D_2}(\Omega)\cap H^k(\Omega)$ with $\chi\le 0$. Indeed, for every $s\in(0,1]$ we can take $v(t)+s\chi$ as test function in~\eqref{eq:ph_mincon}. After some computations and by dividing by $s$, we deduce
\begin{equation}\label{eq:min_chi}
\begin{aligned}
&\frac{\mathcal E(u(t),v(t)+s\chi)-\mathcal E(u(t),v(t))}{s}+\partial\mathcal H(v(t))[\chi]+\sum_{i=0}^k\alpha_i(\nabla^i\dot v(t),\nabla^i\chi)_{L^2(\Omega)}\\
&\quad+s\left[\frac{1}{4\e}\norm{\chi}_{L^2(\Omega)}^2+\e\norm{\nabla\chi}_{L^2(\Omega)}^2\right] \ge 0. 
\end{aligned}
\end{equation}
Let us fix $x\in\Omega$. By Lagrange's theorem there exists $z_s(t,x)\in [v(t,x)+s\chi(x),v(t,x)]$ such that
\begin{equation*}
\frac{b(v(t,x)+s\chi(x))-b(v(x))}{s}= \dot b(z_s(t,x))\chi(x),
\end{equation*}
since $b\in C^1(\R)$. In particular, we have
\begin{align*}
&\lim_{s\to 0^+}\frac{b(v(t,x)+s\chi(x))-b(v(x))}{s}= \dot b(v(t,x))\chi(x),\\
&\left|\frac{b(v(t,x)+s\chi(x))-b(v(x))}{s}\right|\le \dot b(1)|\chi(x)|,
\end{align*}
because $\dot b\in C^0(\R)$ is non negative, non decreasing, and $z_s(t,x)\le v(t,x)\le 1$. Then, the dominated convergence theorem yields
\begin{equation*}
\lim_{s\to 0^+}\frac{\mathcal E(u(t),v(t)+s\chi)-\mathcal E(u(t),v(t))}{s}=\frac{1}{2}\int_\Omega \dot b(v(t))\chi\mathbb C Eu(t)\cdot Eu(t)\de x=\partial_v \mathcal E(u(t),v(t))[\chi].
\end{equation*}
By sending $s\to 0^+$ in~\eqref{eq:min_chi} we hence deduce~\eqref{eq:ph_varin}. On the other hand, it is easy to check that~\eqref{eq:ph_varin} implies~\eqref{eq:ph_mincon}, by exploiting the convexity of $v^*\to\mathcal E(u(t),v^*)+\mathcal H(v^*)$ and taking $\chi:=v^*-v(t)$ for every $v^*-w_2\in H^1_{D_2}(\Omega)\cap H^k(\Omega)$ with $v^*\le v(t)$.

The inequality~\eqref{eq:ph_varin} for a.e. $t\in(0,T)$ gives that the distribution
\begin{equation*}
-\frac{1}{2}\dot b(v(t))\mathbb CEu(t)\cdot Eu(t)-\frac{1}{2\e}(v(t)-1)+2\e\Delta v(t)-\sum_{i=0}^k\alpha_i(-1)^i\Delta^i \dot v(t)\in\mathcal D'(\Omega)
\end{equation*}
is positive on $\Omega$. Therefore it coincides with a positive Radon measure $\mu(t)$ on $\Omega$, by Riesz's representation theorem. In particular, since $H^k(\Omega)\hookrightarrow C^0(\overline\Omega)$, for a.e. $t\in(0,T)$ we deduce 
\begin{equation*}
\spr{\zeta(t)}{\chi}_{(H^k(\Omega))'}:=\partial_v\mathcal E(u(t),v(t))[\chi]+\partial \mathcal H(v(t))[\chi]+\sum_{i=0}^k\alpha_i(\nabla^i\dot v(t),\nabla^i\chi)_{L^2(\Omega)}=-\int_\Omega\chi\de \mu(t)
\end{equation*}
for every function $\chi\in H^k(\Omega)$ with compact support in $\Omega$. We combine this fact with the identity~\eqref{eq:ph_stcon} to derive for our model an analogous of the classical activation rule in Griffith's criterion: for a.e. $t\in(0,T)$ the positive measure $\mu(t)$ must vanish on the set of points $x\in\Omega$ where $\dot v(t,x)>0$. Indeed, let us consider a sequence $\{\psi_m\}_m\subset C_c^\infty(\Omega)$ such that $0\le \psi_m\le \psi_{m+1}\le 1$ in $\Omega$ for every $m\in\mathbb N$, and $\psi_m(x)\to 1$ for every $x\in\Omega$ as $m\to\infty$. The function $\dot v(t)$ is admissible in~\eqref{eq:ph_varin} for a.e. $t\in(0,T)$, since $\frac{1}{h}[v(t+h)-v(t)]\in H^1_{D_2}(\Omega)$ converges strongly to $\dot v(t)$ in $H^k(\Omega)$ as $h\to 0^+$, and $t\mapsto v(t)$ is non decreasing in $[0,T]$. Therefore, thanks to~\eqref{eq:ph_stcon} and~\eqref{eq:ph_varin}, for a.e. $t\in(0,T)$ we get
\begin{align*}
0=\spr{\zeta(t)}{\dot v(t)}_{(H^k(\Omega))'}&=\spr{\zeta(t)}{\dot v(t)\psi_m}_{(H^k(\Omega))'}+\spr{\zeta(t)}{\dot v(t)(1-\psi_m)}_{(H^k(\Omega))'}\\
&\ge \spr{\zeta(t)}{\dot v(t)\psi_m}_{(H^k(\Omega))'}=-\int_\Omega \dot v(t)\psi_m\de\mu(t)\ge 0,
\end{align*}
because $\dot v(t)\psi_m\in H^k(\Omega)$ has compact support. Hence, for a.e. $t\in(0,T)$ we have
\begin{align*}
0=\lim_{m\to\infty}\int_\Omega \dot v(t)\psi_m\de\mu(t)=\int_\Omega\dot v(t)\de \mu(t),
\end{align*}
by the monotone convergence theorem, which implies our activation condition.
\end{remark}

%------------------------------------------
% The time discretization scheme
%------------------------------------------

\section{The time discretization scheme}\label{sec:ph3}

In this section we show some general results that are true for every $k\in\mathbb N\cup\{0\}$. In particular, we prove that the problem~\eqref{eq:ph_elsys}--\eqref{eq:ph_ic} admits always a solution $(u,v)$ less regular in time which satisfies the irreversibility condition~\eqref{eq:ph_irr} and the crack stability condition~\eqref{eq:ph_mincon}. Throughout this section, we always assume that $w_1$, $w_2$, $f$, $g$, $u^0$, $u^1$, and $v^0$ satisfy~\eqref{eq:ph_bd1}--\eqref{eq:ph_id}.

We start by introducing the following notion of solution, which requires less regularity on the time variable. 

\begin{definition}\label{def:ph_gensol}
The pair $(u,v)$ is a {\it generalized solution} to~\eqref{eq:ph_elsys}--\eqref{eq:ph_bc3} if 
\begin{align}
&u\in L^\infty(0,T;H^1(\Omega;\mathbb R^d))\cap W^{1,\infty}(0,T;L^2(\Omega;\mathbb R^d))\cap H^2(0,T;H^{-1}_{D_1}(\Omega;\mathbb R^d)),\label{eq:ph_greg1}\\
&u(t)-w_1(t)\in H^1_{D_1}(\Omega;\mathbb R^d)\text{ for every }t\in[0,T],\label{eq:ph_greg2}\\
&v\in L^\infty(0,T;H^1(\Omega))\cap H^1(0,T;H^k(\Omega)),\label{eq:ph_greg3}\\
&v(t)-w_2\in H^1_{D_2}(\Omega)\text{ and }v(t)\le 1\text{ in }\Omega\text{ for every }t\in[0,T],\label{eq:ph_greg4}
\end{align}
and for a.e. $t\in(0,T)$ the equation~\eqref{eq:ph_weak_form} holds.
\end{definition}

\begin{remark}\label{rem:ph_weak_con}
We recall that, given two reflexive Banach spaces $X$ and $Y$, with continuous embedding $X\hookrightarrow Y$, we have
\begin{equation*}
C_w^0([0,T];Y)\cap L^\infty(0,T;X)= C_w^0([0,T];X),
\end{equation*}
see, for instance~\cite[Chapitre XVIII, \S 5, Lemme 6]{DL}. In particular, if $u\in C_w^0([0,T];X)$, then
\begin{equation*}
    \norm{u(t)}_X\le \norm{u}_{L^\infty(0,T;X)}\quad\text{for every $t\in[0,T]$}.
\end{equation*}
By applying this result to a generalized solution $(u,v)$ to~\eqref{eq:ph_elsys}--\eqref{eq:ph_bc3}, we get that $u\in C_w^0([0,T];H^1(\Omega;\mathbb R^d))$, $\dot u\in C_w^0([0,T];L^2(\Omega;\mathbb R^d))$, and $v\in C_w^0([0,T];H^1(\Omega))$. Therefore, the initial conditions~\eqref{eq:ph_ic} makes sense, since the functions $u(t)$, $\dot u(t)$, and $v(t)$ are uniquely defined for every $t\in[0,T]$ as elements of $H^1(\Omega;\mathbb R^d)$, $L^2(\Omega;\mathbb R^d)$, and $H^1(\Omega)$, respectively. 
\end{remark}

To show the existence of a generalized solution to~\eqref{eq:ph_elsys}--\eqref{eq:ph_bc3}, we approximate our problem by mean of a time discretization with an alternate scheme, as done in~\cite{LOS}. We divide the time interval $[0,T]$ by introducing $n$ equispaced nodes, and in each of them we first solve the elastodynamics system~\eqref{eq:ph_ell} with $v$ fixed, and then the crack stability condition~\eqref{eq:ph_mincon} with $u$ fixed. Finally, we consider some interpolants of the discrete solutions and, thanks to an a priori estimate, we pass to the limit as $n\to\infty$.

We fix $n\in\mathbb N$, and we set 
\begin{align*}
&\tau_n=\frac{T}{n},\quad u^0_n:=u^0,\quad u^{-1}_n:=u^0-\tau_nu^1,\quad v^0_n:=v^0,\\
&g_n^j:=g(j\tau_n),\quad w_n^j:=w_1(j\tau_n)\quad\text{for } j=0,\dots,n,\quad f_n^j:=\frac{1}{\tau_n}\int_{(j-1)\tau_n}^{j\tau_n}f(s)\de s\quad\text{for } j=1,\dots,n.
\end{align*}
For $j=1,\dots,n$ we consider the following two minimum problems: 
\begin{itemize}
\item[$(i)$] $u_n^j-w_n^j\in H^1_{D_1}(\Omega;\R^d)$ is the minimizer of
\begin{align*}
u^*\mapsto\frac{1}{2\tau_n^2}\left\Vert u^*-2u_n^{j-1}-u_n^{j-2}\right\Vert_{L^2(\Omega)}^2+\mathcal E(u^*,v_n^{j-1})-(f_n^j,u^*)_{L^2(\Omega)}-\spr{g_n^j}{u^*-w_n^j}_{H^{-1}_{D_1}(\Omega)}
\end{align*}
among every $u^*-w_n^j\in H^1_{D_1}(\Omega;\R^d)$;
\item[$(ii)$] $v_n^j-w_2\in H^1_{D_2}(\Omega)\cap H^k(\Omega)$ with $v_n^j\le v_n^{j-1}$ is the minimizer of
\begin{align*}
&v^*\mapsto\mathcal E(u_n^j,v^*)+\mathcal H(v^*)+\frac{1}{2\tau_n}\mathcal G_k(v^*-v_n^{j-1})
\end{align*}
among every $v^*-w_2\in H^1_{D_2}(\Omega)\cap H^k(\Omega)$ with $v^*\le v_n^{j-1}$.
\end{itemize}
Since $\mathbb C$ and $b$ satisfy~\eqref{eq:ph_A1}--\eqref{eq:ph_A3}, \eqref{eq:ph_breg1} and~\eqref{eq:ph_breg2} the two discrete problems are well defined. In particular, for every $j=1,\dots,n$ there exists a unique pair $(u_n^j,v_n^j)\in H^1(\Omega;\mathbb R^d)\times (H^1(\Omega)\cap H^k(\Omega))$ solution to~$(i)$ and~$(ii)$. 

Let us define
\begin{align*}
&\delta u_n^j:=\frac{u_n^j-u_n^{j-1}}{\tau_n}\quad\text{for } j=0,\dots,n,\quad\delta^2 u_n^j:=\frac{\delta u_n^j-\delta u_n^{j-1}}{\tau_n},\quad \delta v_n^j:=\frac{v_n^j-v_n^{j-1}}{\tau_n}\quad\text{for } j=1,\dots,n.
\end{align*}
For $j=1,\dots,n$ the minimality of $u_n^j$ implies
\begin{equation}\label{eq:ph_un2}
(\delta^2 u_n^j,\psi)_{L^2(\Omega)}+(b(v_n^{j-1})\mathbb C Eu_n^j,E\psi)_{L^2(\Omega)}=(f_n^j,\psi)_{L^2(\Omega)}+\spr{g_n^j}{\psi}_{H^{-1}_{D_1}(\Omega)}
\end{equation}
for every $\psi\in H^1_{D_1}(\Omega;\mathbb R^d)$, which is the discrete counterpart of~\eqref{eq:ph_weak_form}. Moreover, we can characterize the function $v_n^j$ in the following way.

\begin{lemma}\label{lem:ph_dvarin}
For $j=1,\dots,n$ the function $v_n^j-w_2\in H^1_{D_2}(\Omega)\cap H^k(\Omega)$ with $v_n^j\le v_n^{j-1}$ is the unique solution to the variational inequality
\begin{equation}\label{eq:ph_dvarin}
\mathcal E(u_n^j,v^*)-\mathcal E(u_n^j,v_n^j)+\partial \mathcal H(v_n^j)[v^*-v_n^j]+\sum_{i=0}^k\alpha_i(\nabla^i\delta v_n^j,\nabla^iv^*-\nabla^i v_n^j)_{L^2(\Omega)}\ge 0
\end{equation}
among all $v^*-w_2\in H^1_{D_2}(\Omega)\cap H^k(\Omega)$ with $v^*\le v_n^{j-1}$. In particular, we have $v_n^j\le 1$ in $\Omega$ and
\begin{equation}\label{eq:ph_steq}
\frac{\mathcal E(u_n^j,v_n^j)-\mathcal E(u_n^j,v_n^{j-1})}{\tau_n}+\partial\mathcal H(v_n^j)[\delta v_n^j]+\mathcal G_k(\delta v_n^j)\le 0.
\end{equation}
Finally, if $k=0$, $w_2\ge 0$ on $\partial_{D_2}\Omega$, $v^0\ge 0$ in $\Omega$, and $b(s)=(\max\{s,0\})^2+\eta$ for $s\in\R$, then $v_n^j\ge 0$ in $\Omega$ for every $j=1,\dots,n$.
\end{lemma}

\begin{proof}
Let $v_n^j$ be the solution to $(ii)$ and let $v^*-w_2\in H^1_{D_2}(\Omega)\cap H^k(\Omega)$ be such that $v^*\le v_n^{j-1}$. For every $s\in(0,1]$ the function $v_n^j+s(v^*-v_n^j)$ is a competitor for $(ii)$. Hence, by exploiting the minimality of $v_n^j$ and dividing by $s$, we deduce the following inequality
\begin{equation}\label{eq:ph_ft}
\begin{aligned}
&\frac{\mathcal E(u_n^j,v_n^j+s(v^*-v_n^j))-\mathcal E(u_n^j,v_n^j)}{s}+\partial\mathcal H(v_n^j)[v^*-v_n^j]+\sum_{i=0}^k\alpha_i(\nabla^i\delta v_n^j,\nabla^iv^*-\nabla^i v_n^j)_{L^2(\Omega)}\\
&\quad+s\left[\frac{1}{4\e}\norm{v^*-v_n^j}_{L^2(\Omega)}^2+\e\norm{\nabla v^*-\nabla v_n^j}_{L^2(\Omega)}^2+\frac{1}{2\tau_n}\mathcal G_k(v^*-v_n^j)\right]\ge 0.
\end{aligned}
\end{equation}
Notice that
\begin{equation}\label{eq:ph_ft2}
\frac{\mathcal E(u_n^j,v_n^j+s(v^*-v_n^j))-\mathcal E(u_n^j,v_n^j)}{s}\le \mathcal E(u_n^j,v^*)-\mathcal E(u_n^j,v_n^j)\quad\text{for every }s\in(0,1],
\end{equation}
since the difference quotients are non decreasing in $s\in(0,1]$, being $b$ is convex. By combining~\eqref{eq:ph_ft} with~\eqref{eq:ph_ft2} and passing to the limit as $s\to 0^+$, we derive~\eqref{eq:ph_dvarin}. On the other hand, it is easy to see that every solution to~\eqref{eq:ph_dvarin} satisfies $(ii)$, thanks to the convexity of $\mathcal H$ and $\mathcal G_k$. Finally, for every $j=1,\dots,n$ we have $v_n^j\le v^0\le 1$ in $\Omega$, and the inequality~\eqref{eq:ph_steq} is obtained by taking $v^*=v_n^{j-1}$ in~\eqref{eq:ph_dvarin} and dividing by $\tau_n$.

Let us assume that $k=0$, $w_2\ge 0$ on $\partial_{D_2}\Omega$, $v^0\ge 0$ in $\Omega$, and $b(s)=(\max\{s,0\})^2+\eta$ for $s\in\R$. The function $(v_n^1)^+:=\max\{v^1_n,0\}\in H^1(\Omega)$ is a competitor for $(ii)$ and satisfies
\begin{equation*}
\mathcal E(u^1_n,(v_n^1)^+)+\mathcal H((v_n^1)^+)+\frac{1}{2\tau_n}\mathcal G_0((v_n^1)^+-v^0)\le \mathcal E(u^1_n,v^1_n)+\mathcal H(v^1_n)+\frac{1}{2\tau_n}\mathcal G_0(v^1_n-v^0),
\end{equation*}
being $\mathcal E(u^1_n,(v_n^1)^+)=\mathcal E(u^1_n,v^1_n)$, $\mathcal H((v_n^1)^+)\le \mathcal H(v_n^1)$, and $|(v_n^1)^+-v^0|\le |v^1_n-v^0|$ in $\Omega$, which is a consequence of $v^0\ge 0$. Hence, the function $(v_n^1)^+$ solves $(ii)$. This implies $v^1_n=(v_n^1)^+\ge 0$ in $\Omega$, since the minimum point is unique (the $L^2$--norm is strictly convex). We now proceed by induction: if $v_n^{j-1}\ge 0$ in $\Omega$, we can argue as before to get
\begin{equation*}
\mathcal E(u_n^j,(v_n^j)^+)+\mathcal H((v_n^j)^+)+\frac{1}{2\tau_n}\mathcal G_0((v_n^j)^+-v_n^{j-1})\le \mathcal E(u_n^j,v_n^j)+\mathcal H(v_n^j)+\frac{1}{2\tau_n}\mathcal G_0(v_n^j-v_n^{j-1}),
\end{equation*}
which gives $v_n^j=(v_n^j)^+:=\max\{v_n^j,0\}\ge0$ in $\Omega$ for every $j=1\dots,n$.
\end{proof}

As done in~\cite{LOS}, we combine the equation~\eqref{eq:ph_un2} with the inequality~\eqref{eq:ph_steq} to derive a discrete energy inequality for the family $\{(u_n^j, v_n^j)\}_{j=1}^n$.

\begin{lemma}\label{lem:ph_denin}
The family $\{(u_n^j,v_n^j)\}_{j=1}^n$, solution to problems $(i)$ and $(ii)$, satisfies for every $j=1,\dots,n$ the discrete energy inequality
\begin{equation}\label{eq:ph_denin}
\begin{aligned}
&\mathcal F(u^j_n, \delta u^j_n, v^j_n)+\sum_{l=1}^j \tau_n\mathcal G_k(\delta v_n^l)+\sum_{l=1}^j \tau_n^2D_n^l\\
&\le \mathcal F(u^0, u^1, v^0)+\sum_{l=1}^j \tau_n\left[(f_n^l,\delta u_n^l-\delta w_n^l)_{L^2(\Omega)}+(b(v_n^{l-1})\mathbb CEu_n^l,E\delta w_n^l)_{L^2(\Omega)}\right]\\
&\quad-\sum_{l=1}^j \tau_n\left[(\delta u_n^{l-1},\delta^2 w_n^l)_{L^2(\Omega)}-\spr{\delta g_n^l}{u_n^{l-1}-w^{i-1}_n}_{H^{-1}_{D_1}(\Omega)}\right]+(\delta u^j_n,\delta w^j_n)_{L^2(\Omega)}\\
&\quad+\spr{g^j_n}{u^j_n-w^j_n}_{H^{-1}_{D_1}(\Omega)}-(u^1,\dot w_1(0))_{L^2(\Omega)}-\spr{g(0)}{u^0-w_1(0)}_{H^{-1}_{D_1}(\Omega)},
\end{aligned}
\end{equation}
where $\delta w^0_n:=\dot w_1(0)$, $\delta w_n^j:=\frac{1}{\tau_n}[w_n^j-w_n^{j-1}]$, $\delta^2 w_n^j:=\frac{1}{\tau_n}[\delta w_n^j-\delta w_n^{j-1}]$, $\delta g_n^j:=\frac{1}{\tau_n}[g_n^j-g^{n-1}_n]$ for $j=1,\dots,n$, and the dissipations $D_n^j$ are defined as
\begin{equation*}
D_n^j:=\frac{1}{2}\norm{\delta^2u_n^j}_{L^2(\Omega)}^2+\frac{1}{2}(b(v_n^{j-1})\mathbb C E\delta u_n^j,E\delta u_n^j)_{L^2(\Omega)}+\frac{1}{4\e}\norm{\delta v_n^j}_{L^2(\Omega)}^2+\e \norm{\nabla\delta v_n^j}_{L^2(\Omega)}^2\quad\text{for }j=1,\dots,n.
\end{equation*} 
\end{lemma}

\begin{proof}
By using $\psi=\tau_n[\delta u_n^j-\delta w_n^j]\in H^1_{D_1}(\Omega;\mathbb R^d)$ as test function in~\eqref{eq:ph_un2}, for every $j=1,\dots,n$ we deduce the following identity
\begin{align}
&\tau_n(\delta^2u_n^j,\delta u_n^j)_{L^2(\Omega)}+\tau_n(b(v_n^{j-1})\mathbb C Eu_n^j,E\delta u_n^j)_{L^2(\Omega)}\label{eq:ph_bb1}\\
&=\tau_n\left[(f_n^j,\delta u_n^j-\delta w_n^j)_{L^2(\Omega)}+\spr{g_n^j}{\delta u_n^j-\delta w_n^j}_{H^{-1}_{D_1}(\Omega)}+(\delta^2 u_n^j,\delta w_n^j)_{L^2(\Omega)}+(b(v_n^{j-1})\mathbb CEu_n^j,E\delta w_n^j)_{L^2(\Omega)}\right].\nonumber
\end{align}
Thanks to the identity $\abs{a}^2-a\cdot b=\frac{1}{2}\abs{a}^2-\frac{1}{2}\abs{b}^2+\frac{1}{2}\abs{a-b}^2$ for $a,b\in\mathbb R^d$, we can write the first term as
\begin{align}
\tau_n(\delta^2u_n^j,\delta u_n^j)_{L^2(\Omega)}=\norm{\delta u_n^j}_{L^2(\Omega)}^2-(\delta u_n^{j-1},\delta u_n^j)_{L^2(\Omega)}=\mathcal K(\delta u_n^j)-\mathcal K(\delta u_n^{j-1})+\frac{\tau_n^2}{2}\norm{\delta^2u_n^j}_{L^2(\Omega)}^2.
\end{align}
Similarly, we have
\begin{equation}
\begin{aligned}
\tau_n(b(v_n^{j-1})\mathbb C Eu_n^j,E\delta u_n^j)_{L^2(\Omega)}&=\mathcal E(u_n^j, v_n^j)-\mathcal E(u_n^{j-1},v_n^{j-1})+\frac{\tau_n^2}{2}(b(v_n^{j-1})\mathbb C E\delta u_n^j,E\delta u_n^j)_{L^2(\Omega)}\\
&\quad+\frac{1}{2}([b(v_n^{j-1})-b(v_n^j)]\mathbb C Eu_n^j,Eu_n^j)_{L^2(\Omega)}.
\end{aligned}
\end{equation}
We use~\eqref{eq:ph_steq} to estimate from below the last term in the previous inequality in the following way
\begin{equation}\label{eq:ph_bb4}
\begin{aligned}
&\frac{1}{2}([b(v_n^{j-1})-b(v_n^j)]\mathbb C Eu_n^j,Eu_n^j)_{L^2(\Omega)}\\
&\ge\frac{\tau_n}{2\e}(v_n^j-1,\delta v_n^j)_{L^2(\Omega)}+2\e\tau_n(\nabla v_n^j,\nabla\delta v_n^j)_{L^2(\Omega)}+\tau_n\mathcal G_k(\delta v_n^j)\\
&=\mathcal H(v_n^j)-\mathcal H(v_n^{j-1})+\tau_n\mathcal G_k(\delta v_n^j)+\frac{\tau_n^2}{4\e}\norm{\delta v_n^j}_{L^2(\Omega)}^2+\e \tau_n^2\norm{\nabla\delta v_n^j}_{L^2(\Omega)}^2.
\end{aligned}
\end{equation}
By combining~\eqref{eq:ph_bb1}--\eqref{eq:ph_bb4}, for every $j=1,\dots,n$ we obtain
\begin{align*}
&\mathcal F(u_n^j,\delta u_n^j,v_n^j)-\mathcal F(u_n^{j-1},\delta u_n^{j-1},v_n^{j-1})+\tau_n\mathcal G_k(\delta v_n^j)+\tau_n^2D_n^j\\
&\le \tau_n\left[(f_n^j,\delta u_n^j-\delta w_n^j)_{L^2(\Omega)}+\spr{g_n^j}{\delta u_n^j-\delta w_n^j}_{H^{-1}_{D_1}(\Omega)}+(\delta^2 u_n^j,\delta w_n^j)_{L^2(\Omega)}+(b(v_n^{j-1})\mathbb CEu_n^j,E\delta w_n^j)_{L^2(\Omega)}\right]\nonumber.
\end{align*}
Finally, we sum over $l=1,\dots,j$ for every $j\in\{1,\dots,n\}$, and we use the identities
\begin{align}
&\begin{aligned}
\sum_{l=1}^j\tau_n\spr{g_n^l}{\delta u_n^l-\delta w_n^l}_{H^{-1}_{D_1}(\Omega)}&=\spr{g^j_n}{u^j_n-w^j_n}_{H^{-1}_{D_1}(\Omega)}-\spr{g(0)}{u^0-w_1(0)}_{H^{-1}_{D_1}(\Omega)}\\
&\quad-\sum_{l=1}^j\tau_n\spr{\delta g_n^l}{u_n^{l-1}-w_n^{l-1}}_{H^{-1}_{D_1}(\Omega)},\label{eq:ph_gn}
\end{aligned}\\
&\sum_{l=1}^j\tau_n(\delta^2 u_n^l,\delta w_n^l)_{L^2(\Omega)}=(\delta u_n^j,\delta w^j_n)_{L^2(\Omega)}-(u^1,w_1(0))_{L^2(\Omega)}-\sum_{l=1}^j\tau_n(\delta u_n^{l-1},\delta^2w_n^l)_{L^2(\Omega)},\label{eq:ph_wn}
\end{align}
to deduce the discrete energy inequality~\eqref{eq:ph_denin}.
\end{proof}

The first consequence of~\eqref{eq:ph_denin} is the following a priori estimate.

\begin{lemma}\label{lem:ph_bound}
There exists a constant $C>0$, independent of $n$, such that
\begin{equation}\label{eq:ph_est}
\max_{j=1,\dots,n}\{\norm{\delta u_n^j}_{L^2(\Omega)}+\norm{u_n^j}_{H^1(\Omega)}+\norm{v_n^j}_{H^1(\Omega)}\}+\sum_{j=1}^n\tau_n\norm{\delta v_n^j}^2_{H^k(\Omega)}+\sum_{j=1}^n\tau_n^2D_n^j\le C.
\end{equation}
\end{lemma}

\begin{proof}
Thanks to~\eqref{eq:ph_ell} and~\eqref{eq:ph_breg2} we can estimate from below the left--hand side of~\eqref{eq:ph_denin} as
\begin{equation}\label{eq:ph_duj}
\begin{aligned}
&\mathcal F(u^j_n, \delta u^j_n, v^j_n)+\sum_{l=1}^j \tau_n\mathcal G_k(\delta v_n^l)+\sum_{l=1}^j\tau_n^2 D_n^l\ge\frac{1}{2}\norm{\delta u^j_n}_{L^2(\Omega)}^2+\frac{\eta c_0}{2}\norm{u^j_n}_{H^1(\Omega)}^2-\frac{\eta c_1}{2}\norm{u_n^j}_{L^2(\Omega)}^2
\end{aligned}
\end{equation}
for every $j=1,\dots,n$. Let us now bound from above the right--hand side of~\eqref{eq:ph_duj}. We define
\begin{equation*}
L_n:=\max_{j=1,\dots,n}\norm{\delta u_n^j}_{L^2(\Omega)},\quad M_n:=\max_{j=1,\dots,n}\norm{u_n^j}_{H^1(\Omega)},
\end{equation*}
and we use~\eqref{eq:ph_bd1}--\eqref{eq:ph_id} to derive for every $j=1,\dots,n$ the following estimates:
\begin{align}
&\sum_{l=1}^j\tau_n(f_n^l,\delta u_n^l-\delta w_n^l)_{L^2(\Omega)}\le C_1L_n+C_2,\\
&(\delta u_n^j,\delta w^j_n)_{L^2(\Omega)}-(u^1,w_1(0))_{L^2(\Omega)}-\sum_{l=1}^j\tau_n(\delta u_n^{l-1},\delta^2 w_n^l)_{L^2(\Omega)}\le C_1 L_n+C_2,\\
&\spr{g^j_n}{u^j_n-w^j_n}_{H^{-1}_{D_1}(\Omega)}-\spr{g(0)}{u^0-w_1(0)}_{H^{-1}_{D_1}(\Omega)}-\sum_{l=1}^j\tau_n\spr{\delta g_n^l}{u_n^{l-1}-w_n^{l-1}}_{H^{-1}_{D_1}(\Omega)}\le C_1M_n+C_2,
\end{align}
for two positive constants $C_1$ and $C_2$ independent of $n$. Moreover, since $\mathbb C\in L^\infty(\Omega;\mathscr L(\mathbb R^{d\times d};\mathbb R^{d\times d}))$, $b$ is non decreasing, and $v_n^{j-1}\le 1$, we get 
\begin{equation}\label{eq:ph_lmn}
\sum_{l=1}^j\tau_n(b(v_n^{l-1})\mathbb CEu_n^l,E\delta w_n^l)_{L^2(\Omega)}\le b(1)\norm{\mathbb C}_{L^\infty(\Omega)}\sqrt{T}\norm{E\dot w_1}_{L^2(0,T;L^2(\Omega))} M_n
\end{equation} 
for every $j=1,\dots,n$. By combining~\eqref{eq:ph_denin} with~\eqref{eq:ph_duj}--\eqref{eq:ph_lmn} and the following estimate
\begin{equation*}
\norm{u^j_n}_{L^2(\Omega)}\le \sum_{l=1}^n\tau_n\norm{\delta u_n^l}_{L^2(\Omega)}+\norm{u^0}_{L^2(\Omega)}\le T L_n +\norm{u^0}_{L^2(\Omega)}\quad\text{for every }j=1,\dots,n,
\end{equation*} 
we obtain the existence of two positive constants $\tilde C_1$ and $\tilde C_2$, independent of $n$, such that $$(L_n+M_n)^2\le \tilde C_1(L_n+M_n)+\tilde C_2\quad\text{for every $n\in\mathbb N$}.$$ 
This implies that $L_n$ and $M_n$ are uniformly bounded in $n$. In particular, there exists a constant $C>0$, independent of $n$, such that 
\begin{equation*}
\mathcal K(\delta u^j_n)+\mathcal E(u^j_n, v^j_n)+\mathcal H(v^j_n)+\sum_{l=1}^j \tau_n\mathcal G_k(\delta v_n^l)+\sum_{l=1}^j\tau_n^2D_n^l\le C\quad\text{for every }j=1,\dots,n.
\end{equation*}
Finally, for $j=1,\dots,n$ we have
\begin{align*}
\min\left\{\frac{1}{4\e},\e\right\}\norm{v^j_n-1}_{H^1(\Omega)}^2\le\mathcal H(v^j_n)\le C,\quad\beta_0\sum_{j=1}^n\tau_n\norm{\delta v_n^j}_{H^k(\Omega)}^2\le \sum_{j=1}^n\tau_n\mathcal G_k(\delta v_n^j)\le C,
\end{align*}
which gives the remaining estimates.
\end{proof}

\begin{remark}\label{rem:ph_usec}
By combining together~\eqref{eq:ph_un2} and~\eqref{eq:ph_est} we also obtain
\begin{equation*}
\sum_{j=1}^n\tau_n\norm{\delta^2u_n^j}^2_{H^{-1}_{D_1}(\Omega)}+\max_{j=1,\dots,n}\norm{v_n^j}_{H^k(\Omega)}\le C
\end{equation*}
for a positive constant $C$ independent of $n$. Indeed, by~\eqref{eq:ph_un2}, for every $j=1,\dots,n$ we have
\begin{equation*}
\norm{\delta^2 u_n^j}_{H^{-1}_{D_1}(\Omega)}=\sup_{\substack{\psi\in H^1_{D_1}(\Omega;\mathbb R^d) \\ \norm{ \psi}_{H^1(\Omega)}\le 1}}\abs{(\delta^2u_n^j,\psi)_{L^2(\Omega)}}\le b(1)\norm{\mathbb C}_{L^\infty(\Omega)}\norm{Eu_n^j}_{L^2(\Omega)}+\norm{ f_n^j}_{L^2(\Omega)}+\norm{g_n^j}_{H^{-1}_{D_1}(\Omega)}.
\end{equation*}
Hence, thanks to~\eqref{eq:ph_bd2} and~\eqref{eq:ph_est}, there exists a constant $C>0$, independent of $n$, such that 
\begin{equation*}
\sum_{j=1}^n \tau_n\norm{\delta^2 u_n^j}^2_{H^{-1}_{D_1}(\Omega)}\le C(1+\norm{f}_{L^2(0,T;L^2(\Omega))}+\norm{g}_{L^\infty(0,T;H^{-1}_{D_1}(\Omega))}).
\end{equation*}
Finally, also $\norm{v^j_n}_{H^k(\Omega)}$ is uniformly bounded with respect to $j$ and $n$, since 
\begin{equation*}\norm{v^j_n}_{H^k(\Omega)}\le \sqrt{T}\left(\sum_{l=1}^n\tau_n\norm{\delta v_n^l}^2_{H^k(\Omega)}\right)^{1/2}+\norm{v^0}_{H^k(\Omega)}\quad\text{for every }j=1,\dots,n.
\end{equation*}
\end{remark}

We now use the family $\{(u_n^j,v_n^j)\}_{j=1}^n$ to construct a generalized solution to the problem~\eqref{eq:ph_elsys}--\eqref{eq:ph_mincon}. We denote by $u_n\colon [0,T]\to H^1(\Omega;\R^d)$ the piecewise affine interpolant of $\{u_n^j\}_{j=1}^n$, which is defined as
\begin{align*}
&u_n(t):=u_n^j+(t-j\tau_n)\delta u_n^j\quad\text{for } t\in[(j-1)\tau_n,j\tau_n]\text{ and } j=1,\dots,n.
\end{align*}
Moreover, let us define the backward interpolant $\overline u_n\colon [0,T]\to H^1(\Omega;\R^d)$ and the forward interpolant $\underline u_n\colon [0,T]\to H^1(\Omega;\R^d)$ in the following way:
\begin{align*}
&\overline u_n(0)=u^0_n,\quad\overline u_n(t):=u_n^j\quad\text{for }t\in((j-1)\tau_n,j\tau_n]\text{ and } j=1,\dots,n,\\
&\underline u_n(T)=u^n_n,\quad\underline u_n(t):=u_n^{j-1}\quad\text{for }t\in[(j-1)\tau_n,j\tau_n)\text{ and } j=1,\dots,n.
\end{align*}
Similarly, we define the piecewise affine interpolant $v_n\colon [0,T]\to H^1(\Omega)$ of $\{v_n^j\}_{j=1}^n$, as well as the backward interpolant $\overline v_n\colon [0,T]\to H^1(\Omega)$, and the forward interpolant $\underline v_n\colon [0,T]\to H^1(\Omega)$. Finally, we consider the  piecewise affine interpolant $u_n'\colon [0,T]\to L^2(\Omega;\R^d)$ of $\{\delta u_n^j\}_{j=1}^n$, together with the backward interpolant $\overline u_n'\colon [0,T]\to L^2(\Omega;\R^d)$ and the forward interpolant  $\underline u_n'\colon [0,T]\to L^2(\Omega;\R^d)$. Notice that $u_n\in H^1(0,T;L^2(\Omega;\mathbb R^d))$, $u'_n\in H^1(0,T;H^{-1}_{D_1}(\Omega;\mathbb R^d))$, and $v_n\in H^1(0,T;H^k(\Omega))$, with $\dot u_n(t)=\overline u_n'(t)=\delta u_n^j$, $\dot u_n'(t)=\delta^2 u_n^j$, and $\dot v_n(t)=\delta v_n^j$ for $t\in((j-1)\tau_n,j\tau_n)$ and $j=1,\dots,n$.

\begin{lemma}\label{lem:ph_con}
There exist a subsequence of $n$, not relabeled, and two functions
\begin{align*}
&u\in L^\infty(0,T;H^1(\Omega;\mathbb R^d))\cap W^{1,\infty}(0,T;L^2(\Omega;\mathbb R^d))\cap H^2(0,T;H^{-1}_{D_1}(\Omega;\mathbb R^d)),\\
&v\in L^\infty(0,T;H^1(\Omega))\cap H^1(0,T;H^k(\Omega)),
\end{align*}
such that the following convergences hold as $n\to\infty$:
\begin{align*}
&u_n\rightharpoonup u\quad\text{in }H^1(0,T;L^2(\Omega;\mathbb R^d)),& & u_n'\rightharpoonup \dot u\quad\text{in }H^1(0,T;H^{-1}_{D_1}(\Omega;\mathbb R^d)),\\
&u_n\to u\quad\text{in }C^0([0,T];L^2(\Omega;\mathbb R^d)),& & u_n'\to \dot u\quad\text{in }C^0([0,T];H^{-1}_{D_1}(\Omega;\mathbb R^d)),\\
&\overline u_n,\underline u_n\rightharpoonup u\quad\text{in }L^2(0,T;H^1(\Omega;\mathbb R^d)),& & \overline u_n',\underline u_n'\rightharpoonup \dot u\quad\text{in }L^2(0,T;L^2(\Omega;\mathbb R^d)),\\
&v_n\rightharpoonup v\quad\text{in }H^1(0,T;H^k(\Omega)),& &v_n\to v\quad\text{in }C^0([0,T];L^2(\Omega)),\\
&v_n\rightharpoonup v\quad\text{in }L^2(0,T;H^1(\Omega)),& &\overline v_n,\underline v_n\rightharpoonup v\quad\text{in }L^2(0,T;H^1(\Omega)).
\end{align*}
\end{lemma}

\begin{proof}
Thanks to the estimate~\eqref{eq:ph_est}, the sequence $\{u_n\}_{n\in\mathbb N}\subset L^\infty(0,T;H^1(\Omega;\mathbb R^d))\cap W^{1,\infty}(0,T;L^2(\Omega;\mathbb R^d))$ is uniformly bounded. Hence, by Aubin--Lions's lemma (see~\cite[Corollary 4]{S}), there exist a subsequence of $n$, not relabeled, and a function
\begin{equation*}
u\in L^\infty(0,T;H^1(\Omega;\mathbb R^d))\cap W^{1,\infty}(0,T;L^2(\Omega;\mathbb R^d)),
\end{equation*}
such that 
\begin{align*}
&u_n\rightharpoonup u\quad\text{in }H^1(0,T;L^2(\Omega;\mathbb R^d)),\quad  u_n\to u\quad\text{in }C^0([0,T];L^2(\Omega;\mathbb R^d))\quad\text{as $n\to\infty$}.
\end{align*}
Moreover, the sequence $\{\overline u_n\}_{n\in\mathbb N}\subset L^\infty(0,T;H^1(\Omega;\mathbb R^d))$ is uniformly bounded, and satisfies 
\begin{equation}\label{eq:ph_pun}
\norm{u_n(t)-\overline u_n(t)}_{L^2(\Omega)}\le \tau_n\norm{\dot u_n}_{L^\infty(0,T;L^2(\Omega))}\le C\tau_n\quad\text{for every $t\in[0,T]$ and $n\in\mathbb N$},
\end{equation}
where $C$ is a positive constant independent of $n$ and $t$. Therefore, there exists a further subsequence, not relabeled, such that
\begin{align*}
& \overline u_n\rightharpoonup u\quad\text{in }L^2(0,T;H^1(\Omega;\mathbb R^d)),\quad  \overline u_n\to u\quad\text{in }L^\infty(0,T;L^2(\Omega;\mathbb R^d))\quad\text{as $n\to\infty$}.
\end{align*}
Similarly, we have
\begin{align*}
& \underline u_n\rightharpoonup u\quad\text{in }L^2(0,T;H^1(\Omega;\mathbb R^d)),\quad  \underline u_n\to u\quad\text{in }L^\infty(0,T;L^2(\Omega;\mathbb R^d))\quad\text{as $n\to\infty$}.
\end{align*}
	
Let us now consider the sequence $\{u_n'\}_{n\in\mathbb N}\subset L^\infty(0,T;L^2(\Omega;\mathbb R^d))\cap H^1(0,T;H^{-1}_{D_1}(\Omega;\mathbb R^d))$. Since it is uniformly bounded with respect to $n$, we can apply again Aubin--Lions's lemma to deduce the existence of
\begin{equation*}
z\in L^\infty(0,T;L^2(\Omega;\mathbb R^d))\cap H^1(0,T;H^{-1}_{D_1}(\Omega;\mathbb R^d))
\end{equation*}
such that, up to a further (not relabeled) subsequence
\begin{align*}
&u_n'\rightharpoonup z\quad\text{in }H^1(0,T;H^{-1}_{D_1}(\Omega;\mathbb R^d)),\quad  u_n'\to z\quad\text{in }C^0([0,T];H^{-1}_{D_1}(\Omega;\mathbb R^d))\quad\text{as $n\to\infty$}.
\end{align*}
Furthermore, we have 
\begin{equation}\label{eq:ph_pdotun}
\norm{u_n'(t)-\dot u_n(t)}_{H^{-1}_{D_1}(\Omega)}=\norm{u_n'(t)-\overline u_n'(t)}_{H^{-1}_{D_1}(\Omega)}\le \sqrt{\tau_n}\norm{\dot u_n'}_{L^2(0,T;H^{-1}_{D_1}(\Omega))}\le C\sqrt{\tau_n}
\end{equation}
for every $t\in[0,T]$ and $n\in\mathbb N$, with $C>0$ independent of $n$ and $t$. This fact implies that $z=\dot u$, and
\begin{align*}
\overline u_n'\rightharpoonup \dot u\quad\text{in }L^2(0,T;L^2(\Omega;\mathbb R^d)),\quad \overline u_n'\to \dot u\quad\text{in }L^2(0,T;H^{-1}_{D_1}(\Omega;\mathbb R^d))\quad\text{as $n\to\infty$}.
\end{align*}
In a similar way, we get
\begin{align*}
\underline u_n'\rightharpoonup \dot u\quad\text{in }L^2(0,T;L^2(\Omega;\mathbb R^d)),\quad \underline u_n'\to \dot u\quad\text{in }L^2(0,T;H^{-1}_{D_1}(\Omega;\mathbb R^d))\quad\text{as $n\to\infty$}.
\end{align*}

Finally, the thesis for the sequences $\{v_n\}_{n\in\mathbb N}$, $\{\overline v_n\}_{n\in\mathbb N}$, and $\{\underline v_n\}_{n\in\mathbb N}$ is obtained as before, by using~\eqref{eq:ph_est} and the compactness of the embedding $H^1(\Omega)\hookrightarrow L^2(\Omega)$.
\end{proof} 

\begin{remark}\label{rem:ph_pointcon}
As pointed out in Remark~\ref{rem:ph_weak_con}, we have $u\in C_w^0([0,T];H^1(\Omega;\R^d))$, $\dot u\in C_w^0([0,T];L^2(\Omega;\R^d))$, and $v\in C_w^0([0,T];H^1(\Omega))$. By using the estimate~\eqref{eq:ph_est}, we get
$$\norm{u_n(t)}_{H^1(\Omega)}+\norm{u_n'(t)}_{L^2(\Omega)}\le C\quad\text{for every $t\in[0,T]$ and $n\in\mathbb N$}$$
for a constant $C>0$ independent of $n$ and $t$. Hence,  for every $t\in[0,T]$ we derive
\begin{align*}
&u_n(t)\rightharpoonup u(t)\quad\text{in }H^1(\Omega;\R^d),\quad u_n'(t)\rightharpoonup \dot u(t)\quad\text{in }L^2(\Omega;\R^d)\quad\text{as $n\to\infty$},
\end{align*}
thanks to the previous convergences. In particular, for every $t\in[0,T]$ we can use~\eqref{eq:ph_pun} and~\eqref{eq:ph_pdotun} to obtain
\begin{align*}
&\overline u_n(t)\rightharpoonup u(t)\quad\text{in }H^1(\Omega;\R^d),\quad \overline u_n'(t)\rightharpoonup \dot u(t)\quad\text{in }L^2(\Omega;\R^d)\quad\text{as $n\to\infty$},\\
&\underline u_n(t)\rightharpoonup u(t)\quad\text{in }H^1(\Omega;\R^d),\quad \underline u_n'(t)\rightharpoonup \dot u(t)\quad\text{in }L^2(\Omega;\R^d)\quad\text{as $n\to\infty$}.
\end{align*}
With a similar argument, for every $t\in[0,T]$ we have
\begin{align*}
v_n(t)\rightharpoonup v(t),\quad\overline v_n(t)\rightharpoonup v(t),\quad \underline v_n(t)\rightharpoonup v(t)\quad\text{in }H^1(\Omega)\quad\text{as $n\to\infty$}.
\end{align*}
\end{remark}

We are now in position to pass to the limit in the discrete problem~\eqref{eq:ph_un2}.

\begin{lemma}\label{lem:ph_lim1}
The pair $(u,v)$ of Lemma~\ref{lem:ph_con} is a generalized solution to the problem~\eqref{eq:ph_elsys}--\eqref{eq:ph_bc3}. Moreover, $(u,v)$ satisfies the initial conditions~\eqref{eq:ph_ic} and the irreversibility condition~\eqref{eq:ph_irr}. Finally, if $k=0$, $w_2\ge 0$ on $\partial_{D_2}\Omega$, $v^0\ge 0$ in $\Omega$, and $b(s)=(\max\{s,0\})^2+\eta$ for $s\in\R$, then $v(t)\ge 0$ in $\Omega$ for every $t\in[0,T]$.
\end{lemma}

\begin{proof}
The pair $(u,v)$ given by Lemma~\ref{lem:ph_con} satisfies~\eqref{eq:ph_greg1},~\eqref{eq:ph_greg3}, and the initial conditions~\eqref{eq:ph_ic}, since $u^0=u_n(0)\rightharpoonup u(0)$ in $H^1(\Omega;\mathbb R^d)$, $u^1=u_n'(0)\rightharpoonup \dot u(0)$ in $L^2(\Omega;\mathbb R^d)$, and $v^0=v_n(0)\rightharpoonup v(0)$ in $H^1(\Omega)$ as $n\to\infty$. If we consider the piecewise affine interpolant $w_n$ of $\{w_n^j\}_{j=1}^n$, for every $t\in[0,T]$ we have $u_n(t)- w_n(t)\in H^1_{D_1}(\Omega;\R^d)$ for every $n\in\mathbb N$ and $w_n(t)\to w_1(t)$ in $H^1(\Omega;\R^d)$ as $n\to\infty$. Therefore, the function $u$ satisfies~\eqref{eq:ph_greg2}. Similarly, $v_n(t)-w_2\in H^1_{D_2}(\Omega)$ and $v_n(t)\le v_n(s)\le 1$ in $\Omega$ for every $0\le s\le t\le T$ and $n\in\mathbb N$, which give~\eqref{eq:ph_greg4} and~\eqref{eq:ph_irr}. Finally, if $k=0$, $w_2\ge 0$ on $\partial_{D_2}\Omega$, $v^0\ge 0$ in $\Omega$, and $b(s)=(\max\{s,0\})^2+\eta$ for $s\in\R$, then for every $t\in[0,T]$ we deduce $v_n(t)\ge 0$ in $\Omega$ by Lemma~\ref{lem:ph_dvarin}, which implies $v(t)\ge 0$ in $\Omega$. 

It remains to prove the validity of the equation~\eqref{eq:ph_weak_form} for a.e. $t\in(0,T)$. For every $j=1,\dots,n$ we know that $(u_n^j,v_n^j)$ satisfies~\eqref{eq:ph_un2}. In particular, by integrating it in $[t_1,t_2]\subseteq [0,T]$ and using the previous notation, we derive
\begin{equation}
\begin{aligned}\label{lem:ph_intun}
&\int_{t_1}^{t_2}\spr{\dot u_n'(t)}{\psi}_{H^{-1}_{D_1}(\Omega)}\de t+\int_{t_1}^{t_2}(b(\underline v_n(t))\mathbb C E\overline u_n(t),E\psi)_{L^2(\Omega)}\de t\\
&=\int_{t_1}^{t_2}(\overline f_n(t),\psi)_{L^2(\Omega)}\de t+\int_{t_1}^{t_2}\spr{\overline g_n(t)}{\psi}_{H^{-1}_{D_1}(\Omega)}\de t
\end{aligned}
\end{equation}
for every $\psi\in H^1_{D_1}(\Omega;\mathbb R^d)$, where $\overline f_n$ and $\overline g_n$ are the backward interpolants of $\{f_n^j\}_{j=1}^n$ and $\{g_n^j\}_{j=1}^n$, respectively. We now pass to the limit as $n\to\infty$ in~\eqref{lem:ph_intun}. For the first term we have
\begin{equation*}
\lim_{n\to\infty}\int_{t_1}^{t_2}\spr{\dot u_n'(t)}{\psi}_{H^{-1}_{D_1}(\Omega)}\de t=\int_{t_1}^{t_2}\spr{\ddot u(t)}{\psi}_{H^{-1}_{D_1}(\Omega)}\de t,
\end{equation*}
since $\dot u_n'\rightharpoonup\ddot u$ in $L^2(0,T;H^{-1}_{D_1}(\Omega;\mathbb R^d))$ as $n\to\infty$. Moreover, it is easy to check that $\overline f_n$ converges strongly to $f$ in $L^2(0,T;L^2(\Omega;\mathbb R^d))$, and $\overline g_n$ converges strongly to $g$ in $L^2(0,T;H^{-1}_{D_1}(\Omega;\mathbb R^d))$ as $n\to\infty$, which implies
\begin{equation*}
\lim_{n\to\infty}\left[\int_{t_1}^{t_2}(\overline f_n(t),\psi)_{L^2(\Omega)}\de t+\int_{t_1}^{t_2}\spr{\overline g_n(t)}{\psi}_{H^{-1}_{D_1}(\Omega)}\de t\right]=\int_{t_1}^{t_2}(f(t),\psi)_{L^2(\Omega)}\de t+\int_{t_1}^{t_2}\spr{g(t)}{\psi}_{H^{-1}_{D_1}(\Omega)}\de t.
\end{equation*}
It remains to analyze the second term of~\eqref{lem:ph_intun}. By the previous remark and using the compactness of the embedding $H^1(\Omega)\hookrightarrow L^2(\Omega)$, we get that $\underline v_n(t)\to v(t)$ in $L^2(\Omega)$ as $n\to\infty$ for every $t\in[0,T]$. Thanks to the estimate 
$$\abs{b(\underline v_n(t,x))\mathbb C(x)E\psi(x)}\le b(1)\norm{\mathbb C}_{L^\infty(\Omega)}\abs{E\psi(x)}\quad\text{for every $t\in[0,T]$ and a.e. $x\in\Omega$}$$ 
and the dominated convergence theorem, we conclude that $b(\underline v_n)\mathbb C E\psi\to b(v)\mathbb C E\psi$ in $L^2(0,T;L^2(\Omega;\mathbb R^{d\times d}))$. Hence, we obtain
\begin{align*}
\lim_{n\to\infty}\int_{t_1}^{t_2}(b(\underline v_n(t))\mathbb CE\overline u_n(t),E\psi)_{L^2(\Omega)}\de t&=\int_{t_1}^{t_2}(b(v(t))\mathbb C Eu(t),E\psi)_{L^2(\Omega)}\de t,
\end{align*}
since $E\overline u_n\rightharpoonup Eu$ in $L^2(0,T;L^2(\Omega;\mathbb R^{d\times d}))$. Therefore, the pair $(u,v)$ solves
\begin{equation*}
\int_{t_1}^{t_2}\spr{\ddot u(t)}{\psi}_{H^{-1}_{D_1}(\Omega)}\de t+\int_{t_1}^{t_2}(b(v(t))\mathbb C Eu(t),E\psi)_{L^2(\Omega)}\de t=\int_{t_1}^{t_2}(f(t),\psi)_{L^2(\Omega)}\de t+\int_{t_1}^{t_2}\spr{g(t)}{\psi}_{H^{-1}_{D_1}(\Omega)}\de t
\end{equation*}
for every $\psi\in H^1_{D_1}(\Omega;\mathbb R^d)$ and $[t_1,t_2]\subseteq[0,T]$. Let us choose a countable dense set $\mathscr D\subset H^1_{D_1}(\Omega;\mathbb R^d)$. By Lebesgue's differentiation theorem, we obtain that the pair $(u,v)$ solves~\eqref{eq:ph_weak_form} for a.e. $t\in(0,T)$ and for every $\psi\in \mathscr D$. Finally, we use the density of $\mathscr D$ in $H^1_{D_1}(\Omega;\R^d)$ to conclude that the equation~\eqref{eq:ph_weak_form} is satisfied for every $\psi\in H^1_{D_1}(\Omega;\mathbb R^d)$.
\end{proof}

In the next lemma we exploit the inequality~\eqref{eq:ph_dvarin} to prove~\eqref{eq:ph_mincon}. 

\begin{lemma}\label{lem:ph_lim2}
The pair $(u,v)$ of Lemma~\ref{lem:ph_con} satisfies for a.e. $t\in(0,T)$ the crack stability condition~\eqref{eq:ph_mincon}.
\end{lemma}

\begin{proof}
For every $j=1,\dots,n$ the pair $(u_n^j,v_n^j)$ satisfies the inequality~\eqref{eq:ph_dvarin}, that can be rephrased in
\begin{equation}\label{eq:ph_dd1}
\mathcal E(\overline u_n(t),v^*)-\mathcal E(\overline u_n(t),\overline v_n(t))+\partial\mathcal H(\overline v_n(t))[v^*-\overline v_n(t)]+\sum_{i=0}^k\alpha_i(\nabla^i\dot v_n(t),\nabla^iv^*-\nabla^i\overline v_n(t))_{L^2(\Omega)}\ge 0
\end{equation}
for a.e. $t\in(0,T)$ and for every $v^*-w_2\in H^1_{D_2}(\Omega)\cap H^k(\Omega)$ with $v^*\le \underline v_n(t)$. Given $\chi\in H^1_{D_2}(\Omega)\cap H^k(\Omega)$ with $\chi\le 0$, the function $\chi+\overline v_n(t)$ is admissible for~\eqref{eq:ph_dd1}. After an integration in $[t_1,t_2]\subseteq[0,T]$, we deduce the following inequality
\begin{equation}
\begin{aligned}\label{eq:ph_dd2}
&\int_{t_1}^{t_2}[\mathcal E(\overline u_n(t),\chi+\overline v_n(t))-\mathcal E(\overline u_n(t),\overline v_n(t))]\de t\\
&\quad+\int_{t_1}^{t_2}\partial\mathcal H(\overline v_n(t))[\chi]\de t+\sum_{i=0}^k\alpha_i\int_{t_1}^{t_2}(\nabla^i\dot
v_n(t),\nabla^i\chi)_{L^2(\Omega)}\de t\ge 0.
\end{aligned}
\end{equation}
Let us send $n\to\infty$. We have
\begin{equation}
\lim_{n\to\infty}\sum_{i=0}^k\alpha_i\int_{t_1}^{t_2}(\nabla^i\dot v_n(t),\nabla^i\chi)_{L^2(\Omega)}\de t=\sum_{i=0}^k\alpha_i\int_{t_1}^{t_2}(\nabla^i\dot v(t),\nabla^i\chi)_{L^2(\Omega)}\de t,
\end{equation}
since $\dot v_n\rightharpoonup \dot v$ in $L^2(0,T;H^k(\Omega))$. Moreover $\overline v_n\rightharpoonup v$ in $L^2(0,T;H^1(\Omega))$, which implies
\begin{equation}
\lim_{n\to\infty}\int_{t_1}^{t_2}\partial\mathcal H(\overline v_n(t))[\chi]\de t=\int_{t_1}^{t_2}\partial\mathcal H(v(t))[\chi]\de t.
\end{equation}
The function $\phi(x,y,\xi):=\frac{1}{2}[b(y)-b(\chi(x)+y)]\mathbb C(x) \xi^{sym}\cdot \xi^{sym}$, $(x,y,\xi)\in\Omega\times\mathbb R\times\mathbb R^{d\times d}$, satisfies the assumptions of Ioffe--Olech's theorem (see, e.g.,~\cite[Theorem 3.4]{D}). Thus, for every $t\in[0,T]$ we derive
\begin{align*}
\mathcal E(u(t),v(t))-\mathcal E(u(t),\chi+v(t))&=\int_\Omega \phi(x,v(t,x),Eu(t,x))\de x\\
&\le\liminf_{n\to\infty}\int_\Omega \phi(x,\overline v_n(t,x),E\overline u_n(t,x))\de x\\
&=\liminf_{n\to\infty}[\mathcal E(\overline u_n(t),\chi+\overline v_n(t))-\mathcal E(\overline u_n(t),\overline v_n(t))],
\end{align*}
since $\overline v_n(t)\to v(t)$ in $L^2(\Omega)$ and $E\overline u_n(t)\rightharpoonup Eu(t)$ in $L^2(\Omega;\mathbb R^{d\times d})$ for every $t\in[0,T]$. By Fatou's lemma, we conclude
\begin{align*}
\int_{t_1}^{t_2}[\mathcal E(u(t),v(t))-\mathcal E(u(t),\chi+v(t))]\de t&\le\int_{t_1}^{t_2}\liminf_{n\to\infty}[\mathcal E(\overline u_n(t),\overline v_n(t))-\mathcal E(\overline u_n(t),\chi+\overline v_n(t))]\de t\\
&\le\liminf_{n\to\infty}\int_{t_1}^{t_2}[\mathcal E(\overline u_n(t),\overline v_n(t))-\mathcal E(\overline u_n(t),\chi+\overline v_n(t))]\de t,
\end{align*}
which gives
\begin{equation}\label{eq:ph_ee3}
\int_{t_1}^{t_2}[\mathcal E(u(t),\chi+v(t))-\mathcal E(u(t),v(t))]\de t\ge \limsup_{n\to\infty}\int_{t_1}^{t_2}[\mathcal E(\overline u_n(t),\chi+\overline v_n(t))-\mathcal E(\overline u_n(t),\overline v_n(t))]\de t.
\end{equation}
By combining~\eqref{eq:ph_dd2}--\eqref{eq:ph_ee3} we obtain the following inequality
\begin{equation*}
\int_{t_1}^{t_2}[\mathcal E(u(t),\chi+v(t))-\mathcal E(u(t),v(t))]\de t+\int_{t_1}^{t_2}\partial\mathcal H(v(t))[\chi]\de t+\int_{t_1}^{t_2}\sum_{i=0}^k\alpha_i(\nabla^i\dot v(t),\nabla^i \chi)_{L^2(\Omega)}\de t\ge 0.
\end{equation*}
We choose now a countable dense set $\mathscr D\subset\{\chi\in H^1_{D_2}(\Omega)\cap H^k(\Omega):\chi\le 0\}$. Thanks to Lebesgue's differentiation theorem for a.e. $t\in(0,T)$ we derive
\begin{equation}\label{eq:ph_ee4}
\mathcal E(u(t),\chi+v(t))-\mathcal E(u(t),v(t))+\partial\mathcal H(v(t))[\chi]+\sum_{i=0}^k\alpha_i(\nabla^i\dot v(t),\nabla^i \chi)_{L^2(\Omega)}\ge 0\quad\text{for every }\chi\in\mathscr D.
\end{equation}
Finally, we use a density argument and the dominated convergence theorem to deduce that~\eqref{eq:ph_ee4} is satisfied for every $\chi\in H^1_{D_2}(\Omega)\cap H^k(\Omega)$ with $\chi\le 0$. In particular, for a.e. $t\in(0,T)$ we get
\begin{equation*}
\mathcal E(u(t),v^*)-\mathcal E(u(t),v(t))+\partial\mathcal H(v(t))[v^*-v(t)]+\sum_{i=0}^k\alpha_i(\nabla^i \dot v(t),\nabla^i v^*-\nabla^i v(t))_{L^2(\Omega)}\ge 0,
\end{equation*}
for every $v^*-w_2\in H^1_{D_2}(\Omega)\cap H^k(\Omega)$ with $v^*\le v(t)$, by taking $\chi:=v^*-v(t)$. This implies the crack stability condition~\eqref{eq:ph_mincon}, since the map $v^*\mapsto \mathcal H(v^*)$ is convex.
\end{proof}

We conclude this section by showing that the pair $(u,v)$ of Lemma~\ref{lem:ph_con} satisfies an energy--dissipation inequality. Notice that the total work $\mathcal W_{tot}(u,v;t_1,t_2)$ is well defined also for a generalized solution. Indeed, we have $u\in C_w^0([0,T];H^1(\Omega;\mathbb R^d))$ and $\dot u\in C_w^0([0,T];H^1(\Omega;\mathbb R^d))$, which gives that $u(t)-w_1(t)$ and $\dot u(t)$ are uniquely defined for every $t\in[0,T]$ as elements of $H^1_{D_1}(\Omega;\mathbb R^d)$ and $L^2(\Omega;\mathbb R^d)$, respectively. Moreover, by combining the weak continuity of $u$ and $\dot u$, with the strong continuity of $g$, $w_1$, and $\dot w_1$, it is easy to see that the function $(t_1,t_2)\to\mathcal W_{tot}(t_1,t_2,u,v)$ is continuous.

\begin{lemma}\label{lem:ph_enin}
The pair $(u,v)$ of Lemma~\ref{lem:ph_con} satisfies for every $t\in[0,T]$ the energy--dissipation inequality 
\begin{equation}\label{eq:ph_enin}
\mathcal F(u(t),\dot u(t), v(t))+\int_0^t\mathcal G_k(\dot v(s))\de s\le\mathcal F(u^0, u^1, v^0)+\mathcal W_{tot}(u,v;0,t).
\end{equation}
\end{lemma}

\begin{proof}
Let $g_n$, $w_n$, and $w'_n$ be the piecewise affine interpolants of $\{g_n^j\}_{j=1}^n$, $\{w_n^j\}_{j=1}^n$, and $\{\delta w_n^j\}_{j=1}^n$, respectively, and let $\overline w_n,\overline w_n'$ and $\underline w_n,\underline w_n'$ be the backward and the forward interpolants of $\{w_n^j\}_{j=1}^n$ and $\{\delta w_n^j\}_{j=1}^n$, respectively. 

For $t=0$ the inequality~\eqref{eq:ph_enin} trivially holds thanks to our initial conditions~\eqref{eq:ph_ic}. We fix $t\in (0,T]$ and for every $n\in\mathbb N$ we consider the unique $j\in\{1,\dots,n\}$ such that $t\in((j-1)\tau_n,j\tau_n]$. As done before, we use the previous interpolants and~\eqref{eq:ph_denin} to write
\begin{equation}
\begin{aligned}\label{eq:ph_hh0}
&\mathcal F(\overline u_n(t),\overline u_n'(t),\overline v_n(t))+\int_0^{t_n} \mathcal G_k(\dot v_n(s))\de s\\
&\le \mathcal F(u^0, u^1, v^0)+\int_0^{t_n}(\overline f_n(s),\overline u_n'(s)-\overline w_n'(s))_{L^2(\Omega)}\de s+\int_0^{t_n}(b(\underline v_n(s))\mathbb CE\overline u_n(s),E\overline w_n'(s))_{L^2(\Omega)}\de s\\
&\quad+\spr{\overline g_n(t)}{\overline u_n(t)-\overline w_n(t)}_{H^{-1}_{D_1}(\Omega)}-\spr{g(0)}{u^0-w_1(0)}_{H^{-1}_{D_1}(\Omega)}-\int_0^{t_n}\spr{\dot g_n(s)}{\underline u_n(s)-\underline w_n(s)}_{H^{-1}_{D_1}(\Omega)}\de s\\
&\quad+(\overline u_n'(t),\overline w_n'(t))_{L^2(\Omega)}-(u^1,w_1(0))_{L^2(\Omega)}-\int_0^{t_n}(\underline u_n'(s),\dot w_n'(s))_{L^2(\Omega)}\de s,
\end{aligned}
\end{equation}
where we have set $t_n:=j\tau_n$, and we have neglected the terms $D_n^j$, which are non negative. It easy to see that the following convergences hold as $n\to\infty$:
\begin{align*}
&\overline f_n\to f\quad\text{in }L^2(0,T;L^2(\Omega;\mathbb R^d)),& & \dot g_n\to \dot g\quad\text{in }L^2(0,T;H^{-1}_{D_1}(\Omega;\mathbb R^d)),\\
&\underline w_n\to w_1\quad\text{in }L^2(0,T;H^1(\Omega;\mathbb R^d)),& & \overline w_n'\to \dot w_1\quad\text{in }L^2(0,T;H^1(\Omega;\mathbb R^d)),\\
&\dot w_n'\to \ddot w_1\quad\text{in }H^1(0,T;L^2(\Omega;\mathbb R^d)).
\end{align*}
By using also the ones of Lemma~\ref{lem:ph_con} and observing that $t_n\to t$ as $n\to\infty$, we deduce
\begin{align}\label{eq:ph_hh1}
&\lim_{n\to\infty}\int_0^{t_n}(\overline f_n(s),\overline u_n'(s)-\overline w_n'(s))_{L^2(\Omega)}\de s=\int_0^t(f(s),\dot u(s)-\dot w_1(s))_{L^2(\Omega)}\de s,\\
&\lim_{n\to\infty}\int_0^{t_n}\spr{\dot g_n(s)}{\underline u_n(s)-\underline w_n(s)}_{H^{-1}_{D_1}(\Omega)}\de s=\int_0^t\spr{\dot g(s)}{u(s)-w_1(s)}_{H^{-1}_{D_1}(\Omega)}\de s,\\
&\lim_{n\to\infty}\int_0^{t_n}(\underline u_n'(s),\dot w_n'(s))_{L^2(\Omega)}\de s=\int_0^t(\dot u(s),\ddot w_1(s))_{L^2(\Omega)}\de s.
\end{align}
Moreover, the strong continuity of $g$, $w_1$, and $\dot w_1$ in $H^{-1}_{D_1}(\Omega;\mathbb R^d)$, $H^1(\Omega;\mathbb R^d)$, and $L^2(\Omega;\mathbb R^d)$, respectively, and the convergences of Remark~\ref{rem:ph_pointcon}, imply 
\begin{align}
&\lim_{n\to\infty}\spr{\overline g_n(t)}{\overline u_n(t)-\overline w_n(t)}_{H^{-1}_{D_1}(\Omega)}=\spr{g(t)}{u(t)-w_1(t)}_{H^{-1}_{D_1}(\Omega)},\\
&\lim_{n\to\infty}(\overline u_n'(t),\overline w_n'(t))_{L^2(\Omega)}=(\dot u(t),\dot w_1(t))_{L^2(\Omega)}.
\end{align}
It is easy to check that $b(\underline v_n)\mathbb C E\overline w_n'\to b(v)\mathbb C E\dot w_1$ in $L^2(0,T;L^2(\Omega;\mathbb R^{d\times d}))$, thanks to the dominated convergence theorem. By combining it with $E\overline u_n\rightharpoonup Eu$ in $L^2(0,T;L^2(\Omega;\mathbb R^{d\times d}))$, we conclude
\begin{align}
\lim_{n\to\infty}\int_0^{t_n}(b(\underline v_n(s))\mathbb CE\overline u_n(s),E\overline w_n'(s))_{L^2(\Omega)}\de s=\int_0^t(b(v(s))\mathbb CEu(s),E\dot w_1(s))_{L^2(\Omega)}\de s.
\end{align}

If we now consider the left--hand side of~\eqref{eq:ph_hh0}, we get
\begin{align}
&\mathcal K(\dot u(t))\le\liminf_{n\to\infty}\mathcal K(\overline u_n'(t)),\quad\mathcal H(v(t))\le\liminf_{n\to\infty}\mathcal H(\overline v_n(t)),
\end{align}
since $\overline u_n'(t)\rightharpoonup \dot u(t)$ in $L^2(\Omega,\mathbb R^d)$ and $\overline v_n(t)\rightharpoonup v(t)$ in $H^1(\Omega)$. Furthermore, we have $\dot v_n\rightharpoonup \dot v$ in $L^2(0,T;H^k(\Omega))$ and $t\le t_n$, which gives
\begin{equation}
\int_0^t \mathcal G_k(\dot v(s))\de s\le \liminf_{n\to\infty}\int_0^t \mathcal G_k(\dot v_n(s))\de s\le \liminf_{n\to\infty}\int_0^{t_n} \mathcal G_k(\dot v_n(s))\de s.
\end{equation}
Finally, let us consider the function $\phi(x,y,\xi):=\frac{1}{2}b(y)\mathbb C(x) \xi^{sym}\cdot \xi^{sym}$, $(x,y,\xi)\in\Omega\times\mathbb R\times\mathbb R^{d\times d}$. As in the previous lemma, the function $\phi$ satisfies the assumption of Ioffe--Olech's theorem, while $\overline v_n(t)\to v(t)$ in $L^2(\Omega)$, and $E\overline u_n(t)\rightharpoonup Eu(t)$ in $L^2(\Omega;\mathbb R^{d\times d})$. Thus, we obtain
\begin{equation}\label{eq:ph_hh2}
\begin{aligned}
\mathcal E(u(t),v(t))&=\int_\Omega \phi(x,v(t,x),Eu(t,x))\de x\\
&\le\liminf_{n\to\infty}\int_\Omega \phi(x,\overline v_n(t,x),E\overline u_n(t,x))\de x=\liminf_{n\to\infty}\mathcal E(\overline u_n(t),\overline v_n(t)).
\end{aligned}
\end{equation}
By combining~\eqref{eq:ph_hh0} with~\eqref{eq:ph_hh1}--\eqref{eq:ph_hh2} we deduce the inequality~\eqref{eq:ph_enin} for every $t\in(0,T]$.
\end{proof}

\begin{remark}
Thanks to Lemmas~\ref{lem:ph_lim1}--\ref{lem:ph_enin} for every $k\in\mathbb N\cup\{0\}$ we deduce the existence of a phase--field evolution $(u,v)$ which is a generalized solution to the elastodynamic system~\eqref{eq:ph_elsys} with boundary and initial conditions~\eqref{eq:ph_bc1}--\eqref{eq:ph_ic} and satisfies the irreversibility condition~\eqref{eq:ph_irr}, the minimality condition~\eqref{eq:ph_mincon}, and the energy--dissipation inequality \eqref{eq:ph_enin}. Moreover, for $k=0$, $w_2\ge 0$ on $\partial_{D_2}\Omega$, $v^0\ge 0$ in $\Omega$, and $b(s)=(\max\{s,0\})^2+\eta$ for $s\in\R$, we can construct $(u,v)$ in such a way that $v(t)\ge 0$ in $\Omega$ for every $t\in[0,T]$. Unfortunately, when $k\ge 1$ the argument used in Lemma~\ref{lem:ph_dvarin} to prove the non negativity of the $v_n^j$ fails, therefore we do not have $v(t)\ge 0$ in $\Omega$ for every $t\in[0,T]$.
\end{remark}

%--------------------------
% Proof of the main result
%--------------------------

\section{Proof of the main result}\label{sec:ph4}
In this section we show that for $k>d/2$ the generalized solution $(u,v)$ of Lemma~\ref{lem:ph_con} is a weak solution and satisfies the identity~\eqref{eq:ph_stcon}. To this aim we need several lemmas: we start by proving that, given a function $v\in H^1(0,T;C^0(\overline\Omega))$ satisfying~\eqref{eq:ph_irr}, there exists a unique solution $u$ to the equation~\eqref{eq:ph_weak_form}. As a consequence, we deduce the energy--dissipation balance~\eqref{eq:ph_j1} for {\it every} $t\in[0,T]$, which guarantees that the function $u$ is more regular in time, namely $u\in C^0([0,T];H^1(\Omega;\mathbb R^d))\cap C^1([0,T];L^2(\Omega;\mathbb R^d))$. Finally, we use the crack stability condition~\eqref{eq:ph_mincon} and the energy--dissipation inequality~\eqref{eq:ph_enin} to obtain~\eqref{eq:ph_Genb} from~\eqref{eq:ph_j1}.

\begin{lemma}\label{lem:ph_exis}
Let $w_1$, $f$, $g$, $u^0$, and $u^1$ be as in~\eqref{eq:ph_bd1}--\eqref{eq:ph_id}. Let $\sigma\in H^1(0,T;C^0(\overline\Omega))$ be a function satisfying~\eqref{eq:ph_irr}. Then there exists a unique function $z$ which satisfies~\eqref{eq:ph_greg1},~\eqref{eq:ph_greg2}, the initial conditions $z(0)=u^0$ and $\dot z(0)=u^1$, and which solves for a.e. $t\in(0,T)$ the following equation:
\begin{equation}\label{eq:ph_weak2}
\spr{\ddot z(t)}{\psi}_{H^{-1}_{D_1}(\Omega)}+(b(\sigma(t))\mathbb C Ez(t),E\psi)_{L^2(\Omega)}=(f(t),\psi)_{L^2(\Omega)}+\spr{g(t)}{\psi}_{H^{-1}_{D_1}(\Omega)}
\end{equation}
for every $\psi\in H^1_{D_1}(\Omega;\mathbb  R^d)$.
\end{lemma}

\begin{proof}
To prove the existence of a solution $z$ to~\eqref{eq:ph_weak2}, we proceed as before. We fix $n\in\mathbb N$ and we define 
$$\tau_n:=\frac{T}{n},\quad z^0_n:=u^0,\quad z^{-1}_n:=u^0-\tau_nu^1,\quad\sigma_n^j:=\sigma(j\tau_n)\quad\text{for }j=0,\dots,n.$$ 
For $j=1,\dots,n$ we consider the unique solution $z_n^j-w_n^j\in H^1_{D_1}(\Omega;\mathbb R^d)$ to 
\begin{equation}\label{eq:ph_zn}
(\delta^2z_n^j,\psi)_{L^2(\Omega)}+(b(\sigma_n^{j-1})\mathbb C Ez_n^j,E\psi)_{L^2(\Omega)}=(f_n^j,\psi)_{L^2(\Omega)}+\spr{g_n^j}{\psi}_{H^{-1}_{D_1}(\Omega)}
\end{equation}
for every $\psi\in H^1_{D_1}(\Omega;\mathbb R^d)$, where $\delta z_n^j:=\frac{1}{\tau_n}[z_n^j-z_n^{j-1}]$ for $j=0,\dots,n$, and $\delta^2 z_n^j:=\frac{1}{\tau_n}[\delta z_n^j-\delta z_n^{j-1}]$ for $j=1,\dots,n$. By using $\psi=\tau_n[\delta z_n^j-\delta w_n^j]$ as test function in~\eqref{eq:ph_zn} and proceeding as in Lemma~\ref{lem:ph_denin}, we get that the function $z_n^j$ satisfies for $j=1,\dots,n$ 
\begin{align*}
&[\mathcal K(\delta z_n^j)+\mathcal E(z_n^j, \sigma_n^j)]-[\mathcal K(\delta z_n^{j-1})+\mathcal E(z_n^{j-1},\sigma_n^{j-1})]-\frac{1}{2}([b(\sigma_n^j)-b(\sigma_n^{j-1})]\mathbb C Ez_n^j,Ez_n^j)_{L^2(\Omega)}\\
&\le \tau_n(f_n^j,\delta z_n^j-\delta w_n^j)_{L^2(\Omega)}+\tau_n\spr{g_n^j}{\delta z_n^j-\delta w_n^j}_{H^{-1}_{D_1}(\Omega)}+\tau_n(\delta^2 z_n^j,\delta w_n^j)_{L^2(\Omega)}+\tau_n(b(v_n^{j-1})\mathbb CEz_n^j,E\delta w_n^j)_{L^2(\Omega)}.
\end{align*}
In particular, we can sum over $l=1,\dots, j$ for every $j\in\{1,\dots,n\}$ and use the identities~\eqref{eq:ph_gn} and~\eqref{eq:ph_wn} to derive the discrete energy inequality
\begin{equation}
\begin{aligned}\label{eq:ph_denin2}
&\mathcal K(\delta z^j_n)+\mathcal E(z^j_n, \sigma^j_n)-\frac{1}{2}\sum_{l=1}^j([b(\sigma_n^l)-b(\sigma_n^{l-1})]\mathbb C Ez_n^l,Ez_n^l)_{L^2(\Omega)}\\
&\le \mathcal K(u^1)+\mathcal E(u^0,\sigma(0))+\sum_{l=1}^j \tau_n(f_n^l,\delta z_n^l-\delta w_n^l)_{L^2(\Omega)}+\sum_{l=1}^j \tau_n(b(\sigma_n^{l-1})\mathbb CEz_n^l,E\delta w_n^l)_{L^2(\Omega)}\\
&\quad+\spr{g^j_n}{z^j_n-w^j_n}_{H^{-1}_{D_1}(\Omega)}-\spr{g(0)}{u^0-w_1(0)}_{H^{-1}_{D_1}(\Omega)}-\sum_{l=1}^j\tau_n\spr{\delta g_n^l}{z_n^{l-1}-w_n^{l-1}}_{H^{-1}_{D_1}(\Omega)}\\
&\quad+(\delta z^j_n,\delta w^j_n)_{L^2(\Omega)}-(u^1,w_1(0))_{L^2(\Omega)}-\sum_{l=1}^j\tau_n(\delta z_n^{l-1},\delta^2 w_n^l)_{L^2(\Omega)}.
\end{aligned}
\end{equation}
Since $\sigma_n^j\le \sigma_n^{j-1}$ and $b$ is non decreasing, the last term in the left--hand side is non negative. Hence, by arguing as in Lemma~\ref{lem:ph_bound} and in Remark~\ref{rem:ph_usec}, we can find a constant $C>0$, independent of $n$, such that
\begin{equation*}
\max_{j=1,\dots,n}\left[\norm{\delta z_n^j}_{L^2(\Omega)}+\norm{z_n^j}_{H^1(\Omega)}\right]+\sum_{j=1}^n\tau_n\norm{\delta^2 z_n^j}^2_{H^{-1}_{D_1}(\Omega)}\le C.
\end{equation*}

Let $z_n$, $z_n'$, $\overline z_n$, $\overline z_n'$, $\underline z_n$, and $\underline z_n'$ be the piecewise affine, the backward, and the forward interpolants of $\{z_n^j\}_{j=1}^n$ and $\{\delta z_n^j\}_{j=1}^n$, respectively. As in Lemma~\ref{lem:ph_con}, the previous estimate implies the existence of a subsequence of $n$, not relabeled, and function $z$ satisfying~\eqref{eq:ph_greg1},~\eqref{eq:ph_greg2} and the initial conditions $z(0)=u^0$ and $\dot z(0)=u^1$, such that the following convergences hold as $n\to\infty$:
\begin{align*}
&z_n\rightharpoonup z\quad\text{in }H^1(0,T;L^2(\Omega;\mathbb R^d)),& & z_n'\rightharpoonup \dot z\quad\text{in }H^1(0,T;H^{-1}_{D_1}(\Omega;\mathbb R^d)),\\
&z_n\to z\quad\text{in }C^0([0,T];L^2(\Omega;\mathbb R^d)),& & z_n'\to \dot z\quad\text{in }C^0([0,T];H^{-1}_{D_1}(\Omega;\mathbb R^d)),\\
&\overline z_n,\underline z_n\rightharpoonup z\quad\text{in }L^2(0,T;H^1(\Omega;\mathbb R^d)),& & \overline z_n',\underline z_n'\rightharpoonup \dot z\quad\text{in }L^2(0,T;L^2(\Omega;\mathbb R^d)).
\end{align*}

We now define the backward interpolant $\overline \sigma_n$ and the forward interpolant $\underline \sigma_n$ of $\{\sigma_n^j\}_{j=1}^n$. By integrating the equation~\eqref{eq:ph_zn} in the time interval $[t_1,t_2]\subseteq [0,T]$, we obtain
\begin{equation*}
\int_{t_1}^{t_2}\spr{\dot z_n'(t)}{\psi}_{H^{-1}_{D_1}(\Omega)}\de t+\int_{t_1}^{t_2}(b(\underline \sigma_n(t))\mathbb C E\overline z_n(t),E\psi)_{L^2(\Omega)}\de t=\int_{t_1}^{t_2}(\overline f_n(t),\psi)_{L^2(\Omega)}\de t+\int_{t_1}^{t_2}\spr{\overline g_n(t)}{\psi}_{H^{-1}_{D_1}(\Omega)}\de t
\end{equation*}
for every $\psi\in H^1_{D_1}(\Omega;\mathbb R^d)$. Thanks to the previous convergences and the fact that $\sigma\in H^1(0,T;C^0(\overline\Omega))$, we can pass to the limit as $n\to\infty$ as done in Lemma~\ref{lem:ph_lim1}, and we deduce
\begin{equation*}
\int_{t_1}^{t_2}\spr{\ddot z(t)}{\psi}_{H^{-1}_{D_1}(\Omega)}\de t+\int_{t_1}^{t_2}(b(\sigma(t))\mathbb C Ez(t),E\psi)_{L^2(\Omega)}\de t=\int_{t_1}^{t_2}(f(t),\psi)_{L^2(\Omega)}\de t+\int_{t_1}^{t_2}\spr{g(t)}{\psi}_{H^{-1}_{D_1}(\Omega)}\de t
\end{equation*}
for every $\psi\in H^1_{D_1}(\Omega;\mathbb R^d)$. By Lebesgue's differentiation theorem and a density argument we can conclude that the function $z$ solves~\eqref{eq:ph_weak2} for a.e. $t\in(0,T)$ and for every $\psi\in H^1_{D_1}(\Omega;\R^d)$.

To show the uniqueness result, we adapt a standard technique due to Ladyzenskaya (see~\cite{Lad}). Let $z_1$ and $z_2$ be two solutions to~\eqref{eq:ph_weak2} satisfying~\eqref{eq:ph_greg1},~\eqref{eq:ph_greg2}, and the initial conditions $u^0$ and $u^1$. The function $z=:z_1-z_2$ belongs to the space $L^\infty(0,T;H^1_{D_1}(\Omega;\mathbb R^d))\cap W^{1,\infty} (0,T;L^2(\Omega;\mathbb R^d))\cap H^2(0,T;H^{-1}_{D_1}(\Omega;\mathbb R^d))$, and for a.e. $t\in(0,T)$ solves
\begin{equation*}
\spr{\ddot z(t)}{\psi}_{H^{-1}_{D_1}(\Omega)}+(b(\sigma(t))\mathbb C Ez(t),E\psi)_{L^2(\Omega)}=0\quad\text{for every }\psi\in H^1_{D_1}(\Omega;\mathbb  R^d),
\end{equation*}
with initial conditions $z(0)=\dot z(0)=0$. We fix $s\in(0,T]$, and we consider the function
\begin{equation*}
\varphi(t)=
\begin{cases}
-\int_t^s z(r)\de r\quad&\text{if }t\in[0,s],\\
0\quad&\text{if }t\in[s,T].
\end{cases}
\end{equation*}
Clearly, we have $\varphi\in C^0([0,T];H^1_{D_1}(\Omega;\mathbb R^d))$ and $\varphi(s)=0$. Moreover
\begin{equation*}
\dot \varphi(t)=
\begin{cases}
z(t)\quad&\text{if }t\in[0,s),\\
0\quad&\text{if }t\in(s,T],
\end{cases}
\end{equation*}
which implies $\dot \varphi\in L^\infty(0,T;H^1_{D_1}(\Omega;\mathbb R^d))$. We use $\varphi(t)$ as test function in~\eqref{eq:ph_weak2} and we integrate in $[0,s]$ to deduce
\begin{equation}\label{eq:ph_uniqes}
\int_0^s\spr{\ddot z(t)}{\varphi(t)}_{H^{-1}_{D_1}(\Omega)}\de t+\int_0^s(b(\sigma(t))\mathbb C Ez(t),E\varphi(t))_{L^2(\Omega)}\de t=0.
\end{equation}
By integration by parts, the first term becomes
\begin{equation*}
\int_0^s \spr{\ddot z(t)}{\varphi(t)}_{H^{-1}_{D_1}(\Omega)} \de t=-\int_0^s(\dot z(t),z(t))_{L^2(\Omega)}\de t=-\frac{1}{2}\norm{ z(s)}_{L^2(\Omega)}^2,
\end{equation*}
since $\varphi(s)=\dot z(0)=z(0)=0$. Moreover, the function $t\mapsto (b(\sigma(t))\mathbb CE\varphi(t),E\varphi(t))_{L^2(\Omega)}$ is absolutely continuous on $[0,T]$, because $\varphi\in H^1(0,T;H^1_{D_1}(\Omega;\R^d))$ and $\sigma\in H^1(0,T;C^0(\overline\Omega))$. Hence, we can integrate by parts the second terms of~\eqref{eq:ph_uniqes} to obtain
\begin{align*}
&\int_0^s(b(\sigma(t))\mathbb C E(z(t)),E\varphi(t))_{L^2(\Omega)}\de t\\
&=-\frac{1}{2}\int_0^s (\dot b(\sigma(t))\dot \sigma(t)\mathbb C E\varphi(t),E\varphi(t))_{L^2(\Omega)}\de t-\frac{1}{2}(b(\sigma(0)) \mathbb C E\varphi(0),E\varphi(0))_{L^2(\Omega)},
\end{align*}
since $\varphi(s)=0$. These two identities imply that $z$ and $\varphi$ satisfy
\begin{equation*}
\norm{z(s)}_{L^2(\Omega)}^2+(b(\sigma(0))\mathbb C E\varphi(0),E\varphi(0))_{L^2(\Omega)}=-\int_0^s (\dot b(\sigma(t))\dot \sigma(t)\mathbb C E\varphi(t),E\varphi(t))_{L^2(\Omega)}\de t.
\end{equation*}
In particular, we get
\begin{equation*}
\norm{z(s)}_{L^2(\Omega)}^2+\eta \lambda_0\norm{E\varphi(0)}_{L^2(\Omega)}^2 \le \dot b(\norm{\sigma}_{L^\infty(0,T;C^0(\overline\Omega))}) \norm{\mathbb C}_{L^\infty(\Omega)}\int_0^s\norm{\dot \sigma(t)}_{L^\infty(\Omega)}\norm{E\varphi(t)}_{L^2(\Omega)}^2\de t,
\end{equation*}
since $\dot b$ is non decreasing. Let us define $\zeta(t):=\int_0^t z(r)\de r$ for $t\in [0,s]$. Since $\varphi(t)=\zeta(t)-\zeta(s)$ for $t\in[0,s]$, we deduce that $\norm{ E\varphi(0)}_{L^2(\Omega)}=\norm{E\zeta(s)}_{L^2(\Omega)}$ and
\begin{align*}
\int_0^s\norm{\dot \sigma(t)}_{L^\infty(\Omega)}\norm{ E\varphi(t)}_{L^2(\Omega)}^2\de t&\le 2\norm{E\zeta(s)}_{L^2(\Omega)}^2\int_0^s\norm{ \dot \sigma(t)}_{L^\infty(\Omega)} \de t+2\int_0^s\norm{\dot \sigma(t)}_{L^\infty(\Omega)}\norm{ E\zeta(t)}_{L^2(\Omega)}^2\de t\\
&\le 2\sqrt{s}\norm{ \dot \sigma}_{L^2(0,T;C^0(\overline\Omega))}\norm{E\zeta(s)}_{L^2(\Omega)}^2+2\int_0^s\norm{\dot \sigma(t)}_{L^\infty(\Omega)}\norm{E\zeta(t)}_{L^2(\Omega)}^2\de t.
\end{align*}
Hence, we have
\begin{align*}
&\norm{z(s)}_{L^2(\Omega)}^2+\left[\eta \lambda_0-2 \dot b(\norm{\sigma}_{L^\infty(0,T;C^0(\overline\Omega))}) \norm{\mathbb C}_{L^\infty(\Omega)}\norm{ \dot \sigma}_{L^2(0,T;C^0(\overline\Omega))}\sqrt{s}\right]\norm{E\zeta(s)}_{L^2(\Omega)}^2\\
&\le 2 \dot b(\norm{\sigma}_{L^\infty(0,T;C^0(\overline\Omega))})\norm{\mathbb C}_{L^\infty(\Omega)}\int_0^s\norm{\dot \sigma(t)}_{L^\infty(\Omega)}\norm{E\zeta(t)}_{L^2(\Omega)}^2\de t.
\end{align*}
Let us set
\begin{equation*}
t_0:=\left[\frac{\eta \lambda_0}{4 \dot b(\norm{\sigma}_{L^\infty(0,T;C^0(\overline\Omega)})\norm{\mathbb C}_{L^\infty(\Omega)}\norm{ \dot \sigma}_{L^2(0,T;C^0(\overline\Omega))}}\right]^2.
\end{equation*}
By the previous estimate, for every $s\in[0,t_0]$ we derive
\begin{equation*}
\norm{z(s)}_{L^2(\Omega)}^2+\frac{\eta \lambda_0}{2}\norm{E\zeta(s)}_{L^2(\Omega)}^2\le 2 \dot b(\norm{\sigma}_{L^\infty(0,T;C^0(\overline\Omega))}) \norm{\mathbb C}_{L^\infty(\Omega)}\int_0^s\norm{\dot \sigma(t)}_{L^\infty(\Omega)}\norm{E\zeta(t)}_{L^2(\Omega)}^2\de t.
\end{equation*}
Thanks to Gronwall's lemma (see, e.g.,~\cite[Chapitre XVIII, \S 5, Lemme 1]{DL}), this inequality implies that $z(s)=E\zeta(s)=0$ for every $s\in[0,t_0]$. Since $t_0$ depends only on $\mathbb C$, $b$, and $\sigma$, we can repeat this procedure starting from $t_0$ and, with a finite number of steps, we obtain that $z=0$ on the whole interval $[0,T]$.
\end{proof}

\begin{corollary}\label{coro:ph_enin3}
Let $w_1$, $f$, $g$, $u^0$, $u^1$, and $\sigma$ be as in Lemma~\ref{lem:ph_exis}. Then the unique solution $z$ to~\eqref{eq:ph_weak2} associated to these data satisfies for every $t\in[0,T]$  the following energy--dissipation inequality 
\begin{equation}\label{eq:ph_cenin3}
\begin{aligned}
&\mathcal K(\dot z(t))+\mathcal E(z(t), \sigma(t))-\frac{1}{2}\int_0^t (\dot b(\sigma(s))\dot \sigma(s)\mathbb CEz(s),Ez(s))_{L^2(\Omega)}\de s\\
&\le\mathcal K(u^1)+\mathcal E(u^0, \sigma(0))+\mathcal W_{tot}(z,\sigma;0,t).
\end{aligned}
\end{equation}
\end{corollary}

\begin{proof}
For $t=0$ the inequality~\eqref{eq:ph_cenin3} is trivially true, thanks to the initial conditions of $z$. We fix $t\in(0,T]$ and we write the inequality~\eqref{eq:ph_denin2} as
\begin{equation}
\begin{aligned}\label{eq:ph_sz1}
&\mathcal K(\overline z_n'(t))+\mathcal E(\overline z_n(t), \overline \sigma_n(t))-\frac{1}{2\tau_n}\int_0^{t_n}([b(\overline \sigma_n(s))-b(\underline \sigma_n(s))]\mathbb C E\overline z_n(s),E\overline z_n(s))_{L^2(\Omega)}\de s\\
&\le \mathcal K(u^1)+\mathcal E(u^0, \sigma(0))+\int_0^{t_n}[(\overline f_n(s),\overline z_n'(s)-\overline w_n'(s))_{L^2(\Omega)}+(b(\underline \sigma_n(s))\mathbb CE\overline z_n(s),E\overline w_n'(s))_{L^2(\Omega)}]\de s\\
&\quad+\spr{\overline g_n(t)}{\overline z_n(t)-\overline w_n(t)}_{H^{-1}_{D_1}(\Omega)}-\spr{g(0)}{u^0-w_1(0)}_{H^{-1}_{D_1}(\Omega)}-\int_0^{t_n}\spr{\dot g_n(s)}{\underline z_n(s)-\underline w_n(s)}_{H^{-1}_{D_1}(\Omega)}\de s\\
&\quad+(\overline z_n'(t),\overline w_n'(t))_{L^2(\Omega)}-(u^1,w_1(0))_{L^2(\Omega)}-\int_0^{t_n}(\underline z_n'(s),\dot w_n'(s))_{L^2(\Omega)}\de s,
\end{aligned}
\end{equation}
where $t_n:=j\tau_n$, and $j$ is the unique element in $\{1,\dots,n\}$ for which $t\in((j-1)\tau_n,j\tau_n]$. To pass to the limit as $n\to\infty$ in~\eqref{eq:ph_sz1}, we follow the same procedure adopted in Lemma~\ref{lem:ph_enin}. Notice that $\overline z_n(t)\rightharpoonup z(t)$ in $H^1(\Omega;\R^d)$ and $\overline z_n'(t)\rightharpoonup \dot z(t)$ in $L^2(\Omega;\R^d)$, by arguing as in Remark~\ref{rem:ph_pointcon}, while $\overline \sigma_n(t)\to \sigma(t)$ in $C^0(\overline\Omega)$. Hence, we derive
\begin{align}
&\mathcal K(\dot z(t))\le\liminf_{n\to\infty}\mathcal K(\overline z_n'(t)),\quad\mathcal E(z(t),\sigma(t))\le \liminf_{n\to\infty}\mathcal E(\overline z_n(t),\overline \sigma_n(t)). 
\end{align}
Similarly, we combine the convergences given by the previous lemma, with $\underline \sigma_n(s)\to \sigma(s)$ in $C^0(\overline\Omega)$ for every $s\in[0,T]$ and $t_n\to t$ as $n\to\infty$, to deduce
\begin{align}
&\lim_{n\to\infty}\int_0^{t_n}(\overline f_n(s),\overline z_n'(s)-\overline w_n'(s))_{L^2(\Omega)}\de s=\int_0^t(f(s),\dot z(s)-\dot w_1(s))_{L^2(\Omega)}\de s,\\
&\lim_{n\to\infty}\int_0^{t_n}(b(\underline \sigma_n(s))\mathbb CE\overline z_n(s),E\overline w_n'(s))_{L^2(\Omega)}]\de s=\int_0^t(b(\sigma(s))\mathbb CEz(s),E\dot w(s))_{L^2(\Omega)}\de s,\\
&\lim_{n\to\infty}\int_0^{t_n}(\underline z_n'(s),\dot w_n'(s))_{L^2(\Omega)}\de s=\int_0^t(\dot z(s),\ddot w_1(s))_{L^2(\Omega)}\de s,\\
&\lim_{n\to\infty}\int_0^{t_n}\spr{\dot g_n(s)}{\underline z_n(s)-\underline w_n(s)}_{H^{-1}_{D_1}(\Omega)}\de s=\int_0^t\spr{\dot g(s)}{z(s)-w_1(s)}_{H^{-1}_{D_1}(\Omega)}\de s,\\
&\lim_{n\to\infty}(\overline z_n'(t),\overline w_n'(t))_{L^2(\Omega)}=(\dot z(t),\dot w_1(t))_{L^2(\Omega)},\\
&\lim_{n\to\infty}\spr{\overline g_n(t)}{\overline z_n(t)-\overline w_n(t)}_{H^{-1}_{D_1}(\Omega)}=\spr{g(t)}{z(t)-w_1(t)}_{H^{-1}_{D_1}(\Omega)}.\label{eq:ph_sz4}
\end{align}
Finally, for a.e. $s\in(0,T)$ we have
\begin{equation}\label{eq:ph_dotsigman}
\left\Vert\frac{\overline \sigma_n(s)-\underline \sigma_n(s)}{\tau_n}-\dot \sigma(s)\right\Vert_{L^\infty(\Omega)}\le \frac{1}{\tau_n}\int_{s-\tau_n}^{s+\tau_n}\norm{\dot \sigma(r)-\dot\sigma(s)}_{L^\infty(\Omega)}\to 0\quad\text{as $n\to\infty$},
\end{equation}
since $\dot \sigma\in L^2(0,T;C^0(\overline\Omega))$.
Let us fix $s\in(0,T)$ for which~\eqref{eq:ph_dotsigman} holds. By Lagrange's theorem for every $x\in\Omega$ there exists a point $r_n(s,x)\in[\overline \sigma_n(s,x),\underline \sigma_n(s,x)]$ such that
\begin{equation*}
\frac{b(\overline \sigma_n(s,x))-b(\underline \sigma_n(s,x))}{\tau_n}=\dot b(r_n(s,x))\frac{\overline \sigma_n(s,x)-\underline \sigma_n(s,x)}{\tau_n}.
\end{equation*}
Notice that $r_n(s,x)\to \sigma(s,x)$ as $n\to\infty$ for every $x\in\Omega$. Hence, for a.e. $s\in(0,T)$ we get
\begin{align*}
\lim_{n\to\infty}\frac{b(\overline \sigma_n(s,x))-b(\underline \sigma_n(s,x))}{\tau_n}=\dot b(\sigma(s,x))\dot \sigma(s,x)\quad\text{for every $x\in\Omega$}.
\end{align*}
Furthermore, thanks to~\eqref{eq:ph_dotsigman} there is a constant $C_s>0$, which may depend on $s$, but it is independent of $n$, such that for every $x\in\Omega$
\begin{align*}
\left|\frac{b(\overline \sigma_n(s,x))-b(\underline \sigma_n(s,x))}{\tau_n}\right|\le \dot b(\norm{\sigma}_{L^\infty(0,T;C^0(\overline\Omega))})\left\Vert\frac{\overline \sigma_n(s)-\underline \sigma_n(s)}{\tau_n}\right\Vert_{L^\infty(\Omega)}\le \dot b(\norm{\sigma}_{L^\infty(0,T;C^0(\overline\Omega))})C_s.
\end{align*}
Therefore, for a.e. $s\in(0,T)$ we can apply the dominated convergence theorem to deduce 
\begin{equation*}
\frac{b(\overline \sigma_n(s))-b(\underline \sigma_n(s))}{\tau_n}\to \dot b(\sigma(s))\dot \sigma(s)\quad\text{in $L^2(\Omega)$}\quad\text{as $n\to\infty$}.
\end{equation*}
The function $\phi(x,y,\xi):=\frac{1}{2}\abs{y}\mathbb C(x) \xi^{sym}\cdot \xi^{sym}$, $(x,y,\xi)\in\Omega\times\mathbb R\times\mathbb R^{d\times d}$, satisfies the assumptions of Ioffe--Olech's theorem, while $E\overline z_n(s)\rightharpoonup Ez(s)$ in $L^2(\Omega;\mathbb R^{d\times d})$ for every $s\in[0,T]$. Then, we have
\begin{align*}
-\frac{1}{2}(\dot b(\sigma(s))\dot \sigma(s)\mathbb CEz(s),Ez(s))_{L^2(\Omega)}&=\int_\Omega \phi(x,\dot b(\sigma(s))\dot \sigma(s,x),Ez(s,x))\de x\\
&\le\liminf_{n\to\infty}\int_\Omega \phi\left(x,\frac{b(\overline \sigma_n(s,x))-b(\underline \sigma_n(s,x))}{\tau_n},E\overline z_n(s,x)\right)\de x\\
&=\liminf_{n\to\infty}\left[-\frac{1}{2\tau_n}([b(\overline \sigma_n(s))-b(\underline \sigma_n(s))]\mathbb C E\overline z_n(s),E\overline z_n(s))_{L^2(\Omega)}\right]
\end{align*}
for a.e. $s\in(0,T)$, being $b(\overline\sigma_n(s))\le b(\underline\sigma_n(s))$ in $\Omega$. In particular, thanks to Fatou's lemma we get
\begin{equation}
\begin{aligned}\label{eq:ph_sz3}
&-\frac{1}{2}\int_0^t(\dot b(\sigma(s))\dot \sigma(s)\mathbb CEz(s),Ez(s))_{L^2(\Omega)}\de s\\
&\le\int_0^t\liminf_{n\to\infty}\left[-\frac{1}{2\tau_n}([b(\overline \sigma_n(s))-b(\underline \sigma_n(s))]\mathbb C E\overline z_n(s),E\overline z_n(s))_{L^2(\Omega)}\right]\de s\\
&\le \liminf_{n\to\infty}\left[-\frac{1}{2\tau_n}\int_0^{t_n}([b(\overline \sigma_n(s))-b(\underline \sigma_n(s))]\mathbb C E\overline z_n(s),E\overline z_n(s))_{L^2(\Omega)}\de s\right],
\end{aligned}
\end{equation}
since $t\le t_n$. By combining~\eqref{eq:ph_sz1}--\eqref{eq:ph_sz4} with~\eqref{eq:ph_sz3}  we deduce the inequality~\eqref{eq:ph_cenin3} for every $t\in(0,T]$.
\end{proof}

The other inequality, at least for a.e. $t\in(0,T)$, is a consequence of the equation~\eqref{eq:ph_weak2}.

\begin{lemma}\label{lem:ph_enin2}
Let $w_1$, $f$, $g$, $u^0$, $u^1$, and $\sigma$ be as in Lemma~\ref{lem:ph_exis}. Then the unique solution $z$ to~\eqref{eq:ph_weak2} associated to these data satisfies for a.e. $t\in(0,T)$ 
\begin{equation}\label{eq:ph_cenin4}
\begin{aligned}
&\mathcal K(\dot z(t))+\mathcal E(z(t),\sigma(t))-\frac{1}{2}\int_0^t(\dot b(\sigma(s))\dot \sigma(s)\mathbb CEz(s),Ez(s))_{L^2(\Omega)}\de s\\
&\ge\mathcal K(u^1)+\mathcal E(u^0, \sigma(0))+\mathcal W_{tot}(z,\sigma;0,t).
\end{aligned}
\end{equation}
\end{lemma}

\begin{proof}
It is enough to proceed as done in Lemma~\ref{lem:ph_equiv}, by using Lebesgue's differentiation theorem and exploiting the regularity properties $z\in C_w^0([0,T];H^1(\Omega;\R^d))$ and $\dot z\in C_w^0([0,T];L^2(\Omega;\R^d))$. This ensures that $z$ satisfies
\begin{equation*}
\begin{aligned}
&\mathcal K(\dot z(t_2))+\mathcal E(z(t_2),\sigma(t_2))-\frac{1}{2}\int_{t_1}^{t_2}(\dot b(\sigma(s))\dot \sigma(s)\mathbb CEz(s),Ez(s))_{L^2(\Omega)}\de s\\
&=\mathcal K(\dot z(t_1))+\mathcal E(z(t_1), \sigma(t_1))+\mathcal W_{tot}(z,\sigma;t_1,t_2)
\end{aligned}
\end{equation*}
for a.e. $t_1,t_2\in(0,T)$ with $t_1<t_2$. Since the right--hand side is lower semicontinuous with respect to $t_1$, while the left--hand side is continuous, sending $t_1\to 0^+$ we deduce~\eqref{eq:ph_cenin4}.
\end{proof}

By combining the two previous results we obtain that the solution $z$ to~\eqref{eq:ph_weak2} satisfies an energy--dissipation balance for a.e. $t\in(0,T)$. Actually, this is true for every time, as shown in the following lemma.

\begin{lemma}\label{lem:ph_seq}
Let $w_1$, $f$, $g$, $u^0$, $u^1$, and $\sigma$ be as in Lemma~\ref{lem:ph_exis}. Then the unique solution $z$ to~\eqref{eq:ph_weak2} associated to these data satisfies for every $t\in[0,T]$ the energy--dissipation balance 
\begin{align*}
&\mathcal K(\dot z(t))+\mathcal E(z(t),\sigma(t))-\frac{1}{2}\int_0^t (\dot b(\sigma(s))\dot \sigma(s)\mathbb CEz(s),Ez(s))_{L^2(\Omega)}\de s\\
&=\mathcal K(u^1)+\mathcal E(u^0, \sigma(0))+\mathcal W_{tot}(z,\sigma;0,t).
\end{align*}
In particular, the function $t\mapsto\mathcal K(\dot z(t))+\mathcal E(z(t),\sigma(t))$ is continuous from $[0,T]$ to $\mathbb R$ and
\begin{equation}\label{eq:ph_zreg}
z\in C^0([0,T];H^1(\Omega;\mathbb R^d))\cap C^1([0,T];L^2(\Omega;\mathbb R^d)).
\end{equation}
\end{lemma}

\begin{proof}
We may assume that $\sigma$, $w_1$, $f$, and $g$ are defined on $[0,2T]$ and satisfy the hypotheses of Lemma~\ref{lem:ph_exis} with $T$ replaced by $2T$. As for $w_1$ and $\sigma$, it is enough to set $w_1(t) := 2w_1(T)-w_1(2T-t)$ and $\sigma(t):=\sigma(T)$ for $t\in(T, 2T]$, respectively. By Lemma~\ref{lem:ph_exis}, the solution $z$ on $[0,T]$ can be extended to a solution on $[0, 2T]$ still denoted by $z$. Thanks to Corollary~\ref{coro:ph_enin3} and Lemma~\ref{lem:ph_enin2}, the function $z$ satisfies
\begin{equation}\label{eq:ph_acen}
\begin{aligned}
&\mathcal K(\dot z(t))+\mathcal E(z(t),\sigma(t))-\frac{1}{2}\int_0^t (\dot b(\sigma(s))\dot \sigma(s)\mathbb CEz(s),Ez(s))_{L^2(\Omega)}\de s\\
&=\mathcal K(u^1)+\mathcal E(u^0, \sigma(0))+\mathcal W_{tot}(z,\sigma;0,t)
\end{aligned}
\end{equation}
for a.e. $t\in(0,2T)$, and the inequality~\eqref{eq:ph_cenin3} for every $t\in[0,2T]$. By contradiction assume the existence of a point $t_0\in[0,T]$ such that
\begin{align*}
&\mathcal K(\dot z(t_0))+\mathcal E(z(t_0),\sigma(t_0))-\frac{1}{2}\int_0^{t_0}(\dot b(\sigma(s))\dot \sigma(s)\mathbb CEz(s),Ez(s))_{L^2(\Omega)}\de s\\
&<\mathcal K(u^1)+\mathcal E(u^0, \sigma(0))+\mathcal W_{tot}(z,\sigma;0,t_0).
\end{align*}
Since $z\in C_w^0([0,T];H^1(\Omega;\mathbb R^d))$ and $\dot z\in C_w^0([0,T];L^2(\Omega;\mathbb R^d))$, we have that $z(t_0)-w(t_0)\in H^1_{D_1}(\Omega;\mathbb R^d)$ and $\dot z(t_0)\in L^2(\Omega;\mathbb R^d)$. Then we can consider the solution $z_0$ to~\eqref{eq:ph_weak2} in $[t_0,2T]$ with these initial conditions. The function defined by $z$ in $[0,t_0]$ and $z_0$ in $[t_0,2T]$ is still a solution to~\eqref{eq:ph_weak2} in $[0,2T]$ and so, by uniqueness, we have $z=z_0$ in $[t_0,2T]$. Furthermore, in view of~\eqref{eq:ph_cenin3} we deduce 
\begin{align*}
&\mathcal K(\dot z(t))+\mathcal E(z(t),\sigma(t))-\frac{1}{2}\int_{t_0}^t (\dot b(\sigma(s))\dot \sigma(s)\mathbb CEz(s),Ez(s))_{L^2(\Omega)}\de s\\
&\le\mathcal K(z(t_0))+\mathcal E(z(t_0), \sigma(t_0))+\mathcal W_{tot}(z,\sigma;t_0,t)
\end{align*}
for every $t\in[t_0,2T]$. By combining the last two inequalities, we get
\begin{align*}
&\mathcal K(\dot z(t))+\mathcal E(z(t),\sigma(t))-\frac{1}{2}\int_0^t (\dot b(\sigma(s))\dot \sigma(s)\mathbb CEz(s),Ez(s))_{L^2(\Omega)}\de s\\
&\le\mathcal K(z(t_0))+\mathcal E(z(t_0), \sigma(t_0))+\mathcal W_{tot}(z,\sigma;t_0,t)-\frac{1}{2}\int_0^{t_0}(\dot b(\sigma(s))\dot \sigma(s)\mathbb CEz(s),Ez(s))_{L^2(\Omega)}\de s\\
&<\mathcal K(u^1)+\mathcal E(u^0, \sigma(0))+\mathcal W_{tot}(z,\sigma;0,t_0)+\mathcal W_{tot}(z,\sigma;t_0,t)\\
&=\mathcal K(u^1)+\mathcal E(u^0, \sigma(0))+\mathcal W_{tot}(z,\sigma;0,t)
\end{align*}
for every $t\in[t_0,2T]$, which contradicts~\eqref{eq:ph_acen}. Therefore, the equality~\eqref{eq:ph_acen} holds for every $t\in[0,T]$, which implies the continuity of the map $t\mapsto\mathcal K(\dot z(t))+\mathcal E(z(t),\sigma(t))$ from $[0,T]$ to $\mathbb R$.

Let us now prove~\eqref{eq:ph_zreg}. We fix $t_0\in[0,T]$ and we consider a sequence of points $\{t_m\}_{m\in\mathbb N}$ converging to $t_0$ as $m\to\infty$. Since $z\in C_w^0([0,T];H^1(\Omega;\mathbb R^d))$ and $\dot z	\in C_w^0([0,T];L^2(\Omega;\mathbb R^d))$, we have
\begin{align*}
&\mathcal K(\dot z(t_0))\le \liminf_{m\to\infty}\mathcal K(\dot z(t_m)),\quad\mathcal E(z(t_0),\sigma(t_0))\le\liminf_{m\to\infty}\mathcal E(z(t_m),\sigma(t_0)).
\end{align*}
Moreover, $\sigma\in C^0([0,T];C^0(\overline\Omega))$ and $b\in C^1(\R)$, which implies as $m\to\infty$
\begin{align*}
&\left|\mathcal E(z(t_m),\sigma(t_0))-\mathcal E(z(t_m),\sigma(t_m))\right|\\
&\le \frac{1}{2}\dot b(\norm{\sigma}_{L^\infty(0,T;C^0(\overline\Omega))})\norm{\mathbb C}_{L^\infty(\Omega)}\norm{E z}^2_{L^\infty(0,T;L^2(\Omega))}\norm{ \sigma(t_0)-\sigma(t_m)}_{L^\infty(\Omega)}\to 0.
\end{align*}
In particular, we deduce
\begin{align*}
&\mathcal E(z(t_0),\sigma(t_0))\le\liminf_{m\to\infty}\mathcal E(z(t_m),\sigma(t_m)).
\end{align*}
The above inequalities and the continuity of $t\mapsto\mathcal K(\dot z(t))+\mathcal E(z(t),\sigma(t))$ gives
\begin{align*}
\mathcal K(\dot z(t_0))+\mathcal E(z(t_0),\sigma(t_0))&\le\liminf_{m\to\infty}\mathcal K(\dot z(t_m))+\liminf_{m\to\infty}\mathcal E(z(t_m),\sigma(t_m))\\
&\le\lim_{m\to\infty}[\mathcal K(\dot z(t_m))+\mathcal E(z(t_m),\sigma(t_m))]=\mathcal K(\dot z(t_0))+\mathcal E(z(t_0),\sigma(t_0)),
\end{align*}
which implies the continuity of $t\mapsto \mathcal K(\dot z(t))$ and $t\mapsto \mathcal E(z(t),\sigma(t))$ in $t_0\in[0,T]$. In particular, we derive that the functions $t\mapsto \norm{\dot z(t)}_{L^2(\Omega)}$ and $t\mapsto \norm{z(t)}_{H^1(\Omega)}$ are continuous from $[0,T]$ to $\mathbb R$. By combining this fact with the weak continuity of $\dot z$ and $z$, we get~\eqref{eq:ph_zreg}.
\end{proof}

We are now in a position to prove Theorem~\ref{thm:ph:main_res}.

\begin{proof}[Proof of Theorem~\ref{thm:ph:main_res}]
By Lemmas~\ref{lem:ph_lim1} and~\ref{lem:ph_lim2}, there exists a generalized solution $(u,v)$ to~\eqref{eq:ph_elsys}--\eqref{eq:ph_bc3} satisfying the initial conditions~\eqref{eq:ph_ic}, the irreversibility condition~\eqref{eq:ph_irr}, and the unilateral crack stability condition~\eqref{eq:ph_mincon}. Clearly, the function $v$ satisfies~\eqref{eq:ph_reg3}, since $k\ge 1$. Moreover, the function $v=\sigma$ is admissible in Lemmas~\ref{lem:ph_exis} and~\ref{lem:ph_seq}, since $H^k(\Omega)\hookrightarrow C^0(\overline\Omega)$. Therefore, by uniqueness we have that $u=z$ satisfies~\eqref{eq:ph_zreg}, which gives that $(u,v)$ is a weak solution to~\eqref{eq:ph_elsys}--\eqref{eq:ph_bc3}.

It remains to prove that $(u,v)$ satisfies the Griffith's dynamic energy--dissipation balance~\eqref{eq:ph_Genb}. As observed in Remark~\ref{rem:ph_varin}, for $k>d/2$ the crack stability condition~\eqref{eq:ph_mincon} is equivalent to the variational inequality~\eqref{eq:ph_varin} and the function $\dot v(t)\in H^k(\Omega)$ is admissible in~\eqref{eq:ph_varin} for a.e. $t\in(0,T)$. Therefore, we have 
\begin{equation*}
\partial_v\mathcal E(u(t),v(t))[\dot v(t)]+\partial\mathcal H(v(t))[\dot v(t)]+\mathcal G_k(\dot v(t))\ge 0\quad\text{for a.e. }t\in(0,T).
\end{equation*}
By integrating the above inequality in $[0,t_0]$ for every $t_0\in[0,T]$, we get 
\begin{equation}\label{eq:ph_j0}
\int_0^{t_0}\partial_v\mathcal E(u(t),v(t))[\dot v(t)]\de t+\mathcal H(v(t_0))-\mathcal H(v^0)+\int_0^{t_0}\mathcal G_k(\dot v(t))\de t\ge 0.
\end{equation}
Thanks to Lemma~\ref{lem:ph_seq}, for every $t_0\in[0,T]$ the pair $(u,v)$ satisfies the energy--dissipation balance 
\begin{equation}\label{eq:ph_j1}
\begin{aligned}
&\mathcal K(\dot u(t_0))+\mathcal E(u(t_0),v(t_0))-\frac{1}{2}\int_0^{t_0} (\dot b(v(t))\dot v(t)\mathbb CEu(t),Eu(t))_{L^2(\Omega)}\de t\\
&=\mathcal K(u^1)+\mathcal E(u^0, v^0)+\mathcal W_{tot}(u,v;0,t_0).
\end{aligned}
\end{equation}
Hence, by combining~\eqref{eq:ph_j0} and~\eqref{eq:ph_j1}, we deduce
\begin{align*}
&\mathcal F(u(t_0),\dot u(t_0),v(t_0))+\int_0^{t_0}\mathcal G_k(\dot v(t))\de t\ge\mathcal F(u^0,u^1,v^0)+\mathcal W_{tot}(u,v;0,t_0)
\end{align*}
for every $t_0\in[0,T]$. This inequality, together with~\eqref{eq:ph_enin}, implies~\eqref{eq:ph_Genb} and concludes the proof.
\end{proof}

%------------------------------------------
% The case without dissipative terms
%------------------------------------------

\section{The case without dissipative terms}\label{sec:ph5}

We conclude the paper by analyzing the dynamic phase--field model of crack propagation without dissipative terms. Given $w_1$, $w_2$, $f$, $g$, $u^0$, $u^1$, and $v^0$ satisfying~\eqref{eq:ph_bd1}--\eqref{eq:ph_id} and
\begin{align}
&v^0\in\argmin\{\mathcal E(u^0,v^*)+\mathcal H(v^*):v^*-w_2\in H^1_{D_2}(\Omega),\,v^*\le v^0\text{ in $\Omega$}\},\label{eq:ph_id2}
\end{align}
we search a pair $(u,v)$ which solves the elastodynamics system~\eqref{eq:ph_elsys} with boundary and initial conditions~\eqref{eq:ph_bc1}--\eqref{eq:ph_ic}, the {\it irreversibility} condition~\eqref{eq:ph_irr}, and for every $t\in[0,T]$ the following {\it crack stability condition}
\begin{equation}\label{eq:ph_mincon2}
\mathcal E(u(t),v(t))+\mathcal H(v(t))\le \mathcal E(u(t),v^*)+\mathcal H(v^*)
\end{equation}
among all $v^*-w_2\in H^1_{D_2}(\Omega)$ with $v^*\le v(t)$. 

\begin{remark}
In the dynamic phase--field model $(\tilde D_1)$--$(\tilde D_3)$ we require the crack stability condition~\eqref{eq:ph_mincon} for a.e. $t\in(0,T)$ since it appears $\dot v$, which is defined only a.e. in $(0,T)$. In the case without dissipative terms we can instead construct a dynamic phase--field evolution which is defined pointwise everywhere in $[0,T]$ and satisfies~\eqref{eq:ph_mincon} for every $t\in[0,T]$. Notice that we need the compatibility conditions~\eqref{eq:ph_id2} for the initial data $(u^0,v^0)$, if we want that~\eqref{eq:ph_mincon2} is satisfied for $t=0$. Given $u^0\in H^1(\Omega;\R^d)$, an admissible $v^0$ can be constructed by minimizing $v^*\mapsto \mathcal E(u^0,v^*)+\mathcal H(v^*)$ among all $v^*-w_2\in H^1_{D_2}(\Omega)$ with $v^*\le 1$ in $\Omega$.
\end{remark}

In this section we consider the following notion of solution, which is a slightly modification of Definition~\ref{def:ph_gensol}. 

\begin{definition}
Let $w_1$, $w_2$, $f$, and $g$ be as in~\eqref{eq:ph_bd1} and~\eqref{eq:ph_bd2}. The pair $(u,v)$ is a {\it generalized solution} to~\eqref{eq:ph_elsys}--\eqref{eq:ph_bc3} if 
\begin{align}
&u\in L^\infty(0,T;H^1(\Omega;\mathbb R^d))\cap W^{1,\infty}(0,T;L^2(\Omega;\mathbb R^d))\cap H^2(0,T;H^{-1}_{D_1}(\Omega;\mathbb R^d)),\label{eq:ph_greg12}\\
&u(t)-w_1(t)\in H^1_{D_1}(\Omega;\mathbb R^d)\text{ for every }t\in[0,T],\\
&v\colon [0,T]\to H^1(\Omega)\text{ with }v\in L^\infty(0,T;H^1(\Omega)),\\
&v(t)-w_2\in H^1_{D_2}(\Omega)\text{ and }v(t)\le 1\text{ in }\Omega\text{ for every }t\in[0,T],\label{eq:ph_greg42}
\end{align}
and for a.e. $t\in(0,T)$ we have~\eqref{eq:ph_weak_form}.
\end{definition}

\begin{remark}
By exploiting~\eqref{eq:ph_greg12}, we deduce that $u\in C_w^0([0,T];H^1(\Omega;\R^d))$, while $\dot u\in C_w^0([0,T];L^2(\Omega;\R^d))$. Therefore, it makes sense to evaluate $u$ and $\dot u$ at time $0$. On the other hand, the function $v$ is defined pointwise for every $t\in[0,T]$, and in the initial condition~\eqref{eq:ph_ic} we consider its precise value at $0$.
\end{remark}

The main result of this section is the following theorem.

\begin{theorem}\label{thm:ph_mainres2}
Assume that $w_1$, $w_2$, $f$, $g$, $u^0$, $u^1$, and $v^0$ satisfy~\eqref{eq:ph_bd1}--\eqref{eq:ph_id} and~\eqref{eq:ph_id2}. Then there exists a generalized solution $(u,v)$ to problem~\eqref{eq:ph_elsys}--\eqref{eq:ph_bc3} satisfying the initial condition~\eqref{eq:ph_ic}, the irreversibility condition~\eqref{eq:ph_irr}, and the crack stability condition~\eqref{eq:ph_mincon2}. Moreover, the pair $(u,v)$ satisfies for every $t\in[0,T]$ the following energy--dissipation inequality
\begin{equation}\label{eq:ph_enin2}
\mathcal F(u(t),\dot u(t),v(t))\le\mathcal F(u^0,u^1,v^0)+\mathcal W_{tot}(u,v;0,t).
\end{equation}
Finally, if $w_2\ge 0$ on $\partial_{D_2}\Omega$, $v^0\ge 0$ in $\Omega$, and $b(s)=(\max\{s,0\})^2+\eta$ for $s\in\R$, then we can take $v(t)\ge 0$ in $\Omega$ for every $t\in [0,T]$.
\end{theorem}

\begin{remark}
By choosing $b(s)=(\max\{s,0\})^2+\eta$ for $s\in\R$ we deduce the existence of a dynamic phase--field evolution $(u,v)$ satisfying $(D_1)$ and $(D_2)$, since we can take $v(t)\ge 0$ in $\Omega$ and 
\begin{equation*}
    \int_\Omega [(\max\{v(x),0\})^2+\eta]\mathbb C(x)Eu(x)\cdot Eu(x)\de x\le \int_\Omega [(v(x))^2+\eta]\mathbb C(x)Eu(x)\cdot Eu(x)\de x
\end{equation*}
for every $u\in H^1(\Omega;\R^d)$ and $v\in H^1(\Omega)$.
Without adding a dissipative term to the model, we are not able to show the dynamic energy--dissipation balance $(D_3)$. However, we can always select a solution $(u,v)$ which satisfies~\eqref{eq:ph_enin2} for every $t\in[0,T]$.
\end{remark}

To prove Theorem~\ref{thm:ph_mainres2} we perform a time discretization, as done in the previous sections. From now on we assume that $w_1$, $w_2$, $f$, $g$, $u^0$, $u^1$, and $v^0$ satisfy~\eqref{eq:ph_bd1}--\eqref{eq:ph_id} and~\eqref{eq:ph_id2}. We fix $n\in\mathbb N$ and for every $j=1,\dots,n$ we define inductively:
\begin{itemize}
\item[$(i)$] $u_n^j-w_n^j\in H^1_{D_1}(\Omega;\R^d)$ is the minimizer of
\begin{align*}
u^*\mapsto\frac{1}{2\tau_n^2}\left\Vert u^*-2u_n^{j-1}-u_n^{j-2}\right\Vert_{L^2(\Omega)}^2+\mathcal E(u^*,v_n^{j-1})-(f_n^j,u^*)_{L^2(\Omega)}-\spr{g_n^j}{u^*-w_n^j}_{H^{-1}_{D_1}(\Omega)}
\end{align*}
among every $u^*-w_n^j\in H^1_{D_1}(\Omega;\R^d)$;
\item[$(ii)$] $v_n^j-w_2\in H^1_{D_2}(\Omega)$ with $v_n^j\le v_n^{j-1}$ is the minimizer of
\begin{align*}
&v^*\mapsto\mathcal E(u_n^j,v^*)+\mathcal H(v^*)
\end{align*}
among every $v^*-w_2\in H^1_{D_2}(\Omega)$ with $v^*\le v_n^{j-1}$.
\end{itemize}

As observed before, for every $j=1,\dots,n$ there exists a unique pair $(u_n^j,v_n^j)\in H^1(\Omega;\mathbb R^d)\times H^1(\Omega)$ solution to problems $(i)$ and $(ii)$. Moreover, the function $u_n^j$ solves~\eqref{eq:ph_un2}, while the function $v_n^j$ satisfies 
\begin{equation}\label{eq:ph_varin2}
\mathcal E(u_n^j,v^*)-\mathcal E(u_n^j,v_n^j)+\partial \mathcal H(v_n^j)[v^*- v_n^j]\ge 0
\end{equation}
among all $v^*-w_2\in H^1_{D_2}(\Omega)$ with $v^*\le v_n^{j-1}$, by arguing as in Lemma~\ref{lem:ph_dvarin}. In particular, if $w_2\ge 0$ on $\partial_{D_2}\Omega$, $v^0\ge 0$ in $\Omega$, and $b(s)=(\max\{s,0\})^2+\eta$ for $s\in\R$, then for every $j=1,\dots,n$ we can use $(v_n^j)^+:=\max\{v_n^j,0\}\in H^1(\Omega)$ as a competitor in $(ii)$ to derive that $v_n^j=(v_n^j)^+\ge 0$ in $\Omega$.

\begin{lemma}\label{lem:ph_bound2}
The family $\{(u_n^j,v_n^j)\}_{j=1}^n$, solution to problems $(i)$ and $(ii)$, satisfies for $j=1,\dots,n$ the discrete energy inequality
\begin{align*}
\mathcal F(u^j_n, \delta u^j_n, v^j_n)+\sum_{l=1}^j \tau_n^2D_n^l&\le \mathcal F(u^0, u^1, v^0)+\sum_{l=1}^j \tau_n\left[(f_n^l,\delta u_n^l-\delta w_n^l)_{L^2(\Omega)}+(b(v_n^{l-1})\mathbb CEu_n^l,E\delta w_n^l)_{L^2(\Omega)}\right]\\
&\quad-\sum_{l=1}^j \tau_n\left[(\delta u_n^{l-1},\delta^2 w_n^l)_{L^2(\Omega)}-\spr{\delta g_n^l}{u_n^{l-1}-w^{i-1}_n}_{H^{-1}_{D_1}(\Omega)}\right]+(\delta u^j_n,\delta w^j_n)_{L^2(\Omega)}\\
&\quad+\spr{g^j_n}{u^j_n-w^j_n}_{H^{-1}_{D_1}(\Omega)}-(u^1,\dot w_1(0))_{L^2(\Omega)}-\spr{g(0)}{u^0-w_1(0)}_{H^{-1}_{D_1}(\Omega)}.
\end{align*}
In particular, there exists a constant $C>0$, independent of $n$, such that
\begin{equation}\label{eq:ph_est2}
\max_{j=1,\dots,n}\left[\norm{\delta u_n^j}_{L^2(\Omega)}+\norm{u_n^j}_{H^1(\Omega)}+\norm{ v_n^j}_{H^1(\Omega)}\right]+\sum_{j=1}^n\tau_n\norm{\delta^2u_n^j}^2_{H^{-1}_{D_1}(\Omega)}+\sum_{j=1}^n\tau_n^2 D_n^j\le C.
\end{equation}
\end{lemma}

\begin{proof}
It is enough to proceed as in Lemmas~\ref{lem:ph_denin} and~\ref{lem:ph_bound}, and Remark~\ref{rem:ph_usec}.
\end{proof}

As done in Section~\ref{sec:ph3}, we use the family $\{(u_n^j,v_n^j)\}_{j=1}^n$ and the estimate~\eqref{eq:ph_est2} to construct a generalized solution $(u,v)$ to~\eqref{eq:ph_elsys}--\eqref{eq:ph_bc3}. Let $u_n$, $u_n'$, $\overline u_n$, $\overline u_n'$, $\underline u_n$ and $\underline u_n'$ be, respectively, the piecewise affine, the backward, and the forward interpolants of $\{u_n^j\}_{j=1}^n$ and $\{\delta u_n^j\}_{j=1}^n$. Moreover, we consider the backward and the forward interpolants $\overline v_n$ and $\underline v_n$ of $\{v_n^j\}_{j=1}^n$, respectively.

Before passing to the limit as $n\to\infty$, we recall the following Helly's type result for vector--valued functions.

\begin{lemma}\label{lem:ph_Helly}
Let $[a,b]\subset\R$ and let $\varphi_m\colon[a,b]\to L^2(\Omega)$, $m\in\mathbb N$, be a sequence of functions satisfying
\begin{equation*}
\varphi_m(s)\le \varphi_m(t)\quad\text{in $\Omega$}\quad\text{for every $a\le s\le t\le b$ and $m\in\mathbb N$}.
\end{equation*}
Assume there exists a constant $C$, independent of $m$, such that
\begin{equation*}
\norm{\varphi_m(t)}_{L^2(\Omega)}\le C\quad\text{for every $t\in[a,b]$ and $m\in\mathbb N$}.
\end{equation*}
Then there is a subsequence of $m$, not relabeled, and a function $\varphi\colon[a,b]\to L^2(\Omega)$ such that for every $t\in[a,b]$
\begin{equation*}
\varphi_m(t)\rightharpoonup \varphi(t)\quad\text{in $L^2(\Omega)$}\quad\text{as $m\to\infty$}.
\end{equation*}
Moreover, we have $\norm{\varphi(t)}_{L^2(\Omega)}\le C$ for every $t\in[a,b]$ and
\begin{equation}\label{eq:ph_irr2}
\varphi(s)\le \varphi(t)\quad\text{in $\Omega$}\quad\text{for every $a\le s\le t\le b$}.
\end{equation}
\end{lemma}

\begin{proof}
Let us consider a countable dense set $\mathscr D\subset \{\chi\in L^2(\Omega):\chi\ge 0\}$ and let us fix $\chi\in \mathscr D$. For every $m\in\mathbb N$ the map $t\mapsto \int_\Omega \varphi_m(t,x)\chi(x)\de x$
is non decreasing and uniformly bounded in $[a,b]$, since
\begin{equation}\label{eq:ph_varphiest}
\left|\int_\Omega \varphi_m(t,x)\chi(x)\de x\right|\le C\norm{\chi}_{L^2(\Omega)}\quad\text{for every $t\in[a,b]$}.
\end{equation}
By applying Helly's theorem, we can find a subsequence of $m$, not relabeled, and a function $a_\chi\colon [a,b]\to \R$ such that for every $t\in[a,b]$
\begin{equation*}
\int_\Omega \varphi_m(t,x)\chi(x)\de x\to a_\chi(t)\quad\text{as } m\to\infty.
\end{equation*}
Moreover, thanks to a diagonal argument, the subsequence of $m$ can be chosen independent of $\chi\in\mathcal D$.

We now fix $\chi\in L^2(\Omega)$ with $\chi\ge 0$ and $t\in[a,b]$. Given $h>0$, there is $\chi_h\in\mathscr D$ such that $\norm{\chi-\chi_h}_{L^2(\Omega)}<h$ and, thanks to the previous convergence, we can find $\bar m\in\mathbb N$ such that for every $m,l>\bar m$
\begin{equation*}
\left|\int_\Omega (\varphi_m(t,x)-\varphi_j(t,x))\chi_h(x)\de x\right|<h.
\end{equation*}
Therefore, the sequence $\int_\Omega \varphi_m(t,x)\chi(x)\de x$, $m\in\mathbb N$, is Cauchy in $\R$. Indeed, for every $h>0$ there exists $\bar m\in\mathbb N$ such that for every $m,l>\bar m$
\begin{align*}
\left|\int_\Omega \varphi_m(t,x)\chi(x)\de x-\int_\Omega \varphi_l(t,x)\chi(x)\de x\right|&\le 2C\norm{\chi-\chi_h}_{L^2(\Omega)}+\left|\int_\Omega (\varphi_m(t,x)-\varphi_l(t,x))\chi_h(x)\de x\right|\\
&<(2C+1)h.
\end{align*}
Hence, we can find an element $a_\chi(t)\in\R$ such that 
\begin{equation*}
\int_\Omega \varphi_m(t,x)\chi(x)\de x\to a_\chi(t)\quad\text{as $m\to\infty$}.
\end{equation*}
In particular, for every $t\in[a,b]$ and $\chi\in L^2(\Omega)$ we have 
\begin{align*}
\int_\Omega \varphi_m(t,x)\chi(x)\de x&=\int_\Omega \varphi_m(t,x)\chi^+(x)\de x-\int_\Omega \varphi_m(t,x)\chi^-(x)\de x\to a_{\chi^+}(t)-a_{\chi^-}(t)=:a_\chi(t)\quad\text{as $m\to\infty$},
\end{align*}
where we have set $\chi^+:=\max\{\chi,0\}$ and $\chi^-:=\max\{-\chi,0\}$. For every $t\in[a,b]$ fixed, let us consider the functional $\zeta(t)\colon L^2(\Omega)\to \R$ defined by 
$$\zeta(t)(\chi):=a_\chi(t)\quad\text{for $\chi\in L^2(\Omega)$}.$$
We have that $\zeta(t)$ linear and continuous on $L^2(\Omega)$. Indeed, by~\eqref{eq:ph_varphiest} we deduce
\begin{equation*}
|\zeta(t)(\chi)|\le C\norm{\chi}_{L^2(\Omega)}\quad\text{for every $\chi\in L^2(\Omega)$}.
\end{equation*}
Hence, Riesz's representation theorem implies the existence of a function $\varphi(t)\in L^2(\Omega)$ such that 
\begin{equation*}
a_\chi(t)=\int_\Omega \varphi(t,x)\chi(x)\de x\quad\text{for every $\chi\in L^2(\Omega)$}.
\end{equation*}
In particular, for every $t\in[a,b]$ we have $\varphi_m(t)\rightharpoonup \varphi(t)$ in $L^2(\Omega)$ as $m\to\infty$ and $\norm{\varphi(t)}_{L^2(\Omega)}\le C$. Finally observe that $\{\chi\in L^2(\Omega):\chi\ge 0\}$ is a weakly closed subset of $L^2(\Omega)$. Therefore, we derive~\eqref{eq:ph_irr2}, since $\varphi_m(t)-\varphi_m(s)\rightharpoonup \varphi(t)-\varphi(s)$ in $L^2(\Omega)$ as $m\to \infty$ and $\varphi_m(t)-\varphi_m(s)\in \{\chi\in L^2(\Omega):\chi\ge 0\}$ for every $m\in\mathbb N$ and $a\le s\le t\le b$.
\end{proof}

\begin{lemma}\label{lem:ph_con2}
There exist a subsequence of $n$, not relabeled, and two functions
\begin{align*}
&u\in L^\infty(0,T;H^1(\Omega;\R^d))\cap W^{1,\infty}(0,T;L^2(\Omega;\R^d))\cap H^2(0,T;H^{-1}_{D_1}(\Omega;\R^d)),\\
&v\colon [0,T]\to H^1(\Omega)\text{ with }v\in L^\infty(0,T;H^1(\Omega)),
\end{align*}
such that as $n\to\infty$
\begin{align*}
&u_n\rightharpoonup u\quad\text{in }H^1(0,T;L^2(\Omega;\mathbb R^d)),& & u_n'\rightharpoonup \dot u\quad\text{in }H^1(0,T;H^{-1}_{D_1}(\Omega;\mathbb R^d)),\\
&u_n\to u\quad\text{in }C^0([0,T];L^2(\Omega;\mathbb R^d)),& & u_n'\to \dot u\quad\text{in }C^0([0,T];H^{-1}_{D_1}(\Omega;\mathbb R^d)),\\
&\overline u_n,\underline u_n\rightharpoonup u\quad\text{in }L^2(0,T;H^1(\Omega;\mathbb R^d)),& & \overline u_n',\underline u_n'\rightharpoonup \dot u\quad\text{in }L^2(0,T;L^2(\Omega;\mathbb R^d)),\\
&\overline v_n,\underline v_n\rightarrow v\quad\text{in }L^2(0,T;L^2(\Omega)),& & \overline v_n,\underline v_n\rightharpoonup v\quad\text{in }L^2(0,T;H^1(\Omega)).
\end{align*}
Moreover, for every $t\in [0,T]$ as $n\to\infty$ we have
\begin{align*}
&\overline v_n(t)\rightarrow v(t)\quad\text{in }L^2(\Omega),\quad \overline v_n(t)\rightharpoonup v(t)\quad\text{in }H^1(\Omega).
\end{align*}
\end{lemma}

\begin{proof}
The existence of a limit point $u$ and the related convergences can be obtained by arguing as in Lemma~\ref{lem:ph_con}. Let us now consider the sequence $\{\overline v_n\}_{n\in\mathbb N}$. For every $n\in\mathbb N$ the functions $\overline v_n\colon [0,T]\to L^2(\Omega)$ are non increasing in $[0,T]$, that is
\begin{equation*}
\overline v_n(t)\le \overline v_n(s)\quad\text{in $\Omega$}\quad\text{for every $0\le s\le t\le T$},
\end{equation*}
and, in view of Lemma~\ref{lem:ph_bound2}, there exists $C>0$, independent of $n$, such that for every $t\in[0,T]$ and $n\in\mathbb N$
\begin{equation}\label{eq:ph_vnest}
\norm{\overline v_n(t)}_{H^1(\Omega)}\le C.
\end{equation}
Therefore, we can apply Lemma~\ref{lem:ph_Helly}. Up to extract a subsequence (not relabeled), we obtain the existence of a non increasing function $v\colon [0,T]\to L^2(\Omega)$ such that for every $t\in[0,T]$ as $n\to\infty$
\begin{align*}
\overline v_n(t)\rightharpoonup v(t)\quad\text{in }L^2(\Omega).
\end{align*}
Moreover, by~\eqref{eq:ph_vnest} for every $t\in[0,T]$ we derive that $v(t)\in H^1(\Omega)$ and as $n\to\infty$
\begin{align*}
&\overline v_n(t)\rightharpoonup v(t)\quad\text{in }H^1(\Omega),\quad\overline v_n(t)\rightarrow v(t)\quad\text{in }L^2(\Omega),
\end{align*}
thanks to Rellich's theorem. Notice that the function $v\colon [0,T]\to H^1(\Omega)$ is strongly measurable. Indeed, it is weak measurable, since it is non increasing, and with values in a separable Hilbert space. In particular, we have $v\in L^\infty(0,T;H^1(\Omega))$, since $\norm{v(t)}_{H^1(\Omega)}\le C$ for every $t\in[0,T]$. By the dominated convergence theorem, as $n\to\infty$ we conclude 
\begin{align*}
&\overline v_n\to v\quad\text{in }L^2(0,T;L^2(\Omega)),\quad \overline v_n\rightharpoonup v\quad\text{in }L^2(0,T;H^1(\Omega)).
\end{align*}
Finally, as $n\to\infty$ we have
\begin{align*}
&\underline v_n\to v\quad\text{in }L^2(0,T;L^2(\Omega)),\quad  \underline v_n\rightharpoonup v\quad\text{in }L^2(0,T;H^1(\Omega)),
\end{align*}
since $\underline v_n(t)=\overline v_n(t-\tau_n)$ for a.e. $t\in(\tau_n,T)$.
\end{proof}

\begin{remark}
As pointed out in Remark~\ref{rem:ph_pointcon}, for every $t\in[0,T]$ we have as $n\to\infty$
\begin{align*}
&\overline u_n(t)\rightharpoonup u(t)\quad\text{in }H^1(\Omega;\R^d),\quad \overline u_n'(t)\rightharpoonup \dot u(t)\quad\text{in }L^2(\Omega;\R^d),\\
&\underline u_n(t)\rightharpoonup u(t)\quad\text{in }H^1(\Omega;\R^d),\quad \underline u_n'(t)\rightharpoonup \dot u(t)\quad\text{in }L^2(\Omega;\R^d).
\end{align*}
\end{remark}

We are now in position to prove Theorem~\ref{thm:ph_mainres2}.

\begin{proof}[Proof of Theorem~\ref{thm:ph_mainres2}]
Thanks to the previous lemma there exists a pair $(u,v)$ satisfying~\eqref{eq:ph_greg12}--\eqref{eq:ph_greg42}, since $u_n(t)-w_n(t)\in H^1_{D_1}(\Omega;\R^d)$ and $\overline v_n(t)-w_2\in H^1_{D_2}(\Omega)$ for every $t\in[0,T]$ and $n\in\mathbb N$. Moreover, $(u,v)$ satisfies the irreversibility condition~\eqref{eq:ph_irr} and the initial conditions~\eqref{eq:ph_ic}, thanks to~\eqref{eq:ph_irr2} and $u^0=u_n(0)\rightharpoonup u(0)$ in $H^1(\Omega;\mathbb R^d)$, $u^1=u_n'(0)\rightharpoonup \dot u(0)$ in $L^2(\Omega;\mathbb R^d)$, and $v^0=\overline v_n(0)\rightharpoonup  v(0)$ in $H^1(\Omega)$ as $n\to\infty$.

For every $n\in\mathbb N$ and $j=1,\dots,n$ the pair $(u_n^j,v_n^j)$ solves the equation~\eqref{eq:ph_un2}. In particular, by integrating it over the time interval $[t_1,t_2]\subseteq [0,T]$, for every $\psi\in H^1_{D_1}(\Omega;\mathbb R^d)$ we deduce
\begin{equation*}
\int_{t_1}^{t_2}\spr{\dot u_n'(t)}{\psi}_{H^{-1}_{D_1}(\Omega)}\de t+\int_{t_1}^{t_2}(b(\underline v_n(t))\mathbb C E\overline u_n(t),E\psi)_{L^2(\Omega)}\de t=\int_{t_1}^{t_2}(\overline f_n(t),\psi)_{L^2(\Omega)}\de t+\int_{t_1}^{t_2}\spr{\overline g_n(t)}{\psi}_{H^{-1}_{D_1}(\Omega)}\de t.
\end{equation*}
Let us use the convergences of Lemma~\ref{lem:ph_con2} to pass to the limit as $n\to\infty$. We have 
\begin{align*}
&\lim_{n\to\infty}\int_{t_1}^{t_2}\spr{\dot u_n'(t)}{\psi}_{H^{-1}_{D_1}(\Omega)}\de t=\int_{t_1}^{t_2}\spr{\ddot u(t)}{\psi}_{H^{-1}_{D_1}(\Omega)}\de t,\\
&\lim_{n\to\infty}\int_{t_1}^{t_2}(\overline f_n(t),\psi)_{L^2(\Omega)}\de t=\int_{t_1}^{t_2}(f(t),\psi)_{L^2(\Omega)}\de t,\\
&\lim_{n\to\infty}\int_{t_1}^{t_2}\spr{\overline g_n(t)}{\psi}_{H^{-1}_{D_1}(\Omega)}\de t=\int_{t_1}^{t_2}\spr{ g(t)}{\psi}_{H^{-1}_{D_1}(\Omega)}\de t,
\end{align*}
since $\dot u_n'\rightharpoonup\ddot u$ and $\overline g_n\to g$ in $L^2(0,T;H^{-1}_{D_1}(\Omega;\mathbb R^d))$, and $f_n\to f$ in $L^2(0,T;L^2(\Omega;\mathbb R^d))$ as $n\to\infty$. Moreover, the dominated convergence theorem yields $b(\underline v_n)\mathbb C E\psi\to b(v)\mathbb C E\psi$ in $L^2(0,T;L^2(\Omega;\mathbb R^{d\times d}))$ as $n\to\infty$, being
$$\left|b(\underline v_n(t,x))\mathbb C(x)E\psi(x)\right|\le b(1)\norm{\mathbb C}_{L^\infty(\Omega)}|E\psi(x)|\quad\text{for every $t\in[0,T]$ and a.e. $x\in\Omega$},$$ 
and $\underline v_n\to v$ in $L^2(0,T;L^2(\Omega))$. Therefore, we derive
\begin{align*}
\lim_{n\to\infty}\int_{t_1}^{t_2}(b(\underline v_n(t))\mathbb CE\overline u_n(t),E\psi)_{L^2(\Omega)}\de t=\int_{t_1}^{t_2}(b(v(t))\mathbb C Eu(t),E\psi)_{L^2(\Omega)}\de t,
\end{align*}
because $E\overline u_n\rightharpoonup Eu$ in $L^2(0,T;L^2(\Omega;\R^{d\times d}))$ as $n\to\infty$.
These facts imply that the pair $(u,v)$ solves
\begin{equation*}
\int_{t_1}^{t_2}\spr{\ddot u(t)}{\psi}_{H^{-1}_{D_1}(\Omega)}\de t+\int_{t_1}^{t_2}(b(v(t))\mathbb C Eu(t),E\psi)_{L^2(\Omega)}\de t=\int_{t_1}^{t_2}(f(t),\psi)_{L^2(\Omega)}\de t+\int_{t_1}^{t_2}\spr{ g(t)}{\psi}_{H^{-1}_{D_1}(\Omega)}\de t
\end{equation*}
for every $\psi\in H^1_{D_1}(\Omega;\mathbb R^d)$ and $[t_1,t_2]\subseteq[0,T]$. By Lebesgue's differentiation theorem and a density argument we hence obtain~\eqref{eq:ph_weak_form} for a.e. $t\in(0,T)$. 

For $t=0$ the crack stability condition~\eqref{eq:ph_mincon2} is trivially true, since $(u,v)$ satisfies the initial conditions~\eqref{eq:ph_ic} and the compatibility condition~\eqref{eq:ph_id2}. We fix $t\in(0,T]$ and, by the variational inequality~\eqref{eq:ph_varin2}, we derive
\begin{equation}\label{eq:ph_d2}
\mathcal E(\overline u_n(t),v^*)-\mathcal E(\overline u_n(t),\overline v_n(t))+\partial\mathcal H(\overline v_n(t))[v^*-\overline v_n(t)]\ge 0
\end{equation}
among all $v^*-w_2\in H^1_{D_2}(\Omega)$ with $v^*\le \overline v_n(t-\tau_n)$.  Given $\chi\in H^1_{D_2}(\Omega)$, with $\chi\le 0$ in $\Omega$, the function $\chi+\overline v_n(t)$ is admissible for~\eqref{eq:ph_d2}. Hence, we have
\begin{equation*}
\mathcal E(\overline u_n(t),\chi+\overline v_n(t))-\mathcal E(\overline u_n(t),\overline v_n(t))+\partial\mathcal H(\overline v_n(t))[\chi]\ge 0.
\end{equation*}
Let us send $n\to\infty$. Since
$\overline v_n(t)\rightharpoonup v(t)$ in $H^1(\Omega)$, we deduce
\begin{align*}
\lim_{n\to\infty}\partial\mathcal H(\overline v_n(t))[\chi]=\partial\mathcal H(v(t))[\chi].
\end{align*}
Moreover, $E\overline u_n(t) \rightharpoonup Eu(t)$ in $L^2(\Omega;\R^{d\times d})$ and $\overline v_n(t)\to v(t)$ in $L^2(\Omega)$ as $n\to\infty$, which implies
\begin{equation*}
\begin{aligned}
\mathcal E(u(t),\chi+v(t))-\mathcal E(u(t),v(t)) \ge \limsup_{n\to\infty} \left[\mathcal E(\overline u_n(t),\chi+\overline v_n(t))-\mathcal E(\overline u_n(t),\overline v_n(t))\right]
\end{aligned}
\end{equation*}
by Ioffe--Olech's theorem, as in Lemma~\ref{lem:ph_lim2}. If we combine these two results, for every $t\in(0,T]$ we get
\begin{equation*}
\mathcal E(u(t),\chi+v(t))-\mathcal E(u(t),v(t))+\partial\mathcal H(v(t))[\chi]\ge 0
\end{equation*}
for every $\chi\in H^1_{D_2}(\Omega)$ with $\chi\le 0$ in $\Omega$. This implies~\eqref{eq:ph_mincon2}, since the map $v^*\mapsto\mathcal H(v^*)$ is convex.

It remains to prove the energy inequality~\eqref{eq:ph_enin2} for every $t\in[0,T]$. For $t=0$ we have actually the equality, thanks to the initial conditions~\eqref{eq:ph_ic}. We fix now $t\in(0,T]$, and we use the inequality~\eqref{eq:ph_denin2} to write
\begin{align*}
&\mathcal F(\overline u_n(t),\overline u_n'(t),\overline v_n(t))\\
&\le \mathcal F(u^0, u^1, v^0)+\int_0^{t_n}(\overline f_n(s),\overline u_n'(s)-\overline w_n'(s))_{L^2(\Omega)}\de s+\int_0^{t_n}(b(\underline v_n(s))\mathbb CE\overline u_n(s),E\overline w_n'(s))_{L^2(\Omega)}\de s\\
&\quad+\spr{\overline g_n(t)}{\overline u_n(t)-\overline w_n(t)}_{H^{-1}_{D_1}(\Omega)}-\spr{g(0)}{u^0-w_1(0)}_{H^{-1}_{D_1}(\Omega)}-\int_0^{t_n}\spr{\dot g_n(s)}{\underline u_n(s)-\underline w_n(s)}_{H^{-1}_{D_1}(\Omega)}\de s\\
&\quad+(\overline u_n'(t),\overline w_n'(t))_{L^2(\Omega)}-(u^1,w_1(0))_{L^2(\Omega)}-\int_0^{t_n}(\underline u_n'(s),\dot w_n'(s))_{L^2(\Omega)}\de s
\end{align*}
for every $n\in\mathbb N$, where $t_n$ is the same number defined in Lemma~\ref{lem:ph_enin}. By $\overline v_n(t)\rightharpoonup v(t)$ in $H^1(\Omega)$ as $n\to\infty$, we deduce
\begin{equation*}
\mathcal H(v(t)) \le \liminf_{n\to\infty}\mathcal H(\overline v_n(t)).
\end{equation*}
Similarly, thanks to Ioffe--Olech's theorem, we derive
\begin{equation*}
\mathcal E(u(t),v(t))\le \liminf_{n\to\infty}\mathcal E(\overline u_n(t),\overline v_n(t)),
\end{equation*}
since $\overline v_n(t)\to v(t)$ in $L^2(\Omega)$ and $E\overline u_n(t)\rightharpoonup Eu(t)$ in $L^2(\Omega;\R^{d\times d})$. Finally, we can argue as in Lemma~\ref{lem:ph_enin} to derive that the remaining terms converge to $\mathcal W_{tot}(u,v;0,t)$ as $n\to\infty$. By combining the previous results, we deduce~\eqref{eq:ph_enin2} for every $t\in(0,T]$.

Finally, when $w_2\ge 0$ on $\partial_{D_2}\Omega$, $v^0\ge 0$ in $\Omega$, and $b(s)=(\max\{s,0\})^2+\eta$ for $s\in\R$, we have $v_n^j\ge 0$ in $\Omega$ for every $j=1,\dots,n$. Therefore, we get $\overline v_n(t)\ge 0$ in $\Omega$ for every $t\in[0,T]$, which implies $v(t)\ge 0$ in $\Omega$.
\end{proof}

%------------------------------------------
% Acknowledgments
%------------------------------------------

\begin{acknowledgements}
% The author acknowledge SISSA, to which he was affiliated during the preparation
% of this work.
The author wishes to thank Prof. Gianni Dal Maso for having proposed the problem and for many helpful discussions on the topic. The author is member of the {\em Gruppo Nazionale per l'Analisi Ma\-te\-ma\-ti\-ca, la Probabilit\`a e le loro Applicazioni}~(GNAMPA) of the {\em Istituto Nazionale di Alta Matematica}~(INdAM). 
%The manuscript was realized within the auspices of the INdAM -- GNAMPA Project 2017 {\em Equazioni Differenziali non lineari} (Prot\_2017\_0000265).
\end{acknowledgements}

%------------------------------------------
% Bibliography
%------------------------------------------


\begin{thebibliography}{99}
	
\bibitem{A}{\sc R.A.~Adams}: Sobolev spaces. Pure and Applied Mathematics, Vol. 65, Academic Press, New York, 1975.

\bibitem{ABN}{\sc S.~Almi, S.~Belz, and M.~Negri}: Convergence of discrete and continuous unilateral flows for Ambrosio--Tortorelli energies and application to mechanics. {\it ESAIM Math. Model. Numer. Anal.} {\bf 53} (2019), 659--699.

\bibitem{AT}{\sc L.~Ambrosio and V.M.~Tortorelli}: Approximation of functionals depending on jumps by elliptic functionals via $\Gamma$--convergence. {\it Comm. Pure Appl. Math.} {\bf 43} (1990), 999--1036.

\bibitem{BFM}{\sc B.~Bourdin, G.A.~Francfort, and J.J.~Marigo}: The variational approach to fracture. Reprinted from {\it J. Elasticity} {\bf 91} (2008), Springer, New York, 2008.

\bibitem{BLR}{\sc B.~Bourdin, C.J.~Larsen, and C.L.~Richardson}: A time--discrete model for dynamic fracture based on crack regularization. {\it Int. J. Fracture} {\bf 168} (2011), 133--143.

\bibitem{D}{\sc B.~Dacorogna}: Direct methods in the calculus of variations. Applied Mathematical Sciences, Vol. 78, Springer-Verlag, Berlin, 1989.

\bibitem{DM-Lar}{\sc G.~Dal~Maso and C.J.~Larsen}: Existence for wave equations on domains with arbitrary growing cracks. {\it Atti Accad. Naz. Lincei Rend. Lincei Mat. Appl.} {\bf 22} (2011), 387--408.

\bibitem{DMS}{\sc G.~Dal~Maso and R.~Scala}: Quasistatic evolution in perfect plasticity as limit of dynamic processes. {\it J. Dynam. Differential Equations} {\bf 26} (2014), 915--954.

\bibitem{DL}{\sc R.~Dautray and J.L.~Lions}: Analyse math\'ematique et calcul num\'erique pour les sciences et les techniques. Vol. 8. \'Evolution: semi-groupe, variationnel. Masson, Paris, 1988.

\bibitem{FM}{\sc G.A.~Francfort and J.J.~Marigo}: Revisiting brittle fracture as an energy minimization problem. {\it J. Mech. Phys. Solids} {\bf 46} (1998), 1319--1342.

\bibitem{Gi}{\sc A.~Giacomini}: Ambrosio--Tortorelli approximation of quasi--static evolution of brittle fractures. {\it Calc. Var. Partial Differential Equations} {\bf 22} (2005), 129--172.

\bibitem{Grif}{\sc A.A.~Griffith}: The phenomena of rupture and flow in solids. {\it Philos. Trans. Roy. Soc. London} {\bf 221-A} (1920), 163--198.

\bibitem{Lad}{\sc O.A.~Ladyzenskaya}: On integral estimates, convergence, approximate methods, and solution in functionals for elliptic operators. {\it Vestnik Leningrad. Univ.} {\bf 13} (1958), 60--69.

\bibitem{Lar}{\sc C.J.~Larsen}: Models for dynamic fracture based on Griffith's criterion. In: Hackl K. (eds.) ``IUTAM Symposium on Variational Concepts with Applications to the Mechanics of Materials", IUTAM Bookseries, Vol 21, Springer, Dordrecht, 2010, 131--140.

\bibitem{LOS}{\sc C.J.~Larsen, C.~Ortner, and E.~S\"uli}: Existence of solutions to a regularized model of dynamic fracture. {\it Math. Models Methods Appl. Sci.} {\bf 20} (2010), 1021--1048. 

\bibitem{LN}{\sc G.~Lazzaroni and L.~Nardini}, Analysis of a dynamic peeling test with speed-dependent toughness. {\it SIAM J. Appl. Math.} {\bf 78} (2018), 1206--1227.

\bibitem{Laz-Toa2}{\sc G.~Lazzaroni and R.~Toader}: A model for crack propagation based on viscous approximation. {\it Math. Models Methods Appl. Sci.} {\bf 21} (2011), 2019--2047.

\bibitem{Mott}{\sc N.F.~Mott}: Brittle fracture in mild steel plates. {\it Engineering} {\bf 165} (1948), 16--18.

\bibitem{Mum-Shah}{\sc D. Mumford and J. Shah}: Optimal approximation by piecewise smooth functions and associated variational problems. {\it Commun. Pure Appl. Math.} {\bf 42} (1989), 577--684.

\bibitem{Negri}{\sc M.~Negri}: A unilateral $L^2$--gradient flow and its quasi--static limit in phase--field fracture by an alternate minimizing movement. {\it Adv. Calc. Var.} {\bf 12} (2019), 1--29.

\bibitem{OSY}{\sc O.A.~Oleinik, A.S.~Shamaev, and G.A.~Yosifian}: Mathematical problems in elasticity and homogenization. Studies in Mathematics and its Applications, 26. North-Holland Publishing Co., Amsterdam, 1992.

\bibitem{R}{\sc S.~Racca}: A viscosity--driven crack evolution. {\it Adv. Calc. Var.} {\bf 5} (2012), 433--483.

\bibitem{S}{\sc J.~Simon}: Compact sets in the space $L^p(0,T;B)$. {\it Ann. Mat. Pura Appl.} {\bf 146} (1987), 65--96.

\bibitem{T1}{\sc E.~Tasso}: Weak formulation of elastodynamics in domains with growing cracks. Online on {\it Ann. Math. Pura Appl.} (2019), DOI: 10.1007/s10231-019-00932-y.
\end{thebibliography}
\end{document}